\documentclass[12pt]{article}

\pdfoutput=1

\usepackage{amsmath}
\usepackage{amsthm}
\usepackage{amssymb}
\usepackage{enumerate}
\usepackage{cases}
\usepackage{multirow}
\usepackage{subfigure}
\usepackage{graphicx}
\usepackage{appendix}
\usepackage{pgf,tikz}
\usetikzlibrary{arrows}
\graphicspath{{./}}
\DeclareMathOperator{\sign}{sgn}

\DeclareMathOperator{\essinf}{ess\,inf}
\begin{document}
\newcommand{\done}[2]{\dfrac{d {#1}}{d {#2}}}
\newcommand{\donet}[2]{\frac{d {#1}}{d {#2}}}
\newcommand{\pdone}[2]{\dfrac{\partial {#1}}{\partial {#2}}}
\newcommand{\pdonet}[2]{\frac{\partial {#1}}{\partial {#2}}}
\newcommand{\pdonetext}[2]{\partial {#1}/\partial {#2}}
\newcommand{\pdtwo}[2]{\dfrac{\partial^2 {#1}}{\partial {#2}^2}}
\newcommand{\pdtwot}[2]{\frac{\partial^2 {#1}}{\partial {#2}^2}}
\newcommand{\pdtwomix}[3]{\dfrac{\partial^2 {#1}}{\partial {#2}\partial {#3}}}
\newcommand{\pdtwomixt}[3]{\frac{\partial^2 {#1}}{\partial {#2}\partial {#3}}}
\newcommand{\bs}[1]{\mathbf{#1}}
\newcommand{\bx}{\mathbf{x}}
\newcommand{\by}{\mathbf{y}}
\newcommand{\bz}{\mathbf{z}}
\newcommand{\bd}{\mathbf{d}} 
\newcommand{\bn}{\mathbf{n}} 
\newcommand{\bP}{\mathbf{P}} 
\newcommand{\bR}{\mathbf{R}} 
\newcommand{\bS}{\mathbf{S}} 
\newcommand{\bT}{\mathbf{T}} 
\newcommand{\bp}{\mathbf{p}} 
\newcommand{\ol}[1]{\overline{#1}}
\newcommand{\rf}[1]{\eqref{#1}}
\newcommand{\xt}{\mathbf{x},t}
\newcommand{\hs}[1]{\hspace{#1mm}}
\newcommand{\vs}[1]{\vspace{#1mm}}
\newcommand{\eps}{\varepsilon}
\newcommand{\ord}[1]{\mathcal{O}\left(#1\right)} 
\newcommand{\oord}[1]{o\left(#1\right)}
\newcommand{\Ord}[1]{\Theta\left(#1\right)}
\newcommand{\PhiF}{\Phi_{\rm freq}}
\newcommand{\real}[1]{{\rm Re}\left[#1\right]} 
\newcommand{\im}[1]{{\rm Im}\left[#1\right]}
\newcommand{\hsnorm}[1]{||#1||_{H^{s}(\bs{R})}}
\newcommand{\hnorm}[1]{||#1||_{\tilde{H}^{-1/2}((0,1))}}
\newcommand{\norm}[2]{\left\|#1\right\|_{#2}}
\newcommand{\normt}[2]{\|#1\|_{#2}}
\newcommand{\on}[1]{\Vert{#1} \Vert_{1}}
\newcommand{\tn}[1]{\Vert{#1} \Vert_{2}}
\newcommand{\ts}{\tilde{s}}
\newcommand{\darg}[1]{\left|{\rm arg}\left[ #1 \right]\right|}
\newcommand{\bnabla}{\boldsymbol{\nabla}}
\newcommand{\dive}{\boldsymbol{\nabla}\cdot}
\newcommand{\curl}{\boldsymbol{\nabla}\times}
\newcommand{\Phixy}{\Phi(\bx,\by)}
\newcommand{\PhiOxy}{\Phi_0(\bx,\by)}
\newcommand{\dxPhixy}{\pdone{\Phi}{n(\bx)}(\bx,\by)}
\newcommand{\dyPhixy}{\pdone{\Phi}{n(\by)}(\bx,\by)}
\newcommand{\dxPhiOxy}{\pdone{\Phi_0}{n(\bx)}(\bx,\by)}
\newcommand{\dyPhiOxy}{\pdone{\Phi_0}{n(\by)}(\bx,\by)}

\newcommand{\rd}{\mathrm{d}}
\newcommand{\R}{\mathbb{R}}
\newcommand{\N}{\mathbb{N}}
\newcommand{\Z}{\mathbb{Z}}
\newcommand{\C}{\mathbb{C}}
\newcommand{\K}{{\mathbb{K}}}
\newcommand{\ri}{{\mathrm{i}}}
\newcommand{\re}{{\mathrm{e}}} 

\newcommand{\cA}{\mathcal{A}}
\newcommand{\cC}{\mathcal{C}}
\newcommand{\cS}{\mathcal{S}}
\newcommand{\cD}{\mathcal{D}}
\newcommand{\cone}{{c_{j}^\pm}}
\newcommand{\ctwo}{{c_{2,j}^\pm}}
\newcommand{\cthree}{{c_{3,j}^\pm}}

\newtheorem{thm}{Theorem}[section]
\newtheorem{lem}[thm]{Lemma}
\newtheorem{defn}[thm]{Definition}
\newtheorem{prop}[thm]{Proposition}
\newtheorem{cor}[thm]{Corollary}
\newtheorem{rem}[thm]{Remark}
\newtheorem{conj}[thm]{Conjecture}
\newtheorem{ass}[thm]{Assumption}
\newcommand{\JC}{J_{\rm C}}
\newcommand{\JI}{J_{\rm NCI}}
\newcommand{\JII}{J_{\rm NCII}}
\newcommand{\bQ}{\mathbf{Q}}
\newcommand{\uM}{M(u)}
\newcommand{\tW}{W^s}
\newcommand{\sW}{W^i}
\definecolor{cqcqcq}{rgb}{0.9,0.9,0.9}
\title{A high frequency boundary element method for scattering by a class of nonconvex obstacles}
%
%
%
\author{S.\ N.\ Chandler-Wilde\footnotemark[1], D.\ P.\ Hewett\footnotemark[1] \footnotemark[2], S.\ Langdon\footnotemark[1], A.\ Twigger\footnotemark[1]}

\renewcommand{\thefootnote}{\fnsymbol{footnote}}
\footnotetext[1]{Department of Mathematics and Statistics, University of Reading, Whiteknights PO Box 220, Reading RG6 6AX, UK. This work was supported by EPSRC grant EP/F067798/1.}

\footnotetext[2]{Current address: Mathematical Institute, University of Oxford, Radcliffe Observatory Quarter, Woodstock Road, Oxford, OX2 6GG, UK. Email: \texttt{hewett@maths.ox.ac.uk}}

\maketitle
\renewcommand{\thefootnote}{\arabic{footnote}}
\begin{abstract}
In this paper we propose and analyse a hybrid numerical-asymptotic boundary element method for the solution of problems of high frequency acoustic scattering by a class of sound-soft nonconvex polygons.  The approximation space is enriched with carefully chosen oscillatory basis functions; these are selected via a study of the high frequency asymptotic behaviour of the solution. We demonstrate via a rigorous error analysis, supported by numerical examples, that to achieve any desired accuracy it is sufficient for the number of degrees of freedom to grow only in proportion to the logarithm of the frequency as the frequency increases, in contrast to the at least linear growth required by conventional methods.
This appears to be the first such numerical analysis result for any problem of scattering by a nonconvex obstacle.
Our analysis is based on new frequency-explicit bounds on the normal derivative of the solution on the boundary and on its analytic continuation into the complex plane.
%
\end{abstract}
\section{Introduction}
\label{Introduction}

There has been considerable interest in recent years in the development of numerical methods for time harmonic acoustic and electromagnetic scattering problems that can efficiently resolve the scattered field at high frequencies.  Standard finite or boundary element methods, with piecewise polynomial approximation spaces, suffer from the restriction that a fixed number of degrees of freedom is required per wavelength in order to represent the oscillatory solution, leading to excessive computational cost when the scatterer is large compared to the wavelength.

A general methodology that has shown a great deal of promise is the so-called ``hybrid numerical-asymptotic'' (HNA)
approach, where the numerical approximation space is enriched with oscillatory functions, chosen using partial knowledge of the high frequency (short wavelength) asymptotic behaviour. %
We refer to~\cite{ChGrLaSp:11} (and the very many references therein) for a review of this fast-evolving field and its historical development.
The HNA
approach is particularly attractive when employed within a boundary element method (BEM) framework, since knowledge of the high frequency asymptotics is required only on the boundary of the scatterer.
In this setting one first
reformulates the boundary value problem (defined precisely in~\S\ref{sec:prob}) as a boundary integral equation, with frequency dependent solution $V$, and then seeks to approximate $V$ using an ansatz of the form
\begin{equation} \label{eqn:ansatz}
  V(\bx,k)  \approx  V_{0}(\bx,k)  + \sum_{m=1}^M  V_m(\bx,k) \, \exp(\ri k \psi_m(\bx)), \quad \bx\in\Gamma,
\end{equation}
where $k$ (the wavenumber) is proportional to the frequency of the incident wave, $\Gamma$ is the boundary of the scatterer,
$V_0$ is a known (generally oscillatory)
function (derived from the high frequency asymptotics), the phases $\psi_m$ are chosen {\em a-priori} and
the amplitudes $V_{m}$, $m=1,\ldots,M$, are approximated numerically.  The key idea behind the HNA approach is that if $V_0$ and $\psi_m$, $m=1,\ldots,M$, in \rf{eqn:ansatz} are chosen wisely, then
$V_m(\cdot,k)$, $m=1,\ldots,M$, will be much less oscillatory than $V(\cdot,k)$ and so can be better approximated by piecewise polynomials than $V$ itself.

Indeed, whereas conventional BEMs for two-dimensional (2D) problems require the number of degrees of freedom to grow at least linearly with respect to frequency in order to maintain a prescribed level of accuracy as the frequency increases, HNA BEMs have been shown, for a range of problems, to require a significantly milder (often only logarithmic) growth in computational cost \cite{ChGrLaSp:11}.
However, to date, the vast majority of HNA algorithms have been restricted to problems of scattering by single convex obstacles.

The aim of this paper is to show, via rigorous numerical analysis supported by numerical results, that HNA methods can be as effective for nonconvex scatterers as they are for convex scatterers.
We propose and analyse a HNA BEM for a class of nonconvex polygons, using an ansatz of the form~(\ref{eqn:ansatz}), with $V_m$, $m=1,\ldots,M$, approximated using an $hp$ approximation space. The novelty of our analysis compared to most numerical analysis for scattering problems is that it is uniform with respect to both the discretisation and the frequency. On the one hand, our rigorous error estimates prove that, for fixed frequency, the method converges exponentially as the number of degrees of freedom is increased. On the other hand, they also show that to achieve any prescribed level of accuracy it is sufficient for the number of degrees of freedom to grow only logarithmically with respect to frequency, as frequency increases. This is the same growth as that
required by the scheme for convex polygons in~\cite{HeLaMe:11}. But this is the first time, to our knowledge, that an algorithm has been proposed, for any configuration where multiple scattering is present, that provably maintains accuracy at high frequency with degrees of freedom growing only logarithmically with frequency.

The main difficulty in developing and analysing HNA methods for nonconvex scatterers is that the high frequency asymptotic behaviour %
is significantly more complicated than in the convex case, because of the possibility of highly non-trivial multiple scattering and shadowing effects. %
Indeed, constructing a high-order uniform asymptotic solution for any given nonconvex obstacle, using, for example, the Geometrical Theory of Diffraction \cite{Ke:62,UTD,BoKi:94}, is a formidable task in general
(cf.\ e.g. \cite[\S7-\S8]{BoKi:94}), and proving rigorously the validity of high frequency asymptotic approximations is extremely challenging. Indeed, even for the simpler case of scattering by a convex polygon, while there exist methodologies to construct asymptotic approximations (e.g. \cite{BoKi:94,Ra:12}), the authors know of no rigorous theory which establishes the accuracy of such asymptotic approximations.

The HNA methodology proposed in this paper does not require knowledge of the full asymptotic solution. Rather, in order to design a HNA approximation space one needs only a representation of the form~(\ref{eqn:ansatz}), with an explicit (and relatively simple) term $V_0$ and explicit phases $\psi_m$, that captures the high frequency oscillations present in the solution. But to design HNA algorithms optimally, and prove their effectiveness by rigorous numerical analysis, one needs additionally to understand the regularity of the amplitudes $V_m$, $m=1,\ldots,M$, moreover obtaining bounds on these amplitudes that are explicit in their dependence on the wavenumber. This requires rigorous high frequency asymptotics which aims at coarser information than the full asymptotic solution. Results of this type are proved for the case of convex polygons in \cite{Convex,ChLaMo:11,HeLaMe:11}; we emphasise that even for the considerably simpler case of scattering by convex polygons, the results of these papers are the only rigorous high frequency asymptotics known to the authors. Because of multiple scattering and shadowing effects, developing any sort of rigorous high frequency asymptotics for scattering by {\em nonconvex} polygons is a formidable task. The results of this kind needed to analyse our HNA algorithm form the largest section of the paper and are proved in~\S\ref{sec:regularity} and~\S\ref{NonConvProof} below.%

At present our full analysis applies only to a particular class of nonconvex polygons, defined explicitly in \S\ref{sec:regularity}. Essentially we assume: (i) an ``orthogonality'' condition, that each exterior angle smaller than $\pi$ is a right-angle; (ii) a ``visibility'' condition, ensuring that each point on the boundary is only visible to at most three corners of the polygon (notably, this assumption avoids ``trapping'' domains as discussed, e.g., in \cite{ChGrLaLi:09,BeSp:11}).
As will be discussed in detail in \S\ref{sec:regularity}, these assumptions limit the possible complexity of the high frequency asymptotic behaviour, and hence the complexity of the ansatz \rf{eqn:ansatz}. The reason for adopting them is that they make possible a full frequency-explicit best approximation error analysis of our HNA approximation space (even so, as we shall see, this requires significant new ideas compared to the convex case \cite{HeLaMe:11}).  %
We believe though that the underlying principles behind our method apply much more generally, and in \S\ref{sec:generalisations} we give detailed suggestions as to how these assumptions could be relaxed to allow the development of both algorithms and analysis for more general nonconvex polygons.
An outline of the paper is as follows.  We begin in \S\ref{sec:prob} by stating the scattering problem and its boundary integral equation reformulation. In~\S\ref{sec:regularity} we clarify the class of nonconvex polygons for which our analysis holds, and state the exact form of the ansatz~(\ref{eqn:ansatz}) that we use.  We then provide regularity estimates
for those parts of the solution ($V_m$, $m=1,\ldots,M$) that we will approximate numerically. These estimates, which take the form of $k$-explicit bounds on the analytic continuation of $V_m$, $m=1,\ldots,M$, into the complex plane, constitute one of the main results of this paper, since they prove that $V_m$, $m=1,\ldots,M$, are not oscillatory, which is the key to achieving our goal of approximating the solution in an (almost) frequency independent way. %
The proof of these estimates occupies~\S\ref{NonConvProof}.
We define our $hp$-approximation space in~\S\ref{sec:ApproxSpace}, and prove best approximation estimates based on the results obtained in~\S\ref{sec:regularity}-\ref{NonConvProof}. In~\S\ref{sec:gal} we describe our Galerkin method, combining the results of the earlier sections to derive rigorous $k$-explicit error estimates for our approximations to the boundary solution, the total field in the exterior domain and the far field pattern.  In~\S\ref{sec:num} we present numerical examples, demonstrating the efficiency and accuracy of our scheme, and in \S\ref{sec:generalisations} we discuss extensions to more general geometries.

We end this section with some comments on the existing HNA literature.
Of the few HNA methods previously proposed for nonconvex scatterers we note the algorithm for single smooth nonconvex scatterers outlined in~\cite{Br:03,BrRe:07}. The numerical results presented in~\cite{Br:03,BrRe:07} suggest good performance at high frequencies for certain scattering configurations; however, these results are not supported by a rigorous numerical analysis, and it is not clear how the number of degrees of freedom required to achieve a prescribed accuracy depends on either the frequency or the scatterer geometry. %
We also mention the preliminary work in \cite{NonConvexConference2},
where an outline of some key steps of the algorithm described in this paper is presented without analysis.
We remark also on the related
case of multiple convex scatterers, which shares many of the difficulties associated with single nonconvex scatterers (multiple scattering, shadowing).
The case of multiple smooth convex scatterers has been considered in \cite{BrGeRe:05,Ec:05,EcRe:09,AnBoEcRe:10}.
The key theme of that body of work is a decomposition of the multiple scattering problem into a series of scattering problems for single convex obstacles, with in each case the incident field consisting of the original incident field or previously scattered waves.  Although this approach cannot be applied directly to the single nonconvex scatterers considered in this paper, %
it may, as we will discuss in \S\ref{sec:generalisations}, provide some insight into how to extend the ideas presented here to more general nonconvex scatterers.
\section{Problem statement and integral equation formulation}
\label{sec:prob}
We consider the 2D problem of scattering of a time harmonic incident plane wave
\begin{align}
u^i(\bx)&:=\re^{\ri k \bx\cdot \bd},
\qquad \bx=(x_1,x_2)\in\mathbb{R}^2,
\label{eqn:1}
\end{align}
with wavenumber $k>0$ (proportional to frequency) and unit direction vector $\bd$, by a sound soft polygon. Let $\Omega$ denote the interior of the polygon, and $D:=\mathbb{R}^2\backslash\overline{\Omega}$ the unbounded exterior domain.
The boundary value problem (BVP) we study is: given the incident field $u^i$, determine the total field $u\in C^2\left(D\right)\cap C\left(\overline{D}\right)$ such that
\begin{eqnarray}
 & \Delta u+k^2u  =  0, \quad \mbox{in }D, & \label{eqn:HE} \\
 & u  = 0, \quad \mbox{on }\Gamma:=\partial\Omega,& \label{eqn:bc1}
\end{eqnarray}
and $u^s:=u-u^i$ satisfies the Sommerfeld radiation condition (see, e.g., \cite[(2.9)]{ChGrLaSp:11}). The unique solvability of this BVP is well known (see, e.g., \cite[Theorem 2.12]{ChGrLaSp:11}). Standard arguments connecting formulations in classical function spaces to those in a Sobolev space setting (see, e.g., \cite[Theorem 3.7]{CoKr:92} and \cite[p.~107]{ChGrLaSp:11}) imply that if $u$ satisfies the above BVP then also $u\in H^1_{\mathrm{loc}}(D)$. From standard elliptic regularity results, it follows moreover that $u$ is $C^\infty$ up to the boundary of $\partial D$, excluding the corners of the polygon \cite[Lemma~2.35]{ChGrLaSp:11}.

The starting point of the boundary integral equation (BIE) formulation is that,
if $u$ satisfies the BVP then a form of Green's representation theorem holds, namely
\begin{align}
  \label{RepThm}
  u(\bx) = u^i(\bx) - \int_\Gamma \Phi_k(\bx,\by)\pdone{u}{\bn}(\by)\,\rd s(\by), \qquad \bx\in D
\end{align}
(see  \cite{Convex} and \cite[(2.107)]{ChGrLaSp:11}),
where $\Phi_k(\bx,\by):=(\ri/4)H^{(1)}_0\left(k\left|\bx-\by\right|\right)$ is the fundamental solution for~(\ref{eqn:HE}), $H_\nu^{(1)}$ the Hankel function of the first kind of order $\nu$, and $\pdonetext{u}{\bn}$ is the normal derivative, with $\bn$ the unit normal directed into $D$. We note that, as discussed in \cite{Convex} and \cite[Theorem 2.12]{ChGrLaSp:11}, it holds that
$\pdonetext{u}{\bn}\in L^2(\Gamma)$. It is well known (see, e.g., \cite[\S2]{ChGrLaSp:11}) that, starting from the representation formula \eqref{RepThm}, we can derive various BIEs for $\pdonetext{u}{\bn}\in L^2\left(\Gamma\right)$, each taking the form
\begin{align}
  \cA\frac{\partial u}{\partial \bn}&= f, \label{eqn:op061210}
\end{align}
where $f\in L^2\left(\Gamma\right)$ and $\cA:L^2\left(\Gamma\right)\rightarrow L^2\left(\Gamma\right)$ is a bounded linear operator.

In the standard combined potential formulation (see \cite[(2.114) and (2.69)]{ChGrLaSp:11}),
\begin{align}
\label{eqn:BIE_standard}
  \cA=\cA_{k,\eta} : = \frac{1}{2}\mathcal{I}+\mathcal{D}_k^\prime-\ri\eta \mathcal{S}_k,
\end{align}
and $f=\partial u^i/\partial \bn-\ri\eta u^i$, where $\eta\in\R$ is a coupling parameter, $\mathcal{I}$ is the identity operator, and the single-layer potential operator $\mathcal{S}_k$ and the adjoint double-layer potential operator $\mathcal{D}_k'$ are defined, for $\bx\in\Gamma$ and $\psi\in L^2(\Gamma)$, by
\begin{align*}
  \mathcal{S}_k\psi(\bx):=\int_\Gamma\Phi_k(\bx,\by)\psi(\by)\,\rd s(\by), %
  \qquad
  \mathcal{D}_k^\prime\psi(\bx):=\int_\Gamma\frac{\partial \Phi_k(\bx,\by)}{\partial \bn(\bx)}\psi(\by)\,\rd s(\by).%
\end{align*}
From results in \cite{CoKr:83} for $C^2$ domains, and~\cite{Convex} and \cite[Theorem 2.27]{ChGrLaSp:11} for general Lipschitz domains, $\cA_{k,\eta}$ is invertible for $k>0$, and hence~(\ref{eqn:op061210}) is uniquely solvable,  provided $\eta\in\mathbb{R}\backslash\left\{0\right\}$.
Recent results (\cite[(6.10)]{ChGrLaLi:09}, \cite[Theorem~2.11]{BeChGrLaLi:11}), building on earlier work \cite{Kr:85}, suggest $\eta=k$ is a good choice for large $k$, in that it approximately minimises the condition number of $\cA_{k,\eta}$ and its boundary element discretization.

In an important recent theoretical development \cite{SpChGrSm:11} a new formulation has been derived for the case when $\Omega$ is star-like. This takes the form~(\ref{eqn:op061210}) with %
\begin{align}
  \label{eqn:star_comb}
  \cA = \cA_k := (\bx\cdot \bn)\left(\frac{1}{2}\mathcal{I}+\mathcal{D}_k^\prime\right)+\bx\cdot\nabla_\Gamma \mathcal{S}_k+\left(\frac{1}{2}-\ri k|\bx|\right)\mathcal{S}_k,
\end{align}
the so-called ``star-combined'' operator (in which $\nabla_\Gamma$ denotes surface gradient), and $f(\bx)=\bx\cdot\nabla u^i(\bx)+(1/2-\ri k|\bx|) u^i(\bx)$.  From~\cite{SpChGrSm:11}, for $\Omega$ Lipschitz and star-like with respect to the origin, $\cA_k$ is invertible for all $k>0$. The point of this new formulation, as shown in \cite{SpChGrSm:11} and discussed below, is that $\cA_k$ is coercive on $L^2(\Gamma)$, moreover with a coercivity constant which is explicitly known and wavenumber independent.
For both formulations the following lemma holds provided $\Omega$ is Lipschitz and provided $\left|\eta\right|\leq Ck$ in the standard formulation (we shall assume henceforth that this condition always holds). Here and for the remainder of this paper $C>0$ denotes a constant whose value may change from one occurence to the next, but which is always independent of $k$, although it may (possibly) be dependent on~$\Omega$.

\begin{lem}[Continuity]
\label{lem:1}
\cite[Theorem 3.6]{ChGrLaLi:09}, \cite[Theorem 4.2]{SpChGrSm:11}
Assume that $\Omega$ is a bounded Lipschitz domain and $k_0 > 0$. In the case $\cA = \cA_{k,\eta}$
assume additionally that $\left|\eta\right| \leq C k$. Then for both $\cA = \cA_k$ and $\cA = \cA_{k,\eta}$
there exists a constant $C_0>0$, independent of $k$, such that
\[ %
  \left\|\cA\right\|_{L^2(\Gamma)}\leq C_0k^{1/2}, \qquad   k\geq k_0. %
\]
\end{lem}

Lemma \ref{lem:1} suggests at worst mild growth in $\|\cA\|_{L^2(\Gamma)}$ for both formulations as $k$ increases. For the case $\cA=\cA_{k,\eta}$, with $\eta$ proportional to $k$, it is shown in \cite{ChGrLaLi:09,BeChGrLaLi:11} that $\|\cA\|_{L^2(\Gamma)}$ does grow proportionally to $k^{1/2}$ for a polygonal scatterer, i.e.\ for this case at least it is known that the bound is sharp.

The regularity results we derive in \S\ref{sec:regularity}-\S\ref{NonConvProof}, and the resulting best approximation error estimates in \S\ref{sec:ApproxSpace}, will make use of the following assumption on the boundary solution, which, as will be discussed shortly, is known to hold in certain cases.
\begin{ass}[{proved in \cite[Lemma 4.2]{HeLaMe:11} in the star-like Lipschitz case}]
\label{ass:1}
There exist constants $C_1>0$ and $k_1>0$, independent of $k$, such that
\begin{equation*}
  \left\| \pdone{u}{\bn}\right\|_{L^2(\Gamma)}\leq C_1 k, \qquad k\geq k_1.
  \label{eqn:Ainvbound}
\end{equation*}
\end{ass}

The numerical analysis of our Galerkin method will be based on the following assumption on the boundary integral operator, which, as alluded to above, is also known to hold in certain cases.
\begin{ass}[Coercivity]
\label{ass:2}
There exist constants $C_2>0$ and $k_2>0$, independent of $k$, such that (where $\left\langle \cdot,\cdot\right\rangle_{L^2\left(\Gamma\right)}$ denotes the inner product in $L^2(\Gamma)$)
\[
  \big|\left\langle \cA\psi,\psi\right\rangle_{L^2\left(\Gamma\right)}\big|\geq C_2 \left\|\psi\right\|_{L^2(\Gamma)}^2, \quad \psi\in L^2\left(\Gamma\right),\,k\geq k_2.
\]
\end{ass}
If Assumption \ref{ass:2} holds, then by Lemma \ref{lem:1} and the Lax-Milgram lemma it follows that $\cA$ is invertible; moreover that $\cA^{-1}$ is uniformly bounded as $k\to\infty$, with
\begin{align}
\label{eqn:InverseBound}
  \big\| \cA^{-1}\big\|_{L^2(\Gamma)}\leq 1/C_2, \qquad k\geq k_2.
\end{align}
In particular, since for either formulation there exists a $k$-independent constant $C>0$ such that $\|f\|_{L^2(\Gamma)}\leq Ck$, Assumption \ref{ass:1} then holds with $k_1=k_2$ and $C_1 = C/C_2$.
Moreover, Assumption \ref{ass:2} guarantees that the linear system arising from any Galerkin approximation method for \eqref{eqn:op061210} is invertible, and, via C\'ea's lemma, implies explicit error estimates for the Galerkin solution, as discussed in \S\ref{sec:gal}.

The main achievement of \cite{SpChGrSm:11} is to show, via Morawetz-Ludwig identities, that, for the star-combined formulation $\cA=\cA_{k}$, Assumption \ref{ass:2} (and hence \rf{eqn:InverseBound}) holds for any star-like Lipschitz $\Omega$ (including those star-like members of our class $\cC$ of polygons defined below), and for all $k_2>0$, moreover with the explicit constant
\begin{equation*} %
C_2=\frac{1}{2} \essinf_{\bx\in \Gamma}(\bx\cdot \bn(\bx)).
 \end{equation*}
By contrast, for the standard formulation $\cA=\cA_{k,\eta}$, while \rf{eqn:InverseBound} is known to hold for all star-like Lipschitz $\Omega$ and for all $k_2>0$ (provided $\eta$ is proportional to $k$) \cite{ChMo:08}, Assumption \ref{ass:2} has only been proven to hold (for all $k_2>0$) when the scatterer is circular \cite{DoGrSm:07,SpChGrSm:11} and, for $k_2$ sufficiently large, when the scatterer is a strictly convex $C^3$ domain with strictly positive curvature (\cite{SpKaSm:13} and \cite[Theorem~5.25]{ChGrLaSp:11}). However, recent 2D numerical evidence, based on clever  numerical computations of coercivity constants, suggests that Assumption \ref{ass:2} holds much more generally, in particular for all star-like obstacles, and also for ``non-trapping'' non-star-like polygons (hence for all members of the class of nonconvex polygons (defined in \S\ref{sec:regularity}) we study in this paper)
\cite[Conjecture 6.2]{BeSp:11}.
\section{High frequency asymptotics and regularity of solutions}
\label{sec:regularity}
Our goal is to derive a numerical method for the solution of the BIE~(\ref{eqn:op061210}) (and hence of the scattering problem \rf{eqn:HE}-\rf{eqn:bc1}), whose performance does not deteriorate significantly as the wavenumber $k$ increases, equivalently as the wavelength $\lambda:=2\pi/k$ decreases.  Specifically, we wish to avoid the requirement of conventional schemes for a fixed number of degrees of freedom per wavelength.  To achieve this goal, our numerical method for solving~(\ref{eqn:op061210}) uses a HNA approximation space (defined explicitly in \S\ref{sec:ApproxSpace}) %
adapted to the high frequency asymptotic behaviour of the solution $\partial u/\partial \bn$ on each of the sides of the polygon. For sound-soft convex polygons, this behaviour was determined in \cite{HeLaMe:11,Convex}. A key contribution of this paper is to introduce new methods of argument which enable us to deduce precisely and rigorously this behaviour for a range of cases when the polygon is not convex.

\begin{figure}[t]
\begin{center}
\subfigure[Multiple reflections. MR=multiply-reflected ray, DR=diffracted-reflected ray.]{
\begin{tikzpicture}[>=triangle 45,x=1.0cm,y=1.0cm, scale=1.0]
\node (0,0) {\includegraphics[width=5cm]{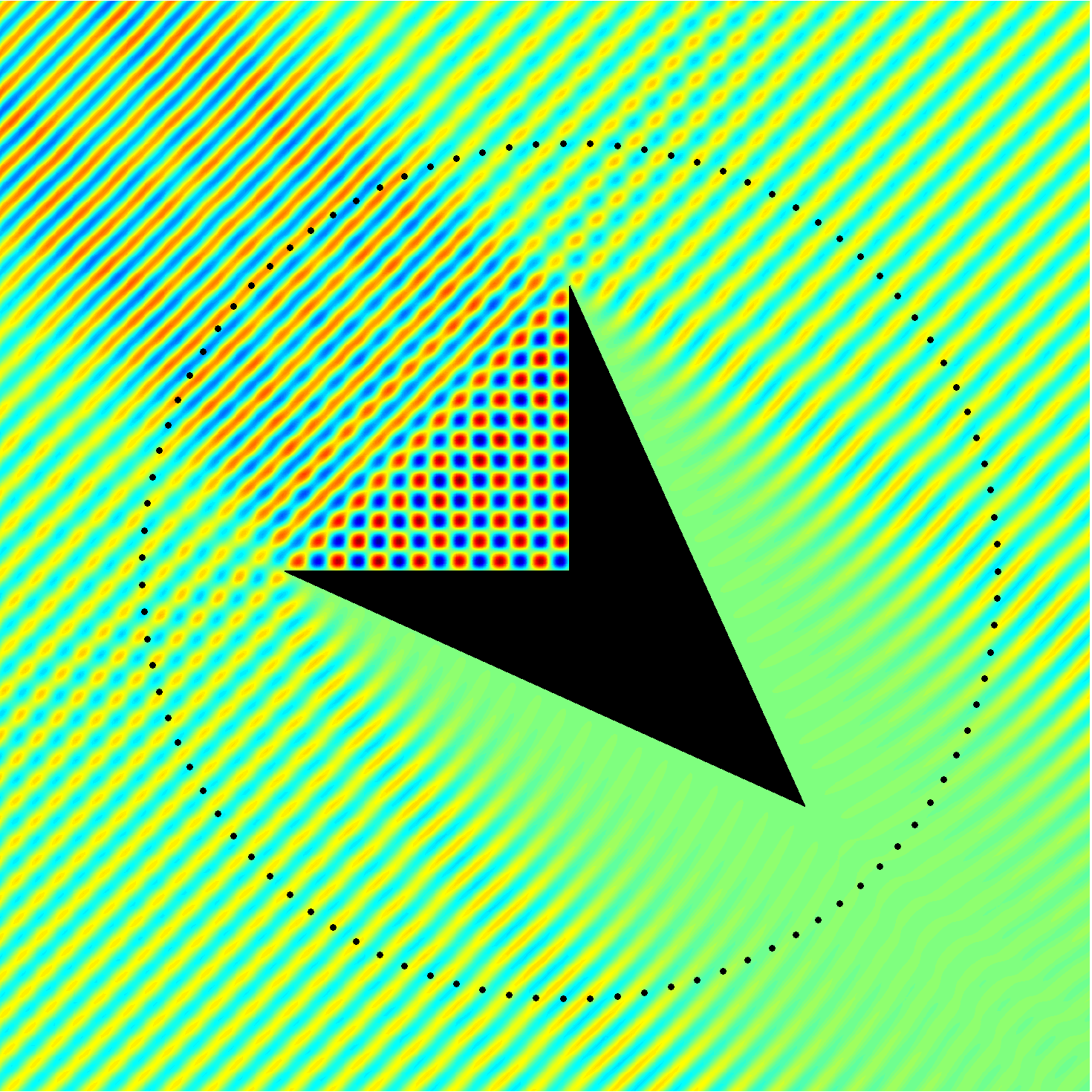}};
\draw[->,thick] (-2,2)--(-1,1);
\begin{scope}[>=latex,xshift=0.12cm,yshift=-0.12cm]
\draw[->,thick] (-0.8,1.1)--(0,0.3) -- (-0.3,0) -- (-1.2,0.9);
\draw (-1.1,0.8) node[anchor=east]{MR};
\draw[->,thick] (-0.4,1.68)--(0,1.28) -- (-0.7,0) -- (-1.05,0.35*1.28);
\draw (-1.0,0.35*1.28-0.1) node[anchor=east]{DR};
\end{scope}
\draw (-1.35,1.75) node {$\mathbf{d}$};
\end{tikzpicture}
}
\hs{5}
\subfigure[Partial illumination. SB=shadow boundary.]{
\begin{tikzpicture}[>=triangle 45,x=1.0cm,y=1.0cm, scale=1.0]
\node (0,0) {\includegraphics[width=5cm]{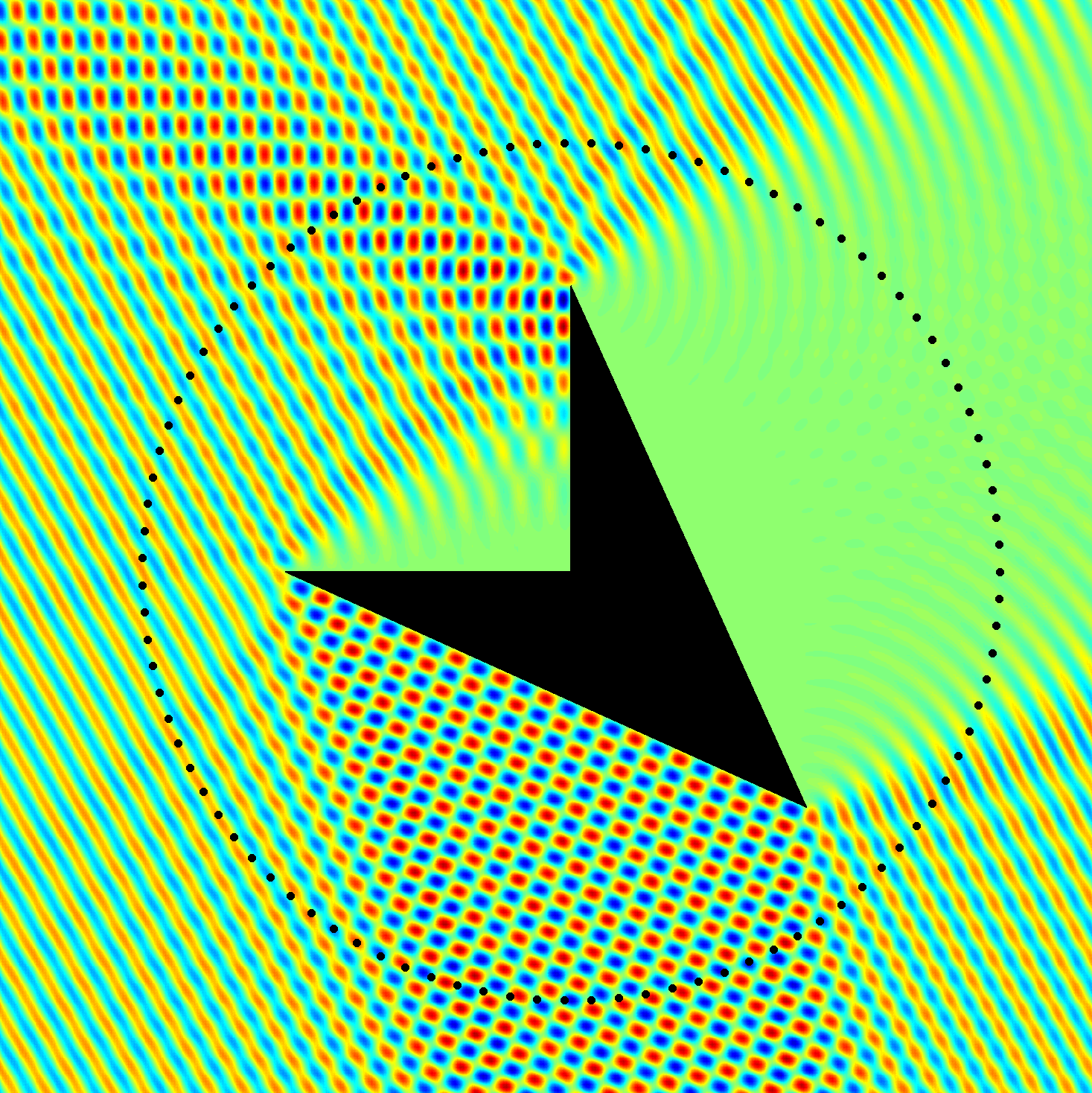}};
\draw[->,thick] (-2,-1.8) -- ++(1.225,0.707);
\draw (-1.5,-1.8) node {$\mathbf{d}$};
\draw[dashed,thick] (-1.2,-0.12)--++(1.1*1.225,1.1*0.707);
\draw (-0.35,0.5) node[anchor=east]{SB};
\end{tikzpicture}
}
\end{center}
\caption{Plots of the real part of the total field for scattering of a plane wave by a sound-soft nonconvex polygon for two incidence directions $\bd$ (exact dimensions are given in~\S\ref{sec:num}; the circle surrounding the scatterer is used for the computation of errors in the total field, see Figure~\ref{fig:total4}).}
\label{fig:total1}
\end{figure}
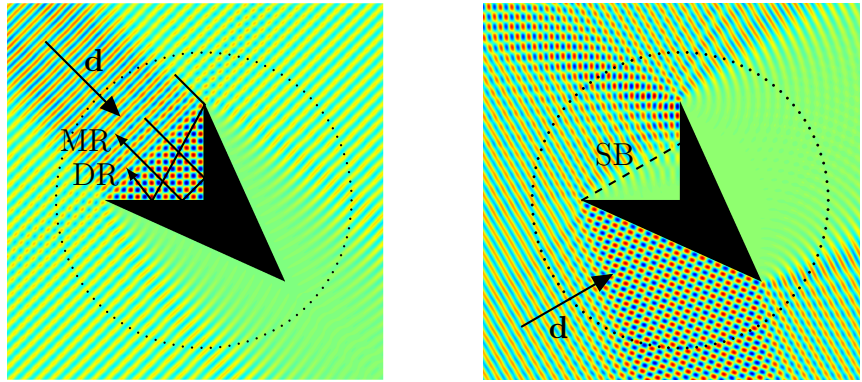

As alluded to in \S\ref{Introduction}, the main difficulty in developing HNA methods for nonconvex scatterers is that the high frequency asymptotic behaviour, %
knowledge of which is required for the choice of $V_0$ and $\psi_m$ in~(\ref{eqn:ansatz}),
is significantly more complicated than in the convex case. %
For polygonal scatterers in 2D two additional complexities are illustrated in Figure~\ref{fig:total1}. %
First, multiply-reflected and diffracted-reflected rays can be present in the asymptotic solution, as in Figure~\ref{fig:total1}(a). (These do not occur in the convex case, where all reflected rays propagate to infinity without further interaction with the scatterer.) We expect this to increase the number of terms required in the HNA ansatz \rf{eqn:ansatz}. Second, there is the possibility of partial illumination of a side of the polygon by one of the ray fields in the asymptotic solution, as in Figure~\ref{fig:total1}(b).
To explain the significance of this effect, we note that in the schemes proposed for convex polygons in \cite{Convex,ChLaMo:11,HeLaMe:11}, the sides of the polygon are classified according to whether they are ``illuminated'' or ``in shadow'' with respect to the incident wave, with a different approximation space being used on the two types of side. In the nonconvex case, a side can be partially illuminated and partially in shadow, because of the shadowing effect of another part of the scatterer, as for the vertical side in Figure~\ref{fig:total1}(b).
Across the shadow boundary between the illuminated and shadow regions the solution varies smoothly, but increasingly rapidly as the frequency increases, approaching the jump discontinuity predicted by the classical ``geometrical optics'' approximation in the limit of infinite frequency. This rapid variation must be correctly captured by the HNA ansatz \rf{eqn:ansatz}.

To restrict the complexity of the asymptotic behaviour that can arise, and to allow a full numerical analysis of our HNA method, we will focus our attention on the following particular class of polygons.%
\begin{defn}
\label{classCDef}
Let $\cC$ denote the class of all polygons $\Omega\subset\R^2$ for which the following two conditions are satisfied:
\begin{enumerate}[(i)]
\item ``Orthogonality'': Each external angle is either greater than $\pi$ or equal to $\pi/2$.
\item ``Visibility'': For each external angle equal to $\pi/2$, if $\Omega$ is rotated into the configuration in Figure \ref{fig:scatterers}(a), then $\Omega$ is contained entirely in the region bounded by the sides $\Gamma_{\rm nc}$ and $\Gamma_{\rm nc}'$ and the two dotted lines.
\end{enumerate}\end{defn}
\begin{figure}[ht!]
  \begin{center}
     \subfigure[]{\label{fig:scatterer0}
  	\begin{tikzpicture}[line cap=round,line join=round,>=triangle 45,x=1.0cm,y=1.0cm, scale=0.55]
\fill[line width=0pt,color=cqcqcq,fill=cqcqcq] (8.26,-0)--(5,3)--  (5,0)-- (2,0)-- (3.08,-2) -- (8.26,-2) -- cycle;
	\draw [line width=1.2pt] (8.26,-0)--(5,3)--  (5,0)-- (2,0)-- (3.08,-2);
	\draw (4.4,1.4) node {$\Gamma_{\rm nc}'$};
	\draw (3.6,0.4) node{$\Gamma_{\rm nc}$};
	\draw [line width=1pt,dotted] (5,3)-- (8.26,3);
	\draw [line width=1pt,dotted] (2,0)-- (2,-2);
	\draw (5,0.4)-- (4.6,0.4) -- (4.6,0);
		\draw (5,2.6)-- (5.4,2.6) -- (5.4,3);
			\draw (2,-0.4)-- (2.4,-0.4)-- (2.4,0);
	\draw (6.3,-1.1) node {$\Omega$};
\end{tikzpicture}
    }
    \hs{2}
    \subfigure[Star-like.]{\label{fig:scatterer1}
   	\begin{tikzpicture}[>=triangle 45,x=1.0cm,y=1.0cm, scale=0.8]
   			\fill[line width=0pt,color=cqcqcq,fill=cqcqcq] (0,0)-- (0,-2) -- (-2,-2) -- (1.65,-3.65)  -- cycle;
		\draw [line width=1.2pt] (0,0)-- (0,-2);
		\draw (0.05,-0.7) node[anchor=east] {NC};
		\draw [line width=1.2pt] (0,-2)-- (-2,-2);
		\draw  (-1,-2) node[anchor=south] {NC};
		\draw [line width=1.2pt] (-2,-2)-- (1.65,-3.65);
		\draw (-0.2,-3) node[anchor=east] {C};
		\draw [line width=1.2pt] (1.65,-3.65)-- (0,0);
		\draw  (0.8250,-1.825) node[anchor=west] {C};
			\draw (0.4,-2.2) node {$\Omega$};
		\end{tikzpicture}
    }
        \hs{2}
    \subfigure[Non-star-like.]{\label{fig:scatterer2}
    \begin{tikzpicture}[>=triangle 45,x=1.0cm,y=1.0cm, scale=0.36]
    				\fill[line width=0pt,color=cqcqcq,fill=cqcqcq] (6,4)-- (1.75,4)-- (1.75,-1)--(0,-1)-- (0,-3)--(4,-3)--(4,2)-- (6,2)-- cycle;
	\draw [line width=1.2pt] (6,4)-- (1.75,4);
		\draw (3.7,4) node[anchor=south] {C};
		\draw [line width=1.2pt] (1.75,4)-- (1.75,-1);
		\draw (1.9,1.8) node[anchor=east] {NC};
		\draw [line width=1.2pt] (1.75,-1)-- (0,-1);
		\draw (0.8,-1.1) node[anchor=south] {NC};
		\draw [line width=1.2pt] (0,-1)-- (0,-3);
		\draw (0,-2) node[anchor=east] {C};
		\draw [line width=1.2pt]  (0,-3)-- (4,-3);
		\draw (2,-3) node[anchor=north] {C};
		\draw [line width=1.2pt] (4,-3)-- (4,2);
		\draw (3.9,-1.1) node[anchor=west] {NC};
		\draw[line width=1.2pt]  (4,2)-- (6,2);
	\draw (5,2.1) node[anchor=north] {NC};
		\draw [line width=1.2pt] (6,2)-- (6,4);
		\draw (6,3) node[anchor=west] {C};
					\draw (2.7,0.5) node {$\Omega$};
				\end{tikzpicture}
}
  \end{center}
  \caption{(a) Illustration of condition 2 in Definition \ref{classCDef}; %
(b)-(c) examples of polygonal scatterers in the class $\cC$, with convex (C) and nonconvex (NC) sides labelled.}
  \label{fig:scatterers}
\end{figure}
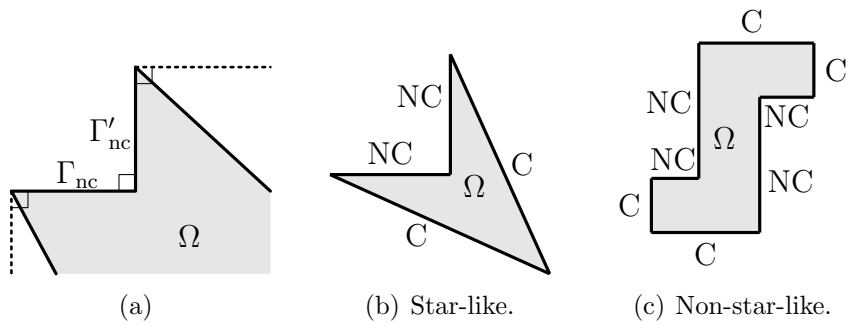
For a polygon in the class $\cC$ we define two types of side:  if the external angles at the endpoints of the side are both greater than $\pi$ then we say that it is a ``convex'' side; if one is equal to $\pi/2$ then we say that it is a ``nonconvex'' side; note that nonconvex sides come in pairs. We say that a convex side is illuminated by the incident wave if $\bd\cdot \bn<0$ on the side, and is in shadow if $\bd\cdot \bn\geq 0$.%

In Figure~\ref{fig:scatterers}(b)-(c) we show two examples of members of the class $\cC$, one star-like and one non-star-like. Of course, $\cC$ represents only a small subset of the set of all nonconvex polygons; in \S\ref{sec:generalisations} we provide detailed suggestions as to how the design of our HNA approximation space, and our rigorous analysis, might be generalised to polygons outside this class.

Our approach to tackling the issues of multiple reflections and partial illumination described above is to follow the spirit of high frequency asymptotic methods such as the Geometrical Theory of Diffraction \cite{BoKi:94}, and consider simple ``canonical problems'' which encapsulate the behaviour in question.
This is consistent with the approach taken for the convex polygon case in \cite{HeLaMe:11,Convex}, where the reflection of the incident wave by the illuminated sides is treated by considering the canonical problem of reflection by a half-plane (cf.\ \cite[pp.~621-622]{Convex}). For nonconvex polygons in the class $\cC$, the canonical problem associated with multiple reflections is that of scattering in a quarter-plane. The canonical problem associated with partial illumination is that of diffraction by a wedge (equivalently, as we shall see, diffraction by a knife edge).
We shall now show how consideration of these canonical problems allows us to choose $V_0$ and $\psi_m$ appropriately in~(\ref{eqn:ansatz}) so that $V_m$, $m=1,\ldots,M$, are non-oscillatory.

\subsection{Behaviour on convex sides}

We first consider the behaviour on a typical convex side, which we denote $\Gamma_{\rm c}$. As illustrated in Figure~\ref{fig:convex}, $\bP^{\pm}$ will denote the endpoints of $\Gamma_{\rm c}$, and $\omega^\pm\in(\pi,2\pi)$ the corresponding exterior angles. A point $\bx$ on $\Gamma_{\rm c}$ is given in terms of the arc length $s$ measured from $\bP^+$ by $\bx(s)=\bP^++ (s/L_{\rm c})(\bP^--\bP^+)$ for $s\in [0,L_{\rm c}]$, where $L_{\rm c}=|\bP^--\bP^+|$ is the length of $\Gamma_{\rm c}$. The analysis for convex polygons in \cite{HeLaMe:11,Convex} carries over virtually verbatim to this case. Precisely, arguing as in \cite[\S 3]{HeLaMe:11} %
gives:
\begin{thm}
\label{thm:1}
On a convex side $\Gamma_{\rm c}$, %
\begin{align}
  \pdone{u}{\bn}(\bx(s)) = \Psi(\bx(s)) + v^+(s)\re^{\ri ks} +v^-(L_{\rm c}-s) \re^{-\ri ks}, %
\label{Decomp}
\end{align}
for $s\in[0,L_{\rm c}]$,
where
\begin{enumerate}[(i)]
\item $\Psi:=2\pdonetext{ u^i}{\bn}$ if $\Gamma_{\rm c}$ is illuminated and $\Psi:=0$ otherwise;
\item the functions $v^\pm(s)$ are analytic in the right half-plane $\real{s}>0$; further, for every $k_0>0$ we have
\begin{align}
  \label{vjpmBounds}
  |v^\pm(s)|\leq \begin{cases} C \uM k|ks|^{-\delta^\pm}, & 0<|s|\leq 1/k,\\
  C \uM k|ks|^{-1/2}, & |s|> 1/k,
  \end{cases}
  \quad \real{s}>0,
\end{align}
for $k\geq k_0$, where $\delta^\pm :=1-\pi/\omega^\pm \in (0,1/2)$,
\begin{align}
\label{}
\uM:=\sup_{\bx\in D}|u(\bx)|,
\end{align}
and the constant $C>0$ depends only on $\Omega$ and $k_0$.
\end{enumerate}
\end{thm}

\begin{figure}[t]
  \begin{center}
  	\begin{tikzpicture}[line cap=round,line join=round,>=triangle 45,x=1.0cm,y=1.0cm, scale=1]
  		\fill[line width=0pt,color=cqcqcq,fill=cqcqcq](-1,-1)--(0,0)--(3,-0) -- (4,-1)-- cycle;
	\draw [line width=1.2pt] (-1,-1)--(0,0)--(3,-0) -- (4,-1);
	\draw (0,0.25) node {$\bP^{-}$};
	\draw (0.6,0) arc (0:225:0.6);
	\draw (-0.6,0.6) node {$\omega^{-}$};
	\draw (3.2,0.25) node {$\bP^{+}$};
	\draw (2.4,0) arc (180:-45:0.6);
	\draw (3.8,0.6) node {$\omega^{+}$};
 	\draw (1.6,0.25) node {$\Gamma_{\rm c}$};
	\draw (1.5,-0.8) node {$\Omega$};
	\filldraw (1,0) circle (2pt);
	\draw (0.9,0.25) node {$\bx$};
	\draw [<->] (1,-0.2) -- (3,-0.2);
	\draw (2,-0.4) node {$s$};
	\draw [line width=1pt,dotted] (-1.5,0)-- (0,0);
	\draw [line width=1pt,dotted] (3,0)-- (4.5,0);
	\draw (-1.3,-0.25) node {$\gamma^-$};
	\draw (4.3,-0.25) node {$\gamma^+$};
	\draw (-1.3,0.8) node {$H$};
\end{tikzpicture}
\end{center}
  \caption{Geometry of a typical convex side $\Gamma_{\rm c}$.}
  \label{fig:convex}
\end{figure}
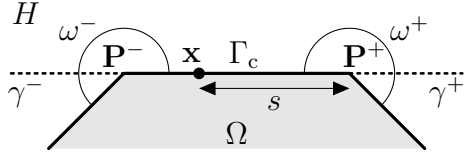

\begin{proof}Follows similar arguments to those used in \cite{HeLaMe:11,Convex}. We merely summarise the key steps in the proof here, in order to motivate the arguments used in the (more complicated) proof of the analogous result for the nonconvex sides (Theorem \ref{Gamma2Thm} below). The key first step is to apply Green's representation theorem in the half plane $H\subset D$ whose boundary extends $\Gamma_{\rm c}$ (cf. Figure~\ref{fig:convex}). The Dirichlet Green's function for $H$ is known explicitly by the method of images in terms of the fundamental solution $\Phi_k$. This gives $\pdonetext{u}{\bn}$ on $\Gamma_{\rm c}$ as a leading order term, plus the sum of two integrals over the contours $\gamma^\pm$ of Figure~\ref{fig:convex}. The integrand in the integral over $\gamma^\pm$ contains $u$ restricted to $\gamma^\pm$ as a factor, and the integrals over $\gamma^+$ and $\gamma^-$ correspond to the diffracted fields emanating from the corners $\bP^+$ and $\bP^-$, respectively. This motivates extracting out as factors the respective phases $\re^{\ri ks}$ and $\re^{-\ri ks}$, leaving the factors $v^+(s)$ and $v^-(L_{\rm c}-s)$ in \eqref{Decomp}.  Finally, using analyticity properties of the Hankel function that appears in the definition of $\Phi_k$, it can be shown that the functions $v^\pm(s)$ can be analytically continued into the complex plane where they satisfy the bounds \eqref{vjpmBounds} (see \cite[Theorem~3.2]{HeLaMe:11}).
\end{proof}

\begin{rem} \label{rem:M} The dependence of $\uM$ on the wavenumber $k$ is not yet fully understood. In \cite[Theorem 4.3]{HeLaMe:11} it is shown that $\uM = \mathcal{O}(k^{1/2}\log^{1/2}k)$ as $k\to\infty$, uniformly with respect to the angle of incidence, when $\Omega$ is a star-like polygon.  However, it is plausible, and consistent with the numerical results in \S\ref{sec:num},
that in fact $\uM=\ord{1}$ as $k\to\infty$ in this case, and indeed for the whole class $\cC$.
\end{rem}

\begin{rem}
\label{rem:PO}
The representation~\rf{Decomp} can be interpreted in terms of high frequency asymptotics %
as follows. The first term, $\Psi$ (corresponding to $V_0$ in~(\ref{eqn:ansatz})), is the geometrical optics approximation to $\pdonetext{u}{\bn}$, representing the contribution of the incident and reflected rays (where they are present). (Using this approximation alone in the representation \rf{RepThm} gives the ``physical optics'' approximation of $u$ in $D$.)  The second and third terms in~\rf{Decomp} represent the combined contribution of all the diffracted rays emanating from the corners $\bP^+$ and $\bP^-$, respectively (including those multiply-diffracted rays which have travelled arbitrarily many times around the boundary).
\end{rem}
\subsection{Behaviour on nonconvex sides}
\label{sec:nonconvexreg}
We now consider the typical behaviour on a nonconvex side, which we denote $\Gamma_{\rm nc}$. As illustrated in Figure \ref{fig-1}(a), $\bP$ and $\bQ$ will denote the endpoints of $\Gamma_{\rm nc}$, and $\bR$ and $\bQ$ the endpoints of the adjoining nonconvex side, which we denote $\Gamma_{\rm nc}'$. We let $L_{\rm nc}$ and $L_{\rm nc}'$ denote the lengths of $\Gamma_{\rm nc}$ and $\Gamma_{\rm nc}'$, respectively, and we denote the exterior angle at $\bP$ by $\omega$. A point $\bx$ on $\Gamma_{\rm nc}$ is then given in terms of the arc length $s$ measured from $\bQ$ by $\bx(s)=\bQ+ (s/L_{\rm nc})(\bP-\bQ)$ for $s\in [0,L_{\rm nc}]$. We also introduce local Cartesian coordinates $\bx=(x_1,x_2)$ and polar coordinates $(r,\theta)$ (both with the origin at $\bR$), as defined in Figure~\ref{fig-1}(a). We note that any nonconvex side can be transformed to this configuration by a rotation and a reflection of $\Omega$.

\begin{figure}[t]
\begin{center}
 \hs{-5}
 \subfigure[Local coordinates on $\Gamma_{\rm nc}$]{
	\begin{tikzpicture}[line cap=round,line join=round,>=triangle 45,x=1.0cm,y=1.0cm, scale=0.6]
	\fill[line width=0pt,color=cqcqcq,fill=cqcqcq](9,0.5)--(5,3)--  (5,0)-- (2,0)-- (3.08,-2) -- (9,-2) -- cycle;
	\draw [line width=1.2pt] (9,0.5)--(5,3)--  (5,0)-- (2,0)-- (3.08,-2);
	\draw (4.6,3.2) node {$\bR$};
	\draw (5.3,-0.2) node {$\bQ$};
	\draw (2,0.3) node {$\bP$};
	\draw (2.6,0) arc (0:300:0.6);
	\draw (1.4,-0.7) node {$\omega$};
	\draw (4.5,1.05) node {$\Gamma_{\rm nc}'$};
	\draw (3.9,0.35) node{$\Gamma_{\rm nc}$};
	\draw (5,0.38)-- (4.58,0.38);
	\draw (4.58,0.38)-- (4.58,0);
	\draw [->] (5,3) -- (5,5);
	\draw [->] (5,3) -- (7,3);
	\draw (7.1,3.3) node {$x_1$};
	\draw (5.5,4.8) node {$x_2$};
	\draw [line width=1pt,dotted] (-0.5,0)-- (2,0);
	\draw [line width=1pt,dotted] (5,3)-- (5,6);
	\draw (-0.4,-0.4) node {$\gamma$};
	\draw (4.6,5.8) node {$\gamma'$};
	\draw (0.5,4.5) node {$Q$};
	\draw [->] (5,2.2) arc (-90:26.57:0.8);
	\draw (6,2.7) node {$\alpha$};
	\draw (9.5,4) -- (8.5,6) ;
	\draw (9.3,3.9) -- (8.3,5.9) ;
	\draw (9.1,3.8) -- (8.1,5.8) ;
	\draw [thick,->] (9,5) -- (8,4.5) ;	
	\draw [dashed] (5,3) -- (9,5);
	\draw (9.5,5.2) node {$u^i$};
	\draw (8.5,4.2) node {$\bd$};
	\draw (6.7,-1.1) node {$\Omega$};
	\filldraw (8,0) circle (3pt);
	\draw (8,0.3) node {$\bP'$};
	\filldraw (2.9,0) circle (3pt);
	\draw (2.8,0.3) node {$\bx$};
	\draw [<->] (3.05,0.2) -- (4.9,2.8);
	\draw (3.7,1.5) node {$r$};
	\draw [->] (5,1.6) arc (-90:237:1.4);
	\draw (4,4.4) node {$\theta$};
	\draw [<->] (3.0,-0.2) -- (4.9,-0.2);
	\draw (3.9,-0.5) node {$s$};
\end{tikzpicture}
}
 \subfigure[Diffraction by a knife edge]{
	\begin{tikzpicture}[line cap=round,line join=round,>=triangle 45,x=1.0cm,y=1.0cm, scale=0.6]
	\draw [line width=1.2pt] (5,3)-- (5,-2);
	\draw [->] (5,2.2) arc (-90:26.57:0.8);
	\draw (6,2.6) node {$\alpha$};
	\draw (9.5,4) -- (8.5,6) ;
	\draw (9.3,3.9) -- (8.3,5.9) ;
	\draw (9.1,3.8) -- (8.1,5.8) ;
	\draw [thick,->] (9,5) -- (8,4.5) ;	
	\draw [dashed] (5,3) -- (9,5);
	\draw (9.5,5.2) node {$u^i$};
	\draw (8.5,4.2) node {$\bd$};
	\draw (4.6,3.2) node {$\bR$};
	\filldraw (2.9,0) circle (3pt);
	\draw (2.8,0.3) node {$\bx$};
	\draw [<->] (3.05,0.2) -- (4.9,2.8);
	\draw (3.7,1.5) node {$r$};
	\draw [->] (5,1.6) arc (-90:237:1.4);
	\draw (4,4.4) node {$\theta$};
\end{tikzpicture}
}\caption{Geometry of a typical nonconvex side $\Gamma_{\rm nc}$.
}
\label{fig-1}
\end{center}
\end{figure}
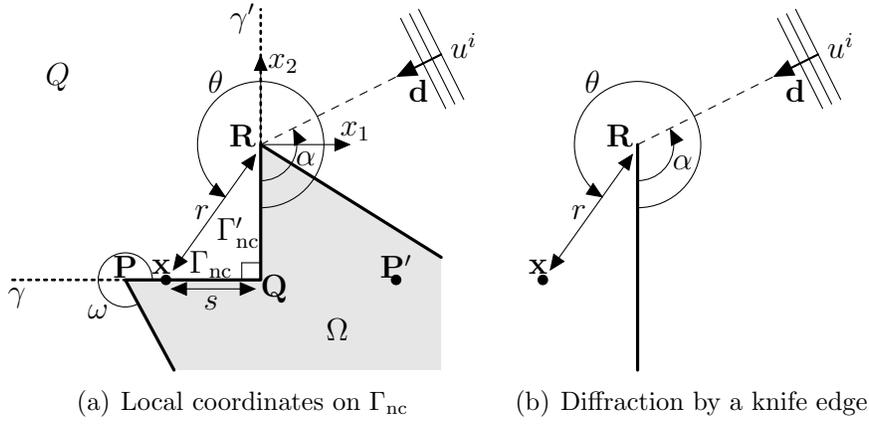

We expect the high frequency asymptotic behaviour of $\pdonetext{u}{\bn}$ on $\Gamma_{\rm nc}$ to involve: diffracted waves from the corners $\bP$ and $\bR$; reflection by the side $\Gamma_{\rm nc}'$; and, depending on the direction of incidence, illumination (partial or otherwise). %
One might expect the leading order behaviour on $\Gamma_{\rm nc}$ to be given by the canonical solution for diffraction of $u^i$ by the infinite wedge formed by extending the two sides emanating from $\bR$ towards the bottom right of Figure~\ref{fig-1}(a). In fact, it is sufficient to consider a simpler canonical solution, namely that for diffraction of $u^i$ by the infinite knife edge
formed by extending the side $\Gamma_{\rm nc}'$ towards the bottom of Figure \ref{fig-1}(a), as illustrated in Figure \ref{fig-1}(b).
This is because, on $\Gamma_{\rm nc}$, the difference between these two canonical solutions is, in the high frequency regime ($kr\to\infty$), a circular wave emanating from $\bR$ of the form $d(\theta)\re^{\ri k r}/\sqrt{kr}$, where the angle-dependent amplitude $d(\theta)$ varies slowly across the incident/reflected shadow boundaries \cite{Oberhettinger}.

\begin{lem}[{\cite[\S8.2]{BoSeUs:87}}]
\label{def:udDef}
Let $(r,\theta)$ be polar coordinates with $r\in[0,\infty)$ and $\theta\in[0,2\pi)$. Let $u^d$ denote the solution to the problem of diffraction of the plane wave $u^i=\re^{\ri k \bx\cdot \bd}$ by the infinite knife edge $\left\{(r,0)\,:\,r\in[0,\infty)\right\}$ with Dirichlet boundary conditions. If $\theta=\alpha$ is the direction from which the incident wave arrives (as in Figure \ref{fig-1}(b)), then%
\[ %
u^d(r,\theta,\alpha) =E(r,\theta-\alpha) - E(r,\theta+\alpha),
\]%
where
$E(r,\psi) = \re^{-\ri kr\cos{\psi}}{\rm Fr}(-\sqrt{2kr}\cos(\psi/2))$,
and ${\rm Fr}$ is a Fresnel integral, defined as the improper integral
${\rm Fr}(\mu) = (\re^{-\ri\pi/4}/\sqrt{\pi}) \int_\mu^\infty \re^{\ri z^2} \,\rd z$.
We note that $E(r,\psi)$ is $4\pi$-periodic in $\psi$, and, by standard properties of ${\rm Fr}$ (cf.\ e.g.\ \cite[\S7]{DLMF}),
\begin{align*}
\label{}
E(r,\psi) \sim
\begin{cases}
\re^{-\ri k r\cos\psi} + \tilde{d}(\psi)\frac{\re^{\ri k r}}{\sqrt{kr}}\left(1+\ord{\frac{1}{kr}}\right), & \psi \in [(4n+1)\pi+\delta,(4n+3)\pi-\delta],\\
\tilde{d}(\psi)\frac{\re^{\ri k r}}{\sqrt{kr}}\left(1+\ord{\frac{1}{kr}}\right), & \psi \in [(4n-1)\pi+\delta,(4n+1)\pi-\delta],
\end{cases}
\end{align*}
as $kr\to \infty$, where $\tilde{d}(\psi) = -\re^{\ri \pi/4}/(2 \sqrt{2\pi}\cos{(\psi/2)})$, $n\in\Z$, $0<\delta<\pi$ is arbitrary and the approximations hold uniformly in $\psi$ in the stated intervals. The term $\re^{-\ri k r\cos\psi}$ represents a plane wave propagating from the direction $\psi=0$, and $\tilde{d}(\psi)\frac{\re^{\ri k r}}{\sqrt{kr}}$ represents a circular wave emanating from $r=0$ with directionality $\tilde{d}(\psi)$.
\end{lem}

The key result that we require for the design of our approximation space is the following theorem, which we prove in~\S\ref{NonConvProof}.
\begin{thm}
\label{Gamma2Thm}
Suppose that Assumption \ref{ass:2} holds. Then, on a nonconvex side $\Gamma_{\rm nc}$, %
\begin{align}
\label{dudnGamma2Rep}
\pdone{u}{\bn}(\bx(s))&=\Psi(\bx(s)) + v^+(L_{\rm nc}+s)\re^{\ri ks} +v^-(L_{\rm nc}-s) \re^{-\ri ks}+v(s)\re^{\ri kr}, %
\end{align}
for $s\in [0,L_{\rm nc}]$,
where $r=r(s)=\sqrt{s^2+L_{\rm nc}'^2}$ and
\begin{enumerate}[(i)]
\item
$\Psi:=2\pdonetext{u^d}{\bn}$ if $\pi/2\leq\alpha\leq3\pi/2,$ and $\Psi:=0$ otherwise;
\item the functions $v^\pm(s)$ are analytic in $\real{s}>0$; further, for every $k_0>0$ they satisfy the bounds \rf{vjpmBounds} for $k\geq k_0$, with $\delta^\pm = 1-\pi/\omega\in(0,1/2)$ and $C>0$ depending only on $\Omega$ and $k_0$;
\item the function $v(s)$ is analytic in the $k$-independent complex neighbourhood
$D_\eps:=\left\{s\in\mathbb{C} : \mathrm{dist}(s,[0,L_{\rm nc}])<\eps\right\}$ of $[0,L_{\rm nc}]$,
where
\begin{equation}
\label{eqn:eps_defn}
\eps:= L_{\rm nc}'/(32\sqrt{2}\,);%
\end{equation}
 further, where $k_1$ and $C_1$ are the constants from Assumption \ref{ass:1},
\begin{align}
\label{tildevEst}
\left|v(s)\right|&\leq C C_1 k\log^{1/2}(2+k), \qquad s\in D_\eps,\, k\geq k_1,
\end{align}
where $C>0$ depends only on $\Omega$ and $k_1$.
\end{enumerate}
\end{thm}

\begin{rem}
\label{rem:LO}
The representation \rf{dudnGamma2Rep} can be interpreted in terms of high frequency asymptotics as follows. The first term, $\Psi$, represents a uniform approximation to the leading order high frequency behaviour of $\pdonetext{u}{\bn}$ on $\Gamma_{\rm nc}$, in the form of a modified geometrical optics approximation; depending on the value of $\alpha$, this includes contributions from the incident wave (via $E(r,\theta-\alpha)$), and the reflection of the incident wave in $\Gamma_{\rm nc}'$ (via $E(r,\theta+\alpha)$), with the jump discontinuties of the geometrical optics approximation smoothed by the use of Fresnel integrals. The final term represents the contribution due to diffracted rays emanating from the corner $\bR$, and also compensates for use of the knife edge canonical solution in $\Psi$ rather than the wedge canonical solution (cf.\ the discussion before Lemma \ref{def:udDef}). The third term represents the contribution due to diffracted rays emanating from the corner $\bP$, and the second term represents the contribution due to ``diffracted-reflected'' rays emanating from the corner $\bP$ and being reflected at $\bQ$. These rays can be thought of as emanating from a non-physical ``image corner'' $\bP'$ (cf.\ Figure \ref{fig-1}(a)), obtained by the reflection of $\bP$ in $\Gamma_{\rm nc}'$. (Hence, while the third term is singular at $\bP$, the second term is not singular at either $\bP$ or $\bQ$.)%
\end{rem}

\section{Proof of Theorem \ref{Gamma2Thm}}
\label{NonConvProof}

We begin by outlining the structure of the proof of Theorem~\ref{Gamma2Thm}.
We adopt a similar methodology to that used in the proof of the corresponding result for convex sides, Theorem~\ref{thm:1}, although significant modifications and new ideas are needed to deal with the nonconvex geometry.
We begin (in Lemma~\ref{uRepThmQ}) by applying Green's representation theorem in the quarter plane $Q$ whose boundary extends the sides $\Gamma_{\rm nc}$ and $\Gamma_{\rm nc}'$, as illustrated in Figure~\ref{fig-1}(a). The Dirichlet Green's function for this domain is known explicitly (see \eqref{GFQP}) by the method of images. (This simple representation for the Green's function simplifies the calculations throughout this section; it is this which motivates the requirement in Definition~\ref{classCDef} that the exterior angles less than $\pi$ are exactly $\pi/2$.) This gives $\pdonetext{u}{\bn}$ on $\Gamma_{\rm nc}$ as a leading order term, plus the sum of integrals over the contours $\gamma$ and $\gamma'$ of Figure \ref{fig-1}(a); these integrals contain $u$ restricted to $\gamma$ or $\gamma'$ as a factor (see~\rf{dudnRepIntegral}).

We expect the integral over $\gamma$ to correspond to the field diffracted at $\bP$, and its subsequent reflection at $\bQ$. In fact, this integral can be analysed exactly as for a convex side, and gives rise to the terms $v^+(L_{\rm nc}+s)\re^{\ri ks}$ and $v^-(L_{\rm nc}-s) \re^{-\ri ks}$ in the representation~\rf{dudnGamma2Rep}. The analysis of the integral over $\gamma'$, which gives rise to the remaining terms in~\rf{dudnGamma2Rep}, corresponding to the field diffracted at $\bR$, is considerably more complicated.

To analyse the integral over $\gamma'$ we split it further, using the fact that $u=u^i+u^s$.  We consider the contribution from $u^s$ in Lemma~\ref{v2tildeThm} (proved in \S\ref{subsec:v2tildeThm}), where we extract the expected phase $\re^{\ri k r}$ and show that the remaining factor $\tW(s)$ can be analytically continued into the complex plane. This is the most technical part of the proof. First we substitute for $u^s$ using the representation theorem~\rf{RepThm}, which gives, after an application of Fubini's theorem, the representation \eqref{eqn:v2} as an integral around $\Gamma$ involving $\pdonetext{u}{\bn}$.   The next task is to show that $K(\cdot, \bz)$ in the integrand, given as the integral \eqref{eqn:Kxz} along $\gamma'$,  has an analytic continuation into a ($\bz$- and $k$-independent) neighbourhood of $[0,L_{\rm nc}]$, and to bound $K(\cdot, \bz)$ in this neighbourhood for $\bz\in \Gamma$ (Lemma \ref{KAnalCor}). To achieve this aim it is convenient first to show that one can deform the contour of integration $\gamma'$ in \eqref{eqn:Kxz} to a contour on which the integrand decays exponentially, obtaining the representation \eqref{KRepDeform}. To show these results we require auxiliary results, Lemmas~\ref{ImChiLem}, \ref{ImphiCor}, and \ref{SkLem}. In a final step we bound $\tW(s)$ in the complex plane via the application \eqref{VEstKdudn} of the Cauchy-Schwarz inequality, bounding $\|\pdonetext{u}{\bn}\|_{L^2(\Gamma)}$ using Assumption \ref{ass:1}. (It is precisely at this point where Assumption \ref{ass:1} is needed.)

We consider the contribution from $u^i$ in Lemma~\ref{v1tildeThm} (proved in~\S\ref{subsec:v1tildeThm}), where we apply a similar (but simpler) approach, making use of the tools developed in the proof of Lemma~\ref{v2tildeThm}.  There is one complication: when $\alpha\in(\pi/2,3\pi/2)$, in which case $\Gamma_{\rm nc }$ is partially (or fully) illuminated by the incident wave, we first have to subtract off the canonical solution $u^d$ from $u^i$. The analysis is completed by applying Green's representation theorem for $u^d$ in the half-plane $x_1<0$ (Proposition~\ref{udLem}).

We thus begin our proof of Theorem~\ref{Gamma2Thm} by deriving a representation formula for $u$ in the quarter-plane whose boundary contains the sides $\Gamma_{\rm nc}'$ and $\Gamma_{\rm nc}$. Let $\gamma:=\left\{(x_1,-L_{\rm nc}')\,:\,x_1<-L_{\rm nc}\right\}$ and $\gamma':=\left\{(0,x_2)\,:\,x_2>0\right\}$ denote the extensions of $\Gamma_{\rm nc}$ and $\Gamma_{\rm nc}'$, respectively (see Figure~\ref{fig-1}(a)). Then $\partial Q:=\gamma\cup\Gamma_{\rm nc}\cup\Gamma_{\rm nc}'\cup\gamma'$ is the boundary of the quarter-plane $Q:=\{(x_1,x_2) : x_1<0,\, x_2>-L_{\rm nc}' \}$ whose Dirichlet Green's function is, by the method of images,
\begin{align} \label{GFQP}
G_k(\bx,\by)&:=\Phi_k(\bx,\by)-\Phi_k(\bx,\by^*)-\Phi_k(\bx,\by')+\Phi_k(\bx,\by^*{}'),
\end{align}
where
$^*$ and $'$ are operations of reflection in the lines $\gamma\cup\Gamma_{\rm nc}$ and $\Gamma_{\rm nc}'\cup\gamma'$, respectively. For reference, the incident wave and its reflections in the extensions of these lines (assuming a sound-soft boundary condition~(\ref{eqn:bc1})) are given explicitly by
\begin{align*}
\label{}
u^i(\bx) &= \exp{(\ri k(-x_1\sin{\alpha} + x_2 \cos{\alpha}))},\\
(u^i)^*(\bx) &= -\exp{(\ri k(-x_1\sin{\alpha} - (x_2+2L_{\rm nc}') \cos{\alpha}))},\\
(u^i)'(\bx)& = -\exp{(\ri k(x_1\sin{\alpha} + x_2 \cos{\alpha}))},\\
(u^i)^*{}'(\bx) &= \exp{(\ri k(x_1\sin{\alpha} - (x_2+2L_{\rm nc}') \cos{\alpha}))}.
\end{align*}
We also recall that $\bn$ is the unit normal directed into $D$, i.e.\ into the interior of $Q$.

We then have the following representation formulae:
\begin{lem}
\label{uRepThmQ}
\begin{enumerate}[(i)]
\item \ \vs{-7}
\begin{align}
\label{usRepQ}
u^s(\bx)=\int_{\partial Q}{\frac{\partial G_k(\bx,\by)}{\partial \bn(\by)}u^s(\by)\,\rd s(\by)}, \qquad \bx\in Q;
\end{align}
\item \ \vs{-7}
\begin{align}
\label{uiRepQ}
u^i(\bx ) =\Psi_1(\bx )+\int_{\partial Q}\frac{\partial G_k(\bx ,\by )}{\partial \bn(\by )}u^i(\by )\,\rd s(\by ), \qquad \bx \in Q,
\end{align}
where, for  $\pi\leq\alpha\leq 3\pi/2$,
\begin{eqnarray*}
\Psi_1(\bx)& := &
u^i(\bx )+ (u^i)^*(\bx ) +(u^i)'(\bx )+ (u^i)^*{}'(\bx )\\
&=&4\exp(-\ri kL_{\rm nc}'\cos{\alpha})\,\sin\left(kx_1\sin\alpha\right)\sin\left(k(x_2+L_{\rm nc}')\cos\alpha\right),
\end{eqnarray*}
while $\Psi_1(\bx) := 0$, otherwise;
\item \ \vs{-7}
\begin{align*}
\label{}
u(\bx ) = \Psi_1(\bx )+\int_{\gamma\cup \gamma'}\frac{\partial G_k(\bx ,\by )}{\partial \bn(\by )}u(\by )\,\rd s(\by ), \qquad \bx \in Q.
\end{align*}
\end{enumerate}
\end{lem}

\begin{proof}
(i) For $R>0$ define $Q_R:=\left\{\by\in Q:\,\left|\by\right|<R\right\}$, with boundary $\partial Q_R$. %
By Green's theorem and Green's representation theorem \cite[Theorems 2.19, 2.20]{ChGrLaSp:11},%
\begin{align}
\label{usRepGammar}
u^s(\bx )=\int_{\partial Q_R}\left(\pdone{G_k(\bx ,\by )}{\bn(\by )}u^s(\by ) - G_k(\bx ,\by )\pdone{u^s}{\bn}(\by )\right)\,\rd s(\by ), \qquad \bx \in Q_R,
\end{align}
where the normal $\bn$ is directed into the interior of $Q_R$. Then, since $G_k(\bx ,\by )=0$ on $\partial Q$ and both $G_k(\bx, \cdot)$ and $u^s$ satisfy the Sommerfeld radiation condition, \rf{usRepQ} is obtained from \rf{usRepGammar} by taking the limit $R\to\infty$ (see, e.g., \cite[Theorem~3.3]{CoKr:83}).

(ii) For $R>0$ define $\bx _R:=-R\bd = R(\sin{\alpha},-\cos{\alpha})$, and $u^i_R(\bx ):=C_R \Phi_k(\bx ,\bx_R)$, where $C_R:=\re^{-\ri \pi/4}\sqrt{8\pi kR}\re^{-\ri kR}$.
Note that, for fixed $R$, $u^i_R(\bx )$ satisfies the Sommerfeld radiation condition as $|\bx |\to\infty$, but, for fixed $\bx $, $u^i_R(\bx )\to u^i(\bx )$ as $R\to\infty$. If $\alpha\not\in[\pi,3\pi/2]$, then $u^i_R(\bx )$ is regular in $Q$, and, arguing as in part~(i),
\begin{align}
\label{uiREqn}
u^i_R(\bx )=\int_{\partial Q}{\frac{\partial G_k(\bx ,\by )}{\partial \bn(\by )}u^i_R(\by )\,\rd s(\by )}, \qquad \bx \in Q.
\end{align}
If $\alpha\in(\pi,3\pi/2)$, $u^i_R(\bx )$ is singular at $\bx =\bx_R\in Q$, and \rf{uiREqn} must be modified to
\begin{align}
\label{uiREqn2}
u^i_R(\bx )=C_RG_k(\bx ,\bx_R) + \int_{\partial Q}{\frac{\partial G_k(\bx ,\by )}{\partial \bn(\by )}u^i_R(\by )\,\rd s(\by )}, \qquad \bx \in Q,\, \bx \neq \bx_R.
\end{align}
By the dominated convergence theorem, formula \rf{uiRepQ} is then obtained by letting $R\to\infty$ in \rf{uiREqn} and \rf{uiREqn2}, since, for fixed $\bx $, $C_RG_k(\bx ,\bx_R)$ tends to $u^i(\bx )+ (u^i)^*(\bx ) + (u^i)'(\bx )+ (u^i)^*{}'(\bx )$ as $R\to\infty$. %
The result for $\alpha=\pi$ and $3\pi/2$ follows by taking the limits $\alpha\to \pi$ and $\alpha\to 3\pi/2$ in \rf{uiRepQ}.

(iii) This is a trivial consequence of (i) and (ii) and the fact that $u=0$ on $\Gamma$.
\end{proof}

As a consequence of Lemma \ref{uRepThmQ}(iii) we have that
\begin{align}
\label{dudnRepIntegral}
\pdone{u}{\bn}(\bx ) = \pdone{\Psi_1}{\bn}(\bx )+\int_{\gamma\cup \gamma'}\frac{\partial^2 G_k(\bx ,\by )}{\partial \bn(\bx )\partial \bn(\by )}u(\by )\,\rd s(\by ), \qquad \bx \in \Gamma_{\rm nc}.
\end{align}
Theorem~\ref{Gamma2Thm} follows from a careful analysis of the integral in~\rf{dudnRepIntegral}. The terms
$v^+(L_{\rm nc}+s)\re^{\ri ks}$ and $v^-(L_{\rm nc}-s) \re^{-\ri ks}$ in the representation~\rf{dudnGamma2Rep} arise from the integral over $\gamma$. Indeed, noting that
\begin{align*}
\pdtwomix{G_k(\bx ,\by ) }{\bn(\bx )}{\bn(\by )}&= 2\pdtwomix{\Phi_k(\bx ,\by )}{\bn(\bx )}{\bn(\by )} - 2\pdtwomix{\Phi_k(\bx ,\by')}{\bn(\bx )}{\bn(\by )}, \quad \bx \in\Gamma_{\rm nc}, \,\,\by \in\gamma,
\end{align*}
where $\by':=(-y_1,y_2)$, we find that, for $ \bx \in\Gamma_{\rm nc}$,
\begin{align}
\label{gamma2Decomp}
\int_{\gamma}\frac{\partial^2 G_k(\bx ,\by )}{\partial \bn(\bx )\partial \bn(\by )}u(\by )\,\rd s(\by ) =
2\int_{\gamma}\frac{\partial^2 \Phi_k(\bx ,\by )}{\partial \bn(\bx )\partial \bn(\by )}u(\by )\,\rd s(\by ) -
2\int_{\tilde{\gamma}}\frac{\partial^2 \Phi_k(\bx ,\by )}{\partial \bn(\bx )\partial \bn(\by )}u(\by')\,\rd s(\by ),
\end{align}
with $\tilde{\gamma}:=\left\{(x_1,-L_{\rm nc}')\,:\,x_1>L_{\rm nc}\right\}$. This expression is very similar to that encountered in the derivation of the regularity results on a convex side. Indeed, arguing almost exactly as in \cite[\S3]{HeLaMe:11} (and see also \cite[\S3]{Convex}), it can be shown from \rf{gamma2Decomp} that
\begin{align*}
\label{}
\int_{\gamma}\frac{\partial^2 G_k(\bx ,\by )}{\partial \bn(\bx )\partial \bn(\by )}u(\by )\,\rd s(\by ) = v^-(L_{\rm nc}-s)\re^{-\ri ks} + v^+(L_{\rm nc}+s)\re^{\ri ks}, \quad \bx(s) \in\Gamma_{\rm nc},
\end{align*}
where $v^\pm(s)$ are analytic in $\real{s}>0$, where they satisfy the bounds (\ref{vjpmBounds}) with $\delta^\pm = 1-\pi/\omega$. This is the assertion in paragraph~(ii) of Theorem \ref{Gamma2Thm}.

We now consider the integral over $\gamma'$ in \rf{dudnRepIntegral}. Noting that
\begin{align*}
\pdtwomix{G_k(\bx ,\by ) }{\bn(\bx )}{\bn(\by )}&= -4\pdtwomix{\Phi_k(\bx ,\by )}{x_2}{y_1}, \qquad \bx \in\Gamma_{\rm nc},\,\, \by \in\gamma',
\end{align*}
and using the decomposition $u=u^i+u^s$, we have, for $ \bx \in\Gamma_{\rm nc}$, that
\begin{align}
\label{gamma1Int}
\int_{\gamma'}\frac{\partial^2 G_k(\bx ,\by )}{\partial \bn(\bx )\partial \bn(\by )}u(\by )\,\rd s(\by ) =-4\int_{\gamma'}\pdtwomix{\Phi_k(\bx ,\by )}{x_2}{y_1}u^i(\by )\,\rd s(\by ) -4\int_{\gamma'}\pdtwomix{\Phi_k(\bx ,\by )}{x_2}{y_1}u^s(\by )\,\rd s(\by ).
\end{align}
The assertions in paragraphs (i) and (iii) of Theorem \ref{Gamma2Thm} then follow from  \rf{dudnRepIntegral}, \rf{gamma1Int} and Lemmas~\ref{v1tildeThm} and~\ref{v2tildeThm} below.

\begin{lem}
\label{v1tildeThm}
For $\bx =(-s,-L_{\rm nc}')\in\Gamma_{\rm nc}$,
\begin{align*}
-4\int_{\gamma'}\pdtwomix{\Phi_k(\bx ,\by )}{x_2}{y_1}u^i(\by )\,\rd s(\by ) = \Psi(\bx ) - \pdone{\Psi_1}{\bn}(\bx ) + \re^{\ri kr} \sW(s),
\end{align*}
where $\sW(s)$ is analytic in $D_\eps$, with $\eps$ given by \eqref{eqn:eps_defn}; further, for every $k_0>0$,
\begin{align*}
\left|\sW(s)\right|&\leq Ck^{1/2}, \qquad s\in D_\eps,\, k\geq k_0,
\end{align*}
 where $C>0$ depends only on $\Omega$ and $k_0$.
\end{lem}

\begin{lem}
\label{v2tildeThm}
If Assumption \ref{ass:1} holds, then, for $\bx =(-s,-L_{\rm nc}')\in\Gamma_{\rm nc}$,
\begin{align*}
-4\int_{\gamma'}\pdtwomix{\Phi_k(\bx ,\by )}{x_2}{y_1}u^s(\by )\,\rd s(\by ) = \re^{\ri kr} \tW(s),
\end{align*}
where $\tW(s)$ is analytic in $D_\eps$, with $\eps$ given by \eqref{eqn:eps_defn}; further,
\begin{align}
\label{VEst}
\left|\tW(s)\right|&\leq CC_1k\log^{1/2}(2+k), \qquad s\in D_\eps ,\, k\geq k_1,
\end{align}
where $k_1$ and $C_1$ are as in Assumption \ref{ass:1} and $C>0$ depends only on $\Omega$ and $k_1$.
\end{lem}

We begin by proving Lemma~\ref{v2tildeThm}.  Some of the intermediate results derived in this proof will be used again in the proof of Lemma~\ref{v1tildeThm}.

\subsection{Proof of Lemma \ref{v2tildeThm}}
\label{subsec:v2tildeThm}
For $\bx =(-s,-L_{\rm nc}')\in\Gamma_{\rm nc}$ we have $r=r(s)=\sqrt{s^2+L_{\rm nc}'^2}$. Thus, to prove Lemma~\ref{v2tildeThm} we have to show that
\begin{align*}
 \tW(s):= -4\exp\left(-\ri k\sqrt{s^2+L_{\rm nc}'^2}\right)\, \int_{\gamma'}\pdtwomix{\Phi_k(\bx ,\by )}{x_2}{y_1}u^s(\by )\,\rd s(\by )
\end{align*}
is analytic in $D_\eps$, satisfying the bound \eqref{VEst}. Substituting for $u^s$ using \rf{RepThm}, and switching the order of integration, justified by Fubini's theorem, gives
\begin{align}
\tW(s)&= \int_{\Gamma}K(s,\bz )\frac{\partial u}{\partial \bn}(\bz )\,\rd s(\bz ), \label{eqn:v2}
\end{align}
where, for $s\in \mathbb{R}$ and $\bz\in \Gamma$,
\begin{align}
K(s,\bz ):=4\exp\left(-\ri k\sqrt{s^2+L_{\rm nc}'^2}\right)\,\int_0^\infty\frac{\partial^2\Phi_k\left((-s,-L_{\rm nc}'),(0,y_2)\right)}{\partial x_2\partial y_1}\Phi_k((0,y_2),\bz)\,\rd y_2, \label{eqn:Kxz}
\end{align}
and, by the recurrence and differentiation formulae for Hankel functions \cite[\S10.6]{DLMF},
\begin{align*}
\Phi_k((0,y_2),\bz)&=\frac{\ri }{4}H_0^{(1)}\left(k\sqrt{z_1^2+(y_2-z_2)^2}\right),\\
\frac{\partial^2\Phi_k((-s,-L_{\rm nc}'),(0,y_2))}{\partial x_2\partial y_1}&=-\frac{\ri k^2s(L_{\rm nc}'+y_2)}{4\left(s^2+(L_{\rm nc}'+y_2)^2\right)}H_2^{(1)}\left(k\sqrt{s^2+(L_{\rm nc}'+y_2)^2}\right).
\end{align*}
We recall that $H_n^{(1)}(z)$ is analytic in $|z|>0$, $|\arg(z)| < \pi$. To derive bounds on $K(s,\bz)$ we need bounds on $H^*_n(z) := \re^{-\ri z} H_n^{(1)}(z)$. From \cite[\S10.2(ii), \S10.17.5]{DLMF} it follows that, for some constant $C>0$,
\begin{align}
\label{H0bound}
|H^*_0(z)| \leq  C|z|^{-1/2}, \quad |z|>0, \; |\arg(z)| \leq \pi/2
\end{align}
and that, for every $c>0$ there exists $C>0$ such that
\begin{align}
\label{H2bound}
|H^*_2(z)| \leq  C|z|^{-1/2}, \quad |z|>c, \; |\arg(z)| \leq \pi/2.
\end{align}

Note that, for $s\in \mathbb{R}$ and $\bz\in \Gamma$,
\begin{align}
\label{KRep}
K(s,\bz ) = \int_0^\infty \re^{\ri k\phi(s,y_2,\bz )}S_k(s,y_2,\bz )\,\rd y_2,
\end{align}
where
$\phi(s,y_2,\bz ):=\chi{(s,L_{\rm nc}',y_2)}-\chi{(s,L_{\rm nc}',0)} + \chi{(z_1,-z_2,y_2)}$
and
\begin{align*}
S_k(s,y_2,\bz )&:=-\frac{k^2s(L_{\rm nc}'+y_2)}{4(\chi(s,L_{\rm nc}',y_2))^2}\, H^*_2(k\chi(s,L_{\rm nc}',y_2))\, H^*_0(k\chi(z_1,-z_2,y_2)),
\end{align*}
with
$\chi{(a,b,c)}:=\sqrt{a^2+(b+c)^2}$.
We now state some elementary properties of $\chi$.%

\begin{lem}
\label{ImChiLem}
Let $a,b\in\mathbb{R}$ and $c\in \mathbb{C}$, and let $\chi(a,b,c):=\sqrt{a^2+(b+c)^2}$, taking the principal value square root. Then $\chi{(a,b,c)}$ is analytic in %
$\real{c}>-b$, with
\begin{align}
\label{ReChiLB}
\real{\chi{(a,b,c)}}&\geq \real{c}+b>0, &&%
\\
\label{ImChiLB}
\im{\chi{(a,b,c)}}&\geq 0, && \mbox{if } \im{c}\geq 0.
\end{align}
In particular, for $b>0$ and  $t\geq 0$,%
\begin{align}
\label{rechibound}
\real{\chi(a,b,t\re^{\ri \pi/4})}&\geq\frac{\sqrt{a^2+b^2}+t}{2},\\
\label{imchibound}
\im{\chi(a,b,t\re^{\ri \pi/4})}&\geq\frac{bt}{\sqrt{2}\sqrt{a^2+b^2}}.
\end{align}
\end{lem}

\begin{proof}
Write $c=c_r+\ri c_i$ where $c_r>-b,c_i\in\mathbb{R}$. Then $a^2+(b+c)^2 = \xi +\ri  \eta$, where
$\xi:=a^2+(b+c_r)^2 - c_i^2$, $\eta:=2c_i(b+c_r)$.
If $c_r>-b$ then $\eta\neq 0$, unless $c_i=0$, in which case $\xi>0$. So $\chi{(a,b,c)}$ is analytic in $\real{c}>-b$ with
\begin{eqnarray}
  \real{\chi(a,b,c)}  &=&  \sqrt{\frac{\xi+\left(\xi^2+\eta^2\right)^{{1/2}}}{2}}>0, \label{RechiLBPrelim}\\
  \im{\chi(a,b,c)}  &=&  \sign(c_i)\,\sqrt{\frac{-\xi+\left(\xi^2+\eta^2\right)^{{1/2}}}{2}}, \label{ImchiLBPrelim}
\end{eqnarray}
which gives \rf{ImChiLB}. Writing
$
2\real{\chi(a,b,c)}^2=2(b+c_r)^2 + \mu_1,
$
where we define
\mbox{$\mu_1 := \sqrt{\left(a^2- c_i^2+(b+c_r)^2 \right)^2 + 4c_i^2(b+c_r)^2} + a^2-c_i^2-(b+c_r)^2$},
and noting that
$
\left(a^2- c_i^2 +(b+c_r)^2 \right)^2 + 4c_i^2(b+c_r)^2 - \left(a^2-c_i^2-(b+c_r)^2 \right)^2 = 4a^2(b+c_r)^2\geq 0$,
it follows that $\mu_1 \geq 0$, and hence~\rf{ReChiLB} holds.

When $c=t\re^{\ri \pi/4}$ with $t\geq 0$, we have
$\xi=a^2+b^2 + \sqrt{2}bt\geq a^2+b^2$ and $\eta=t(\sqrt{2}b + t)\geq t^2$,
and~\rf{rechibound} follows from~\rf{RechiLBPrelim}.  Also,
by~\rf{ImchiLBPrelim}, %
\begin{align*}
  2(a^2+b^2)\,\im{\chi(a,b,t\re^{\ri \pi/4})}^2&= b^2t^2+\mu_2,
\end{align*}
where $\mu_2 :=-(b^2t^2 + (a^2+b^2)\xi) + (a^2+b^2)\sqrt{\xi^2+\eta^2}$.
One can check that
\[ (a^2+b^2)^2(\xi^2+\eta^2) - (b^2t^2 + (a^2+b^2)\xi)^2 = t^3a^2\left(2\sqrt{2}\,b(a^2+b^2) + (a^2+2b^2)t\right)\geq 0, \]
from which we deduce that $\mu_2\geq 0$, from which~\rf{imchibound} follows.
\end{proof}

In order to prove Lemma~\ref{v2tildeThm} we must consider the analytic continuation of $K(s,\bz )$ into the complex $s$-plane. But before complexifying $s$ it is helpful to modify the representation \rf{KRep} by deforming the contour of integration off the real line.
From~\rf{KRep} it follows from Cauchy's theorem that, for $s\in \mathbb{R}$ and $\bz\in \Gamma$, where $f(w):=\re^{\ri k\phi(s,w,\bz )}S_k(s,w,\bz )$ and $\gamma^*:=\{w= t\re^{\ri \pi/4}:t\geq 0\}$,
\begin{align}
\label{KRepDeform}
K(s,\bz ) = \int_{\gamma^*} f(w)\, \rd w =  \re^{\ri \pi/4} \int_0^\infty \re^{\ri k\phi(s,t\re^{\ri \pi/4},\bz )}S_k(s,t\re^{\ri \pi/4},\bz )\,\rd t.
\end{align}
This application of Cauchy's theorem is valid since, by Lemma \ref{ImChiLem}, $f(w)$
is analytic in $\real{w}> 0$; further, $\im{\phi(s,w,\bz )}\geq 0$, so that $|\re^{\ri k\phi(s,w,\bz )}|  \leq 1$, if $\real{w}> 0$ and $\im{w}\geq 0$; moreover, the bounds \eqref{ReChiLB}, \eqref{H0bound}, and \eqref{H2bound} imply that
\begin{align*}
S_k(s,w,\bz ) = \ord{|w|^{-1/2}}, \;\; \mbox{as } |w|\to 0, \quad S_k(s,w,\bz ) = \ord{|w|^{-2}}, \;\; \mbox{as } |w|\to \infty,
\end{align*}
uniformly in $\arg(w)$, for $0\leq \arg(w)\leq \pi/4$.

Having established the validity of the representation \rf{KRepDeform} for $s\in \mathbb{R}$, we now show that this same formula represents the analytic continuation of $K(s,\bz )$.%
\begin{lem}
\label{KAnalCor}
For $\bz\in \Gamma$, $K(s,\bz )$, defined by \rf{KRepDeform}, is analytic as a function of $s$ in $D_\eps $, with $\eps$ given by \eqref{eqn:eps_defn}. Further, for every $k_0>0$,
\begin{align}
\label{KEst1}
\left|K(s,\bz )\right| &\leq Ck^{{1/2}}%
\zeta(\bz ),\qquad s\in D_\eps,\; k\geq k_0, \; \bz\in \Gamma,
\end{align}
where $C>0$ depends only on $\Omega$ and $k_0$, and %
\begin{align*}
\label{}
\zeta(\bz ):=
\begin{cases}
1, & 0<k\left|\bz \right|<1,\\
(k\left|\bz \right|)^{-{1/2}}, & k\left|\bz \right|\geq1.
\end{cases}
\end{align*}
\end{lem}

The proof of Lemma \ref{KAnalCor} is based on the following two intermediate results.

\begin{lem}
\label{ImphiCor}
For $t\geq 0$ and $\bz\in \Gamma$, $\phi(s,t\re^{\ri \pi/4},\bz )$ is analytic as a function of $s$ in $D_\eps $, with $\eps$ given by \eqref{eqn:eps_defn}. Further,
\begin{align}
\label{ImphiCorRes}
\im{\phi(s,t\re^{\ri \pi/4},\bz )}\geq \frac{L_{\rm nc}'t}{2\sqrt{2}\sqrt{L_{\rm nc}'^2+L_{\rm nc}^2}} ,\qquad s\in D_\eps ,\,t\geq0, \, \bz\in \Gamma.
\end{align}
\end{lem}

\begin{proof}
Suppose $t\geq 0$ and $\bz\in \Gamma$. For $s_0\in [0,L_{\rm nc}]$,
\begin{align}
\label{Imphis0Est}
\im{\phi(s_0,t\re^{\ri \pi/4},\bz )} = \im{\chi(s_0,L_{\rm nc}',t\re^{\ri \pi/4})} + \im{\chi(z_1,-z_2,t\re^{\ri \pi/4})}\geq \frac{L_{\rm nc}'t}{\sqrt{2}\sqrt{s_0^2+L_{\rm nc}'^2}},
\end{align}
by \eqref{imchibound} and \eqref{ImChiLB}, applied to $\chi(s_0,L_{\rm nc}',t\re^{\ri \pi/4})$ and $\chi(z_1,-z_2,t\re^{\ri \pi/4})$, respectively.
We next note that,
for $s\in \mathbb{C}$,
\begin{align*}
\label{}
\phi(s,t\re^{\ri \pi/4},\bz ) = \frac{A}{B(s)} + \chi(z_1,-z_2,t\re^{\ri \pi/4}),
\end{align*}
where $A := t\re^{\ri \pi/4}(2L_{\rm nc}'+t\re^{\ri \pi/4})$ and $B(s) :=\chi(s,L_{\rm nc}',t\re^{\ri \pi/4})+\chi(s,L_{\rm nc}',0)$.
Thus, for $s_0\in [0,L_{\rm nc}]$ and $|s-s_0|< \eps$,
\begin{align}
\label{PhiDiff}
|\phi(s,t\re^{\ri \pi/4},\bz )-\phi(s_0,t\re^{\ri \pi/4},\bz )| &=  \frac{\left|A\right|\left|B(s)-B(s_0)\right|}{\left|B(s_0)\right|\left|B(s)\right|}%
\leq \frac{\left|A\right|\left|B(s)-B(s_0)\right|}{\left|B(s_0)\right|\left|\left|B(s_0)\right|-\left|B(s)-B(s_0)\right|\right|}.
\end{align}
Now $|A| \leq t(2L_{\rm nc}'+t)$,
and, by \eqref{ReChiLB},
\begin{align*}
\left|B(s_0)\right| \geq \real{B(s_0)} \geq L_{\rm nc}' + \frac{t}{\sqrt{2}} + \sqrt{s_0^2+L_{\rm nc}'^2}  .
\end{align*}
Also, $\real{\chi(s,L_{\rm nc}',t\re^{\ri \pi/4})}>0 $ for  $s\in D_\eps $, since $\eps<L_{\rm nc}'$ so that
$\real{s^2+(L_{\rm nc}'+t\re^{\ri \pi/4})^2}\geq - \im{s}^2 + L_{\rm nc}'^2 > 0$.
Thus, using \eqref{rechibound},
\begin{align}
\label{chidiff}
\left|\chi(s,L_{\rm nc}',t\re^{\ri \pi/4})-\chi(s_0,L_{\rm nc}',t\re^{\ri \pi/4})\right| &= \frac{\left|s-s_0\right|\left|s+s_0\right|}{\left|\chi(s,L_{\rm nc}',t\re^{\ri \pi/4})+\chi(s_0,L_{\rm nc}',t\re^{\ri \pi/4})\right|}\notag\\
&\leq \frac{{\eps}(2s_0+{\eps})}{\real{\chi(s_0,L_{\rm nc}',t\re^{\ri \pi/4})}}\notag\\
&\leq \frac{4{\eps}(s_0+L_{\rm nc}')}{\sqrt{s_0^2+L_{\rm nc}'^2}} \; \leq 4\sqrt{2}\eps.
\end{align}
This implies that
$\left|B(s)-B(s_0)\right|
\leq 8\sqrt{2}\,\eps$.
Inserting these bounds into \rf{PhiDiff} gives
\begin{eqnarray} \notag
|\phi(s,t\re^{\ri \pi/4},\bz )-\phi(s_0,t\re^{\ri \pi/4},\bz )| & \leq & \frac{8\sqrt{2}\,t\eps(2L_{\rm nc}'+t)}{\left(2L_{\rm nc}'+t/\sqrt{2}\right)\left(L_{\rm nc}' + \sqrt{s_0^2+L_{\rm nc}'^2}-8\sqrt{2}\eps\right)}\\
\label{phidiffs}
& \leq & \frac{L_{\rm nc}'t}{2\sqrt{2}\sqrt{s_0^2+L_{\rm nc}'^2}},
\end{eqnarray}
on using \eqref{eqn:eps_defn}.
The result \rf{ImphiCorRes} follows by combining \rf{Imphis0Est} and \rf{phidiffs}.
\end{proof}

\begin{lem}
\label{SkLem}
For $t\geq 0$ and $\bz\in \Gamma$, $S_k(s,t\re^{\ri \pi/4},\bz )$ is analytic as a function of $s$ in $D_\eps $, with $\eps$ given by \eqref{eqn:eps_defn}. Further, for every $k_0>0$,
\begin{align}
\label{SEst}
|S_k(s,t\re^{\ri \pi/4},\bz )|\leq Ck%
(L_{\rm nc}'+t)(|\bz|+t)^{-1/2},\quad s\in D_\eps,
\end{align}
for $t\geq 0$, $\bz\in \Gamma$, and $k\geq k_0$, where $C>0$ depends only on $\Omega$ and $k_0$.
\end{lem}

\begin{proof}
Suppose $t\geq 0$ and $\bz\in \Gamma$. By \eqref{ReChiLB} and \eqref{rechibound} we have, for $s_0\in[0,L_{\rm nc}]$,
\begin{align}
\label{reChiLB1}
\real{\chi(s_0,L_{\rm nc}',t\re^{\ri \pi/4})}\geq L_{\rm nc}', \quad
\real{\chi(z_1,-z_2,t\re^{\ri \pi/4})}\geq \frac{|\bz|+t}{2}.
\end{align}
Combining \rf{reChiLB1} with \rf{chidiff} and recalling \eqref{eqn:eps_defn} gives
\begin{align}
\label{reChiLB3}
\real{\chi(s,L_{\rm nc}',t\re^{\ri \pi/4})}&\geq \frac{7L_{\rm nc}'}{8}, \qquad s\in D_\eps.
\end{align}
Thus $S_k(s,t\re^{\ri \pi/4},\bz )$ is analytic in $D_\eps $, and applying \eqref{H0bound} and \eqref{H2bound} gives \eqref{SEst}.
\end{proof}

We are now ready to prove Lemma \ref{KAnalCor}.

\begin{proof}[Proof of Lemma \ref{KAnalCor}]
The analyticity of $K(s,\bz )$ follows immediately from that of $\phi(s,t\re^{\ri \pi/4},\bz )$ and $S_k(s,t\re^{\ri \pi/4},\bz )$, and the fact that the integral \rf{KRepDeform} converges uniformly for $s\in D_\eps$ in view of the bounds in  Lemmas~\ref{ImphiCor} and~\ref{SkLem} (see, e.g., \cite[\S1.88, \S4.4]{Titchmarsh}). Further, by Lemmas~\ref{ImphiCor} and~\ref{SkLem} we have, for $s\in D_\eps$, that
\begin{align}
\label{KEstBoth}
\left|K(s,\bz )\right| %
&\leq CkL_{\rm nc}'^{-5/2}(L_{\rm nc}'+L_{\rm nc})\int_0^\infty \frac{L_{\rm nc}'+t}{(|\bz|+t)^{1/2}}\exp\left[-\frac{kL_{\rm nc}'t}{2\sqrt{2}\sqrt{L_{\rm nc}'^2+L_{\rm nc}^2}}\right]\,\rd t.
\end{align}
The integral in \rf{KEstBoth} is bounded above by
\begin{align*}
\int_0^\infty \frac{L_{\rm nc}'+t}{t^{1/2}}\exp\left[-\frac{kL_{\rm nc}'t}{2\sqrt{2}\sqrt{L_{\rm nc}'^2+L_{\rm nc}^2}}\right]\rd t \leq Ck^{-1/2},
\end{align*}
for $k\geq k_0$, for some $C>0$ depending only on $L_{\rm nc}'$, $L_{\rm nc}$ and $k_0$. For $k|\bz|> 1$ a sharper upper bound is
\begin{align*}
|\bz|^{-1/2}\!\!\int_0^\infty\!\! (L_{\rm nc}'+t)\exp\left[-\frac{kL_{\rm nc}'t}{2\sqrt{2}\sqrt{L_{\rm nc}'^2+L_{\rm nc}^2}}\right]\rd t \leq C k^{-1}|\bz|^{-1/2}.
\end{align*}
Combining these two results  gives \rf{KEst1}.
\end{proof}

We can now complete the proof of Lemma \ref{v2tildeThm}.  Applying the Cauchy-Schwarz inequality to \rf{eqn:v2}, and recalling Assumption \ref{ass:1}, we estimate
\begin{align}
\label{VEstKdudn}
|\tW(s)|&\leq\left\|K(s,\cdot)\right\|_{L^2\left(\Gamma\right)}\left\|\frac{\partial u}{\partial \bn}\right\|_{L^2\left(\Gamma\right)}
\leq C_1 k \left\|K(s,\cdot)\right\|_{L^2\left(\Gamma\right)}, \qquad k\geq k_1.
\end{align}
It therefore remains to bound $\normt{K(s,\cdot)}{L^2(\Gamma)}$. But, by \rf{KEst1}, this just requires a bound on $\normt{\zeta}{L^2(\Gamma)}$. Let $\Gamma^*$ denote any one of the sides of $\Gamma$. Then it is clear that, if $\Gamma^*$ is not one of the sides of $\Gamma$ adjacent to
$\bR$, then $\int_{\Gamma^*} (\zeta(\bz ))^2\,\rd s \leq C k^{-1}$, for $k\geq k_0$. On the other hand, if $\Gamma^*$ has length $L^*$ and is adjacent to $\bR$, then
$$
\int_{\Gamma^*} (\zeta(\bz ))^2\,\rd s \leq C \int_0^{1/k} \,\rd s + C k^{-1} \int_{1/k}^{L^*} t^{-1} \, \rd t \leq C k^{-1} \log(2+k).
$$
Thus
$\normt{\zeta}{L^2(\Gamma)} \leq C k^{-1/2}\log^{1/2}(2+k)$,
so that, by \rf{KEst1},
\begin{align}
\label{KEstFinal}
\norm{K(s,\cdot)}{L^2(\Gamma)} \leq C \log^{1/2}(2+k).
\end{align}
Finally, combining \rf{VEstKdudn} and \rf{KEstFinal} proves \rf{VEst}, completing the proof of Lemma \ref{v2tildeThm}.

\subsection{Proof of Lemma \ref{v1tildeThm}}
\label{subsec:v1tildeThm}

Suppose first that $\alpha \not\in(\pi/2,3\pi/2)$, in which case both $\Psi$ and $\Psi_1$ are zero. Then, since $u^i(\by )=\re^{\ri ky_2\cos{\alpha}}$ for $\by \in \gamma'$, we have
\begin{align*}
-4\int_{\gamma'}\pdtwomix{\Phi_k(\bx ,\by )}{x_2}{y_1}u^i(\by )\,\rd s(\by ) =\re^{\ri kr} \sW(s),
\end{align*}
where
\begin{align}
\label{v1tildeRep}
 \sW(s)&=\int_0^\infty \re^{\ri k\varpi(s,y_2,\alpha)}T_k(s,y_2)\,\rd y_2,\\
\varpi(s,y_2,\alpha)&:=\chi(s,L_{\rm nc}',y_2)-\chi(s,L_{\rm nc}',0) + y_2\cos{\alpha},
\notag\\ T_k(s,y_2)&:= \frac{\ri k^2s(L_{\rm nc}'+y_2)}{[\chi(s,L_{\rm nc}',y_2)]^2}\, H^*_2(k \chi(s,L_{\rm nc}',y_2)).\notag
\end{align}
We deform the contour of integration in \rf{v1tildeRep} to $\gamma^*=\{t\re^{\ri \pi/4}:t\geq 0\}$, as in \eqref{KRepDeform}. Then, arguing as in the proofs of Lemmas \ref{ImphiCor} and \ref{SkLem}, we find that $\varpi(s,t\re^{\ri \pi/4},\alpha)$ and $T(s,t\re^{\ri \pi/4})$ are analytic as functions of $s$ in $D_\eps$, with $\eps$ given by \eqref{eqn:eps_defn}. Further, since $\cos{\alpha}\geq 0$, it follows from the calculations in the proof of Lemma \ref{ImphiCor} that
\begin{align}
\label{imvarphiEst}
\im{\varpi(s,t\re^{\ri \pi/4},\alpha)}&\geq\frac{L_{\rm nc}'t}{2\sqrt{2}\sqrt{L_{\rm nc}'^2+L_{\rm nc}^2}}, \quad s\in D_\eps, \; t \geq 0,
\end{align}
while, from \eqref{reChiLB3} and \eqref{H2bound}, it follows that, for all $k_0>0$,
$$
|T_k(s,t\re^{\ri \pi/4})|\leq Ck^{3/2}(L_{\rm nc}'+t), \nonumber
$$
if $s\in D_\eps $, $t\geq0$, and $k\geq k_0$, where $C>0$ depends only on $\Omega$ and $k_0$. Thus
\begin{align}
\label{v1tildeEst}
|\sW(s)|&\leq \int_0^\infty \re^{-k\im{\varpi(s,t\re^{\ri \pi/4},\alpha)}}|T_k(s,t\re^{\ri \pi/4})|\,\rd t \leq  Ck^{{1/2}}.
\end{align}

If $\alpha \in(\pi/2,3\pi/2)$, however, \rf{imvarphiEst} no longer holds (since $\cos{\alpha}< 0$). In this case we write $u^i=u^d + (u^i-u^d)$, and note that, by Lemma \ref{def:udDef}, for  $\by\in\gamma'$,
\begin{align*}
\label{}
u^i(\by )-u^d(\by )
& = 2\re^{\ri ky_2}h\left(\sqrt{2ky_2}\sin\left(\alpha/2\right)\right),
\end{align*}
where $h(w) := \re^{-\ri w^2}{\rm Fr}(w)$.
The function $h(w)$ is entire, and is uniformly bounded in the sector $\arg[w]\in[-\pi/2,\pi]$ (this follows from the asymptotic behaviour of the complementary error function \cite[\S7.12(i)]{DLMF}, and that $h(w) = \frac{1}{2}\re^{-\ri w^2}{\rm{erfc}}(\re^{-\ri\pi/4} w)$).  Hence a similar argument to that leading to \eqref{v1tildeEst}, but applied to $u^i-u^d$ rather than to $u^i$, shows that
\begin{align}
\label{uiminusudRep}
-4\int_{\gamma'}\pdtwomix{\Phi_k(\bx ,\by )}{x_2}{y_1}(u^i(\by )-u^d(\by ))\,\rd s(\by ) =\re^{\ri kr} \sW(s),
\end{align}
with $\sW(s)$ analytic in $D_\eps $, satisfying \rf{v1tildeEst} with $\varpi(s,t\re^{\ri \pi/4},\alpha)$ replaced by $\varpi(s,t\re^{\ri \pi/4},0)$ and $T_k(s,t\re^{\ri \pi/4})$ replaced by $2T_k(s,t\re^{\ri \pi/4})h\left(\sqrt{2kt}\sin\left(\alpha/2\right)\re^{\ri \pi/8}\right)$; in particular, $|\sW(s)|\leq C k^{1/2}$ for $k\geq k_0$ and $s\in D_\eps$, where $C$ depends only on $k_0$ and $\Omega$.
The next result deals with the remaining term, $4\int_{\gamma'}(\partial^2\Phi_k(\bx ,\by )/\partial x_2 \partial y_1)u^d(\by )\,\rd s(\by )$.
\begin{prop}
\label{udLem}
\begin{align}
\label{udRep}
u^d(\bx )&=\Psi_2(\bx )-2\int_{\gamma'}\frac{\partial\Phi_k(\bx ,\by )}{\partial y_1}u^d(\by )\,\rd s(\by ), \qquad x_1<0,
\end{align}
where
\begin{align*}
\label{}
\Psi_2(\bx ):=
\begin{cases}
0, & 0\leq\alpha\leq\pi,\\
u^i(\bx )+(u^i)'(\bx ) %
=-2\ri \re^{\ri kx_2\cos{\alpha}}\sin{(kx_1\sin{\alpha})}, & \pi< \alpha < 2\pi.
\end{cases}
\end{align*}
\end{prop}
\begin{proof}Suppose first that $0\leq \alpha \leq \pi/2$. Let $k$ temporarily have a positive imaginary part. Then it is straightforward to show that $u^d$ is uniformly bounded in the half-plane $x_1<0$, and hence \rf{udRep} follows from \cite[Theorem 3.1]{HalfPlaneRep} for $\im{k}>0$, and for $k>0$ by taking the limit as $\im{k}\to 0$ \cite[Theorem 3.2]{HalfPlaneRep}.

For $\pi/2< \alpha \leq \pi$ this argument fails, since if $k$ has a positive imaginary part, $|u^d(\bx)|$ grows exponentially as $x_2\to \infty$ for any fixed $x_1<0$. However, the argument does provide a proof that \rf{udRep} holds (with $\Psi_2=0$) with $u^d$ replaced by $u^d-u^i$, which is uniformly bounded in $x_1<0$. Also, the analysis of \cite[p.~193]{HalfPlaneRep} shows that \rf{udRep} holds (with $\Psi_2=0$) with $u^d$ replaced by the plane wave $u^i$. Adding together these two results proves \rf{udRep}.

The above two paragraphs prove \rf{udRep} for the case $0\leq \alpha \leq \pi$ when $\Psi_2=0$. For $\pi<\alpha<2\pi$, the above arguments allow us to prove \rf{udRep} (with $\Psi_2=0$) with $u^d$ replaced by $\tilde{u}^d$, the solution to the knife edge scattering problem of Figure \ref{fig-1}(b) corresponding to the incident direction $2\pi-\alpha\in(0,\pi)$. Since $\tilde{u}^d =u^d - u^i - (u^i)'$,
 and in particular $\tilde{u}^d=u^d$ on $\gamma'$, this implies that
\begin{align*}
u^d(\bx )-u^i(\bx )-(u^i)'(\bx )=-2\int_{\gamma'}\frac{\partial\Phi_k(\bx ,\by )}{\partial y_1}u^d(\by )\,\rd s(\by ), \quad x_1<0, %
\end{align*}
i.e.\ that \rf{udRep} holds for $\pi<\alpha<2\pi$.
\end{proof}

For  $\alpha \in(\pi/2,3\pi/2)$  we have $2\pdonetext{\Psi_2}{\bn} = \pdonetext{\Psi_1}{\bn}$ on $\Gamma_{\rm nc}$, and so it follows from Proposition \ref{udLem} that, for $\bx\in \Gamma_{\rm nc}$,
\begin{align*}
-4\int_{\gamma'}\pdtwomix{\Phi_k(\bx ,\by )}{x_2}{y_1} u^d(\by )\,\rd s(\by ) = 2\pdone{u^d}{\bn}(\bx )- \pdone{\Psi_1}{\bn}(\bx ). %
\end{align*}
Recalling \rf{uiminusudRep}, this completes the proof of Lemma \ref{v1tildeThm},
and hence of Theorem \ref{Gamma2Thm}.

\section{hp Approximation Space and Approximation Results}
\label{sec:ApproxSpace}

We now design an $hp$ approximation space for the numerical solution of \rf{eqn:op061210}, based on the regularity results provided by Theorems \ref{thm:1} and \ref{Gamma2Thm}.
Rather than approximating $\pdonetext{u}{\bn}$ itself (as in conventional methods),
we will approximate
\begin{equation}
  \varphi(\bx) := \frac{1}{k}\left(\frac{\partial u}{\partial\bn}(\bx)-\Psi(\bx)\right), \quad \bx\in\Gamma,
  \label{eqn:varphi}
\end{equation}
which represents the difference between $\partial u/\partial\bn$ and the known leading order high frequency behaviour~$\Psi$ (cf.\ Remarks~\ref{rem:PO} and \ref{rem:LO}), scaled by $1/k$ so that $\varphi$ is nondimensional. This leading order behaviour is as defined in Theorems \ref{thm:1} and \ref{Gamma2Thm}. Thus, on a convex side, $\Psi:=2\pdonetext{u^i}{\bn}$ if the side is illuminated and $\Psi:=0$ otherwise. On a nonconvex side, $\Psi:=2u^i(\bR)\pdonetext{u^d}{\bn}$ if $\pi/2\leq\alpha\leq3\pi/2,$ and $\Psi:=0$ otherwise; here $u^d$ is defined as in Lemma \ref{def:udDef} in terms of the local variables $r,\theta,\alpha$ of Figure \ref{fig-1}(a) (as remarked previously, any nonconvex side can be transformed to the configuration in Figure \ref{fig-1}(a) by a suitable rotation and reflection of $\Omega$). The factor $u^i(\bR)$ is a phase shift arising because the origin of the global coordinates $\bx$ may not be located at the point $\bR$, as was assumed in Theorem \ref{Gamma2Thm}.

Furthermore, instead of approximating $\varphi$ directly by conventional piecewise polynomials, on each side of the polygon we use the appropriate representation~\rf{Decomp} or \rf{dudnGamma2Rep}, with the non-oscillatory coefficients $v^\pm$ and $v$ replaced by piecewise polynomial approximations supported on overlapping meshes, graded towards corner singularities (where these are present). Before detailing the approximation space, we introduce some notation.

\begin{defn}%
\label{GeomDefMain}
Given $-\infty<a<b<\infty$ and an integer $p\geq 0$, let $\mathcal{P}_p(a,b)$ denote the space of polynomials on $(a,b)$ of degree $\leq p$.
Given $A>0$ %
and an integer $n\geq 1$ we denote by $\mathcal{G}_n(0,A)=\{x_0,x_1,\ldots,x_n\}$ the geometric mesh on $[0,A]$ with $n$ layers, whose meshpoints $x_i$ are defined by%
\[
  x_0:=0, \qquad x_i := \sigma^{n-i}A, \quad i=1,2,\ldots,n,
\]
where $0<\sigma<1$ is a fixed grading parameter.
We denote by $\mathcal{P}_{p,n}(0,A)$ the space of piecewise polynomials on $\mathcal{G}_n(0,A)$ with degree $\leq p$, i.e.
\begin{align*}
\label{}
\mathcal{P}_{p,n}(0,A):=&\left\{\rho:[0,A]\to \C \,: \,\rho|_{(x_{i-1},x_i)} \in \mathcal{P}_{p}(x_{i-1},x_i),\, i=1,\ldots,n \right \}.
\end{align*}
\end{defn}
A smaller $\sigma$ represents a more severe grading. While $\sigma= (\sqrt{2}-1)^2\approx0.17$ is in some sense an optimal choice, e.g., \cite[p.96]{Sc:98},
it is common practice to slightly ``overrefine'' by taking $\sigma = 0.15$; we use this value in the computations
in~\S\ref{sec:num}.

For simplicity we use the same polynomial degree $p$ and the same number of layers $n$ in each graded mesh in our approximation space. We also assume that
\begin{align}
  \label{npAssump}
  n\geq cp,
\end{align}
for some fixed constant $c>0$.

On a convex side $\Gamma_{\rm c}$, we recall from \rf{Decomp} that
\begin{align*}
\label{}
\varphi(\bx(s)) =\frac{1}{k}\left( v^+(s)\re^{\ri ks} + v^-(L_{\rm c}-s)\re^{-\ri ks}\right), \qquad s\in [0,L_{\rm c}],
\end{align*}
where the coefficients $v^+(s)$ and $v^-(L_{\rm c}-s)$ are singular at $s=0$ and $s=L_{\rm c}$, respectively. To approximate $\varphi$ on $\Gamma_{\rm c}$ we approximate $v^+(s)\approx \rho^+(s)$ and $v^-(L_{\rm c}-s)\approx \rho^-(L_{\rm c}-s)$, for some $\rho_\pm\in\mathcal{P}_{p,n}(0,L_{\rm c})$.

On a nonconvex side $\Gamma_{\rm c}$, we recall from \rf{dudnGamma2Rep} that
\begin{align*}
\label{}
\varphi(\bx(s)) =\frac{1}{k}\left( v^+(L_{\rm nc}+s)\re^{\ri ks} + v^-(L_{\rm nc}-s)\re^{-\ri ks} + v(s)\re^{\ri kr}\right), \qquad s\in [0,L_{\rm nc}].
\end{align*}
The coefficient $v^-(L_{\rm c}-s)$ is singular at $s=L_{\rm nc}$, but the coefficients $v^+(L_{\rm nc}+s)$ and $ v(s)$ are both analytic in a neighbourhood of $[0,L_{\rm nc}]$ and can be approximated by single polynomials supported on the whole side. To approximate $\varphi$ on $\Gamma_{\rm nc}$ we therefore approximate $v^-(L_{\rm nc}-s)\approx \rho^-(L_{\rm nc}-s)$ for some $\rho^-\in\mathcal{P}_{p,n}(0,L_{\rm nc})$, and approximate $v^+(s)\approx \rho^+(s)$ and $v(s)\approx \rho(s)$ for some $\rho^+,\rho\in\mathcal{P}_{p}(0,L_{\rm nc})$.  An illustration of the resulting meshes is given in Figure~\ref{fig:meshes}.

\begin{figure}[ht]
  \begin{center}
\subfigure[Convex side $\Gamma_{\rm c}$]{
  	\begin{tikzpicture}[line cap=round,line join=round,>=triangle 45,x=1.0cm,y=1.0cm, scale=0.95]
	\fill[line width=0pt,color=cqcqcq,fill=cqcqcq] (-1,-1)--(0,0)--(3,-0) -- (4,-1)  -- cycle;
	\draw [line width=1.2pt] (-1,-1)--(0,0)--(3,-0) -- (4,-1);
	\draw (0,0.4)--(3,0.4);
	\draw (0,0.3)--(0,0.5);
	\draw (2,0.3)--(2,0.5);
	\draw (2.7,0.3)--(2.7,0.5);
	\draw (2.9,0.3)--(2.9,0.5);
	\draw (3,0.3)--(3,0.5);
	\draw (-0.6,0.4) node {$v^+(s)$};
\begin{scope}[xscale=-1,yscale=1,xshift=-3cm,yshift=0.6cm]
	\draw (0,0.4)--(3,0.4);
	\draw (0,0.3)--(0,0.5);
	\draw (2,0.3)--(2,0.5);
	\draw (2.7,0.3)--(2.7,0.5);
	\draw (2.9,0.3)--(2.9,0.5);
	\draw (3,0.3)--(3,0.5);
\end{scope}
	\draw (-1.04,1) node {$v^-(L_{\rm c}-s)$};
	\draw [<-] (0,-0.2) -- (3,-0.2);
	\draw (1.5,-0.4) node {$s$};
\end{tikzpicture}
}
\subfigure[Nonconvex side $\Gamma_{\rm nc}$]{
	\begin{tikzpicture}[line cap=round,line join=round,>=triangle 45,x=1.0cm,y=1.0cm, scale=0.95]
		\fill[line width=0pt,color=cqcqcq,fill=cqcqcq]  (6.25,1.5)--(5,2.5)--  (5,0)-- (2,0)-- (2.54,-1) --(6.25,-1)  -- cycle;
	\draw [line width=1.2pt] (6.25,1.5)--(5,2.5)--  (5,0)-- (2,0)-- (2.54,-1);
	\draw [<-] (2.3,-0.2) -- (4.9,-0.2);
	\draw (3.7,-0.4) node {$s$};
\begin{scope}[xscale=-1,yscale=1,xshift=-5cm,yshift=0.6cm]
	\draw (0,0.4)--(3,0.4);
	\draw (0,0.3)--(0,0.5);
	\draw (2,0.3)--(2,0.5);
	\draw (2.7,0.3)--(2.7,0.5);
	\draw (2.9,0.3)--(2.9,0.5);
	\draw (3,0.3)--(3,0.5);
\end{scope}
	\draw (0.9,0.4) node {$v^+(L_{\rm nc}+s)$};
\begin{scope}[xscale=-1,yscale=1,xshift=-5cm,yshift=0cm]
	\draw (0,0.4)--(3,0.4);
	\draw (0,0.3)--(0,0.5);
	\draw (3,0.3)--(3,0.5);
\end{scope}
	\draw (0.9,1) node {$v^-(L_{\rm nc}-s)$};
\begin{scope}[xscale=-1,yscale=1,xshift=-5cm,yshift=1.2cm]
	\draw (0,0.4)--(3,0.4);
	\draw (0,0.3)--(0,0.5);
	\draw (3,0.3)--(3,0.5);
\end{scope}
	\draw (1.54,1.6) node {$v(s)$};
\end{tikzpicture}
}
\end{center}
  \caption{Illustration of the overlapping meshes.}%
  \label{fig:meshes} %
\end{figure}
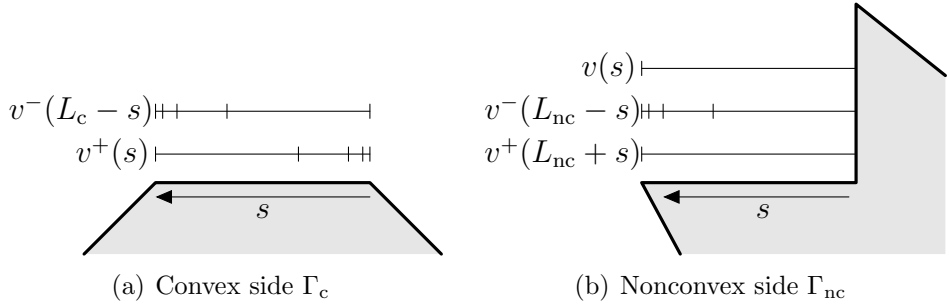

The above construction amounts to constraining the approximation to $\varphi$ to lie in a particular  finite-dimensional approximation space $V_{N,k}\subset L^2(\Gamma)$, of dimension $N$ (the total number of degrees of freedom), given by
\begin{align}
\label{NDof}
N = (p+1)(2n n_{\rm c}+ (n+2) n_{\rm nc}),
\end{align}
where $n_{\rm c}$ and $n_{\rm nc}$ denote the number of convex and nonconvex sides, respectively.

For $a <b$ and $r>b-a$, let
\begin{align}
\label{EllipseDef}
\mathcal{E}_{a,b,r}:=\left\{ w\in\mathbb{C}: |w-a|+|w-b|<r \right\},
\end{align}
the interior of an ellipse with foci $\{a,b\}$. Our best approximation estimates are based on the following standard result, which follows from \cite[Theorem 2.1.1]{Stenger}.

\begin{lem}
\label{EllipseLem}
If the function $g$ is analytic and bounded in $\mathcal{E}_{a,b,r}$, for some $a,b,r\in \mathbb{R}$ with $a <b$ and $r>b-a$,
then
\begin{align*}
\label{}
\inf_{v'\in \mathcal{P}_p(a,b)}\norm{g-v'}{L^\infty(a,b)}\leq\frac{2}{\rho-1}\rho^{-p}\norm{g}{L^\infty(\mathcal{E}_{a,b,r})},
\end{align*}
where $\rho =(r+\sqrt{r^2-(b-a)^2})/(b-a)>1$.
\end{lem}

Lemma \ref{EllipseLem} implies the following best approximation results for the two nonsingular terms in the representation on a nonconvex side. %
\begin{thm}
\label{vjPIThm}
Suppose that Assumption \ref{ass:1} holds. Then, for every $k_0>0$,
for the approximation of $v^+(L_{\rm nc}+s)$ on a nonconvex side $\Gamma_{\rm nc}$ we have
\[%
\inf_{v'\in \mathcal{P}_{p}(0,L_{\rm nc})}\normt{v^+(L_{\rm nc}+\cdot)-v'}{L^2(0,L_{\rm nc})}\leq C\uM k^{1/2}\re^{-p\tau}, \quad k\geq k_0,\nonumber
\]%
where $\tau=\log{(2+\sqrt{3})}$ and $C>0$ depends only on $\Omega$ and $k_0$.
\end{thm}

\begin{proof} By Theorem~\ref{Gamma2Thm}, $v_+(s)$ is analytic in $\real{s}>0$ where it satisfies the bound \rf{vjpmBounds}. Thus $g(s):=v^+(L_{\rm nc}+s)$ is analytic in $\real{s}>-L_{\rm nc}$, in particular analytic and bounded in $\real{s}>-L_{\rm nc}/2$, which contains the ellipse $\mathcal{E}_{0,L_{\rm nc},r}$ with $r=2L_{\rm nc}$. Thus combining Lemma \ref{EllipseLem} with \rf{vjpmBounds} gives
\begin{align*}
\inf_{v'\in \mathcal{P}_{p}(0,L_{\rm nc})}\normt{v^+(L_{\rm nc}+\cdot)-v'}{L^\infty(0,L_{\rm nc})}
& \leq C\uM k^{1/2}\rho^{-p},
\end{align*}
with $\rho = 2+\sqrt{3}$, from which the result follows. %
\end{proof}

\begin{thm}
\label{vjtildeThm}
Suppose that Assumption \ref{ass:1} holds.
Then, for the approximation of $v(s)$ on a nonconvex side $\Gamma_{\rm nc}$, we have
\[%
\inf_{v'\in \mathcal{P}_{p}(0,L_{\rm nc})}\norm{v-v'}{L^2(0,L_{\rm nc})}\leq CC_1k\log^{1/2}(2+k)\re^{-p\tau},\quad k\geq k_1,%
\]%
where $\tau=\log{(\sqrt{1+(2\eps/L_{\rm nc})^2} + 2\eps/L_{\rm nc})}$, $\eps$ is given by \eqref{eqn:eps_defn}, $k_1>0$ and $C_1>0$ are as in Assumption \ref{ass:1}, and $C>0$ depends only on $\Omega$ and $k_1$.
\end{thm}

\begin{proof}
By Theorem~\ref{Gamma2Thm}, $v(s)$ is analytic and bounded in $D_\eps \supset \mathcal{E}_{0,L_{\rm nc},\eps}$.  The result follows by combining Lemma \ref{EllipseLem} with \rf{tildevEst}.
\end{proof}

The remaining terms all have singularities associated with corner singularities requiring geometric mesh refinement. Let
\begin{align}
\label{}
\delta_*:=1-\pi/\omega_{\rm min}\in(0,1/2),
\end{align}
where $\omega_{\rm min}$ denotes the smallest of the exterior angles of the polygon that are larger than $\pi$. Arguing as in \cite[\S 5]{HeLaMe:11} one can use Lemma \ref{EllipseLem} to prove:%
\begin{thm}[cf.~{\cite[Theorem 5.4]{HeLaMe:11}}]
\label{vjPMThm}
If \rf{npAssump} holds, then, for every $k_0>0$, for the approximation of $v^+(s)$ and $v^-(L_{\rm c}-s)$ on a convex side $\Gamma_{\rm c}$ we have
\begin{align*}
  \inf_{v'\in \mathcal{P}_{p,n}(0,L_{\rm c})}\|v^\pm-v'\|_{L^2(0,L_{\rm c})}\leq C\uM k^{1-\delta_*}\,\re^{-p\tau}, \quad k\geq k_0,
\end{align*}
where $\tau>0$ depends only on $\sigma$, the corner angles at the ends of $\Gamma_{\rm c}$, and $c$ (the constant in \eqref{npAssump}), and $C>0$ only on $\Omega$ and $k_0$. If also Assumption \ref{ass:1} holds, then the same estimate holds for the approximation of $v^-(L_{\rm nc}-s)$ on a nonconvex side $\Gamma_{\rm nc}$, except that $L_{\rm c}$ is replaced by $L_{\rm nc}$ in the above formula, and $\tau$ depends now on $\sigma$, $c$, and the exterior angle $\omega$ in Figure~\ref{fig-1}(a).
\end{thm}

We now combine these results into a single estimate for the best approximation error associated with the approximation of $\varphi\in L^2(\Gamma)$ by an element of $V_{N,k}$. From \rf{Decomp}, \rf{dudnGamma2Rep}, \rf{eqn:varphi}, Theorems \ref{vjPIThm}, \ref{vjtildeThm}, and \ref{vjPMThm}, and the definition of the approximation space $V_{N,k}$, the following result follows:
\begin{thm}
\label{dudnThm}
Suppose that Assumption \ref{ass:1} and \rf{npAssump} hold. Then, where $k_1$, $C_1$, and $c$ are the constants in those assumptions, we have
\begin{align}
  \label{BestAppdudn}
  \inf_{v'\in V_{N,k}}\norm{\varphi-v'}{L^2(\Gamma)}\leq C(M(u)k^{-\delta_*}+\log^{1/2}(2+k))\,\re^{-p\tau},\quad k\geq k_1,
\end{align}
where $C>0$ depends only on $C_1$, $\Omega$ and $k_1$, and $\tau>0$ only on $c$, $\sigma$, and $\Omega$.
\end{thm}

\section{Galerkin Method}
\label{sec:gal}
Having designed a HNA approximation space $V_{N,k}$ which can efficiently approximate $\varphi$, we select an element of $V_{N,k}$ by applying the Galerkin method to the integral equation \rf{eqn:op061210}, rewritten with $\varphi$ defined by (\ref{eqn:varphi}) as the unknown. That is, we seek $\varphi_N\in V_{N,k}\subset L^2\left(\Gamma\right)$ such that
\begin{equation}
  \left\langle \cA\varphi_N,v\right\rangle_{L^2\left(\Gamma\right)}=\frac{1}{k}\left\langle f-\cA\Psi,v\right\rangle_{L^2\left(\Gamma\right)},
 \quad\textrm{for all } v\in V_{N,k}.
  \label{eqn:gal}
\end{equation}
If Assumption~\ref{ass:2} holds (cf.\ the discussion at the end of \S\ref{sec:prob}),
then existence and uniqueness of the Galerkin solution $\varphi_N$ is guaranteed by the Lax-Milgram lemma. Moreover, C\'{e}a's lemma (e.g., \cite[Lemma~6.9]{ChGrLaSp:11}) gives the quasi-optimality estimate
\begin{equation}
\label{quasi-opt}
  \norm{\varphi-\varphi_N}{L^2(\Gamma)}\leq\frac{C_0k^{1/2}}{C_2}\inf_{v'\in V_{N,k}}\norm{\varphi-v}{L^2(\Gamma)}, \quad   k\geq k_2,%
\end{equation}
where $C_2$ and $k_2$ are the constants from Assumption~\ref{ass:2}, and $C_0$ is the constant from Lemma~\ref{lem:1} in the case that $k_0=k_2$. Combined with Theorem~\ref{dudnThm}, this gives:
\begin{thm}
\label{GalerkinThm}
Suppose that Assumption~\ref{ass:2} and \rf{npAssump} hold. Then, where $k_2$, $C_2$, and $c$ are the constants in those assumptions, we have
\begin{align}
  \label{GalerkinErrorEst}
  \norm{\varphi-\varphi_N}{L^2(\Gamma)}\leq Ck^{1/2}(M(u)k^{-\delta_*}+\log^{1/2}(2+k))\,\re^{-p\tau}, \quad   k\geq k_2,
\end{align}
where $C>0$ depends only on $C_2$, $\Omega$ and $k_2$, and $\tau>0$ only on $c$, $\sigma$, and $\Omega$.
\end{thm}

An approximation $u_N$ to the solution $u$ of the BVP can be found by inserting the approximation $\partial u/\partial\bn \approx \Psi+ k \varphi_N $ into the formula~\rf{RepThm}, i.e.
\begin{align*}
  u_N(\bx) := u^i(\bx) - \int_{\Gamma}\Phi_k(\bx,\by)\left(\Psi(\by) +  k\varphi_N(\by)\right)\,\rd s(\by), \qquad \bx\in D.
\end{align*}
Arguing as in the proof of \cite[Theorem 6.3]{HeLaMe:11}, noting that $\uM = \norm{u}{L^\infty(D)}\geq 1$ (since $|u(\bx)| \sim |u^i(\bx)| = 1$ as $|\bx|\to \infty$), we deduce:

\begin{thm}
\label{DomainThm} Under the assumptions of Theorem~\ref{GalerkinThm} we have
\begin{align}
  \label{DomainErrorEst}
  \frac{\norm{u-u_N}{L^\infty(D)}}{\norm{u}{L^\infty(D)}} \leq Ck\log(2+k)\,\re^{-p\tau}, \quad   k\geq k_2,
\end{align}
where $C>0$ depends only on $C_2$, $\Omega$ and $k_2$, and $\tau>0$ only on $c$, $\sigma$, and $\Omega$.
\end{thm}

An object of interest in applications is the \emph{far field pattern} of the scattered field.
An asymptotic expansion of the representation~\rf{RepThm} reveals that (cf.~\cite{CoKr:92})%
\begin{equation}
  \label{FarField}
  u^s(\bx) \sim \dfrac{\re^{\ri\pi/4}}{2\sqrt{2\pi}}\dfrac{\re^{\ri kr}}{\sqrt{kr}}F(\hat{\bx}), \quad \mbox{as }r:=|\bx|\rightarrow\infty,
\end{equation}
where $\hat{\bx}:=\bx/|\bx|\in \mathbb{S}^1$, the unit circle, and
\begin{align}
  \label{FDef}
  F(\hat{\bx}):=-\int_\Gamma \re^{-\ri k\hat{\bx}\cdot \by}\pdone{u}{\bn}(\by)\,\rd s(\by), \qquad \hat{\bx}\in \mathbb{S}^1.
\end{align}
An approximation $F_N$ to the far field pattern $F$ can be found by inserting the approximation $\partial u/\partial\bn \approx \Psi + k\varphi_N$ into the formula~\rf{FDef}, i.e.
\begin{equation}
  F_N(\hat{\bx}):=-\int_\Gamma \re^{-\ri k\hat{\bx}\cdot \by} \left(\Psi(\by)+ k\varphi_N(\by)  \right) \,\rd s(\by), \qquad \hat{\bx}\in \mathbb{S}^1.
  \label{eqn:FFP_approx}
\end{equation}
The proof of the following estimate follows precisely that of \cite[Theorem 6.4]{HeLaMe:11}.

\begin{thm}
\label{FarFieldThm} Under the assumptions of Theorem~\ref{GalerkinThm} we have
\begin{align}
  \label{FarFieldErrorEst}
  \norm{F-F_N}{L^\infty(\mathbb{S}^1)}\leq Ck^{3/2}(M(u)k^{-\delta_*}+\log^{1/2}(2+k))\,\re^{-p\tau}, \quad   k\geq k_2,
\end{align}
where $C>0$ depends only on $C_2$, $\Omega$ and $k_2$, and $\tau>0$ only on $c$, $\sigma$, and $\Omega$.
\end{thm}

The above results hold for all polygons $\Omega$ in the class $\cC$ of Definition~\ref{classCDef}, provided that Assumption~\ref{ass:2} holds.
But, as remarked in \S\ref{sec:prob}, if $\Omega$ is star-like and $\cA=\cA_k$, then Assumption \ref{ass:2} holds for every $k_2>0$.
Furthermore, in this case it has been shown in \cite[Theorem 4.3]{HeLaMe:11} (and see Remark~\ref{rem:M}) that
\begin{align}
\label{MBoundProp2}
  \uM \leq C k^{1/2}\log^{1/2}{(2+k)}, \quad    k\geq k_2,
\end{align}
where $C$ depends only on $k_2$ and $\Omega$. %
Thus the above results have the following corollary which requires no coercivity assumption.

\begin{cor}
\label{StarlikeCor}
Suppose that $\Omega$ is a star-like member of the class $\cC$. Suppose also that we choose $\cA=\cA_k$, the star-combined potential operator defined in \rf{eqn:star_comb}, and that we choose $n$ so that \rf{npAssump} holds. Then, for any $k_2>0$, for $k\geq k_2$ we have  %
\begin{align}
\label{GalerkinErrorEstStarlike}
\norm{\varphi-\varphi_N}{L^2(\Gamma)}&\leq Ck^{1-\delta_*}\log^{1/2}(2+k)\,\re^{-p\tau},\\
\label{DomainErrorEstStarlike}
\frac{\norm{u-u_N}{L^\infty(D)}}{\norm{u}{L^\infty(D)}} &\leq Ck\log(2+k)\,\re^{-p\tau},\\
\label{FarFieldErrorEstStarlike}
\norm{F-F_N}{L^\infty(\mathbb{S}^1)}&\leq Ck^{2-\delta_*}\log^{1/2}(2+k)\,\re^{-p\tau},%
\end{align}
where $C>0$ depends only on $\Omega$ and $k_2$, and $\tau>0$ depends only on $c$, $\sigma$, and $\Omega$.
\end{cor}
\begin{rem}
\label{rem:sncw}
As remarked at the end of \S\ref{sec:prob}, it is reasonable, based on the numerical evidence in \cite{BeSp:11}, to conjecture that Assumption \ref{ass:2} holds for every $k_2>0$ also for $\cA=\cA_{k,k}$, for all members of the class $\cC$ (not necessarily star-like). Thus we conjecture that \rf{GalerkinErrorEstStarlike}--\rf{FarFieldErrorEstStarlike} hold also for $\cA=\cA_{k,k}$, for all members of the class $\cC$.
\end{rem}

\begin{rem}
  \label{rem:log2k}
The algebraically $k$-dependent prefactors in the error estimates of this section can be absorbed into the exponentially decaying factors by allowing $p$ to grow modestly with increasing~$k$. We illustrate this in the case of \rf{DomainErrorEstStarlike}. If %
\[ %
  p\geq \frac{\log{(k\log(2+k))}}{c_0},
\] %
for some $0<c_0<\tau$, then~\rf{DomainErrorEstStarlike} can be replaced by
\begin{align}
  \label{GalerkinExpDecay}
  \frac{\norm{u-u_N}{L^\infty(D)}}{\norm{u}{L^\infty(D)}}\leq C\re^{-p\kappa}, \quad   k\geq k_2,
\end{align}
where $\kappa=\tau-c_0$, and both $C$ and $\kappa$ are independent of $k$.
 Since the number of degrees of freedom, $N$, is given by \rf{NDof}, and it is sufficient to increase $n$ in proportion to $p$ for \rf{GalerkinExpDecay} to hold, it follows from \rf{GalerkinExpDecay} that, to maintain a fixed accuracy, we need only increase $N$ in proportion to $(\log(k\log k))^2$ as $k\to\infty$.
\end{rem}

\section{Numerical Results}
\label{sec:num}
We present numerical computations of the Galerkin approximation $\varphi_N$ defined by (\ref{eqn:gal}), using the standard combined-potential formulation, $\cA=\cA_{k,k}$, given by \eqref{eqn:BIE_standard}, for a particular star-like scatterer in the class $\cC$. In contrast to the choice $\cA=\cA_k$, the star-combined operator given by \eqref{eqn:star_comb}, for which Corollory \ref{StarlikeCor} holds, we do not have a complete theory for $\cA=\cA_{k,k}$ in the sense that, while Theorems \ref{GalerkinThm}--\ref{FarFieldThm} apply,  Assumption~\ref{ass:2} has not been shown to hold for obstacles in the class $\cC$ for $\cA=\cA_{k,k}$. One point of the computations in this section is to provide evidence for the conjecture in Remark \ref{rem:sncw} that \eqref{GalerkinErrorEstStarlike}--\eqref{FarFieldErrorEstStarlike} hold also for $\cA=\cA_{k,k}$.

The scatterer we consider is shown in Figure~\ref{fig:scatterer1}. Its nonconvex sides have length $2\pi$ and its convex sides length $4\pi$, so the total length of the boundary is $12\pi$, which is $6k$ wavelengths since the wavelength $\lambda=2\pi/k$.  We consider two different incident directions $\alpha$, measured anticlockwise from the downwards vertical (as in Figure~\ref{fig-1}):
\begin{enumerate}
  \item $\alpha=5\pi/4$, as shown in Figure~\ref{fig:total1}(a); in this case, multiply-reflected rays are present in the asymptotic solution.
  \item $\alpha=5\pi/3$, as shown in Figure~\ref{fig:total1}(b); in this case, one of the nonconvex sides is partially illuminated.%
\end{enumerate}
The scatterers, the incident directions, the corresponding total fields for $k=10$, and a circle of radius $3\pi$ on which we compute the total field for the purpose of calculating errors (see Figure~\ref{fig:total4} below) are plotted in Figure~\ref{fig:total1}.
In all of our experiments we take $n=2(p+1)$. From~(\ref{NDof}), %
the total number of degrees of freedom is then $N=12p^2+28p + 16$.
Quadrature routines for the evaluation of oscillatory integrals similar to those that appear in our ($N$-dimensional) linear system (arising from \rf{eqn:gal}) are described in \cite[\S4]{ChGrLaSp:11} - for more details see \cite[\S4]{TwiggerThesis:12}.
We evaluate the Fresnel integral ${\rm Fr}$ appearing in the leading order behaviour $\Psi$ on nonconvex sides (cf.\ Lemma \ref{def:udDef} and Theorem \ref{Gamma2Thm}),
and hence in the integrals on the right hand side of \rf{eqn:gal}, efficiently and accurately using the method of \cite{AlChLP:13}.

We will demonstrate exponential decay of $\norm{\varphi-\varphi_N}{L^2(\Gamma)}$ as $p$ increases, for fixed $k$, as predicted by~(\ref{GalerkinErrorEst}).
More significantly, we will also see that, as $k$ increases with $p$ fixed, $\norm{\varphi-\varphi_N}{L^2(\Gamma)}$ actually decreases, suggesting that we can maintain accuracy as $k\to\infty$ with a fixed number of degrees of freedom, and that the bound~(\ref{GalerkinErrorEstStarlike}) is not sharp.  Similarly, we will see that the relative error, $\norm{\varphi-\varphi_N}{L^2(\Gamma)}/\norm{\varphi}{L^2(\Gamma)}$, grows only very slowly as $k$ increases with $N$ fixed.  We will also compute the solution in the domain and the far field pattern, making comparison with the error estimates (\ref{DomainErrorEstStarlike}) and~(\ref{FarFieldErrorEstStarlike}).

Since $N$ depends only on $p$, and the values of $p$ are more intuitively meaningful, we introduce the additional notation $\psi_p(s):=\varphi_{N}(s)$.  We begin in Figure~\ref{fig:bdy_soln1} by plotting $|\psi_7(s)|$ (sampled at 100,000 evenly spaced points on the boundary) for $\alpha=5\pi/3$ and
$k=10$ and $160$.
\begin{figure}[t]
\begin{center}
\subfigure[$\alpha=5\pi/3$, $k=10$]{\includegraphics[width=5.9cm]{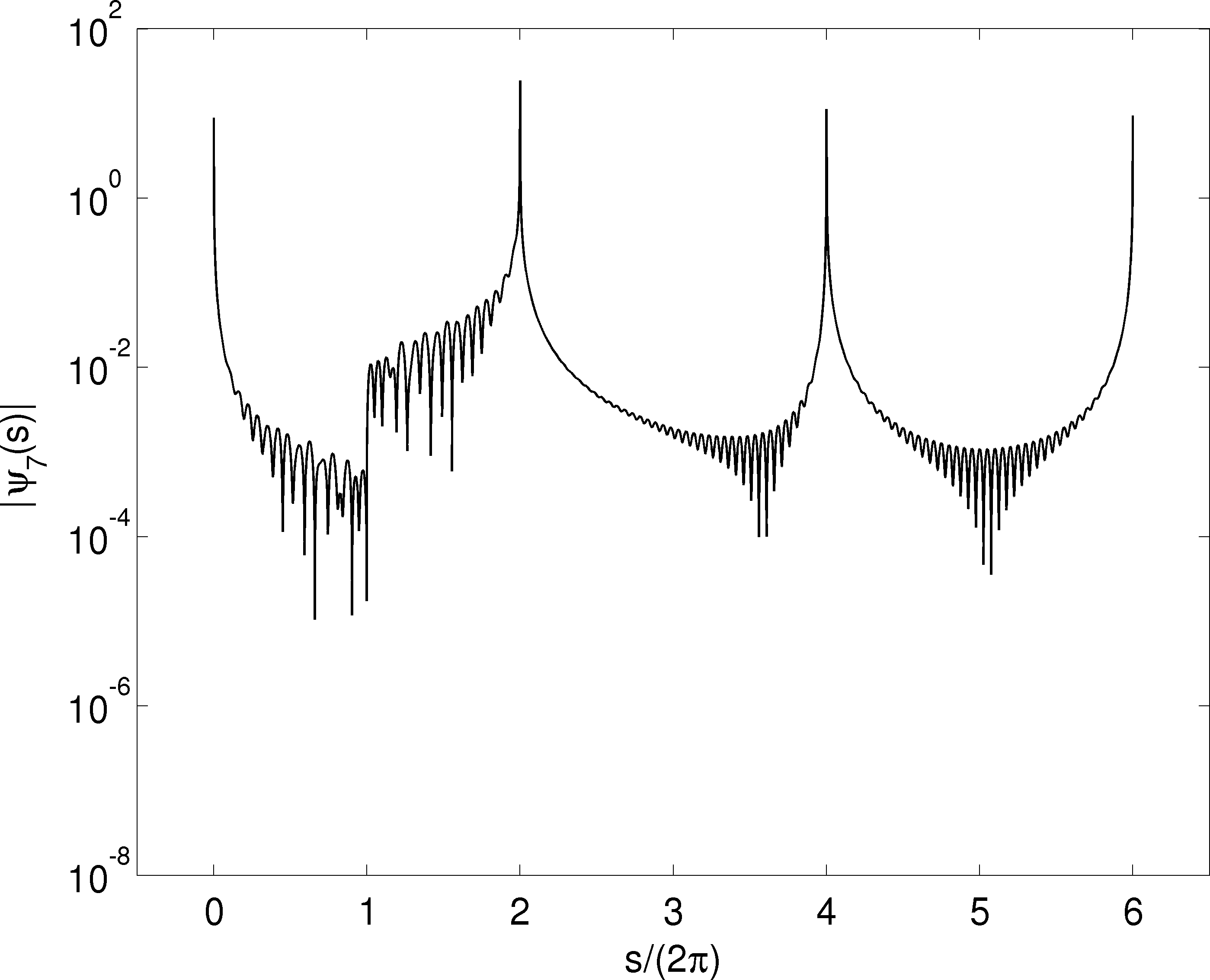}}
\hs{2}
\subfigure[$\alpha=5\pi/3$, $k=160$]{\includegraphics[width=5.9cm]{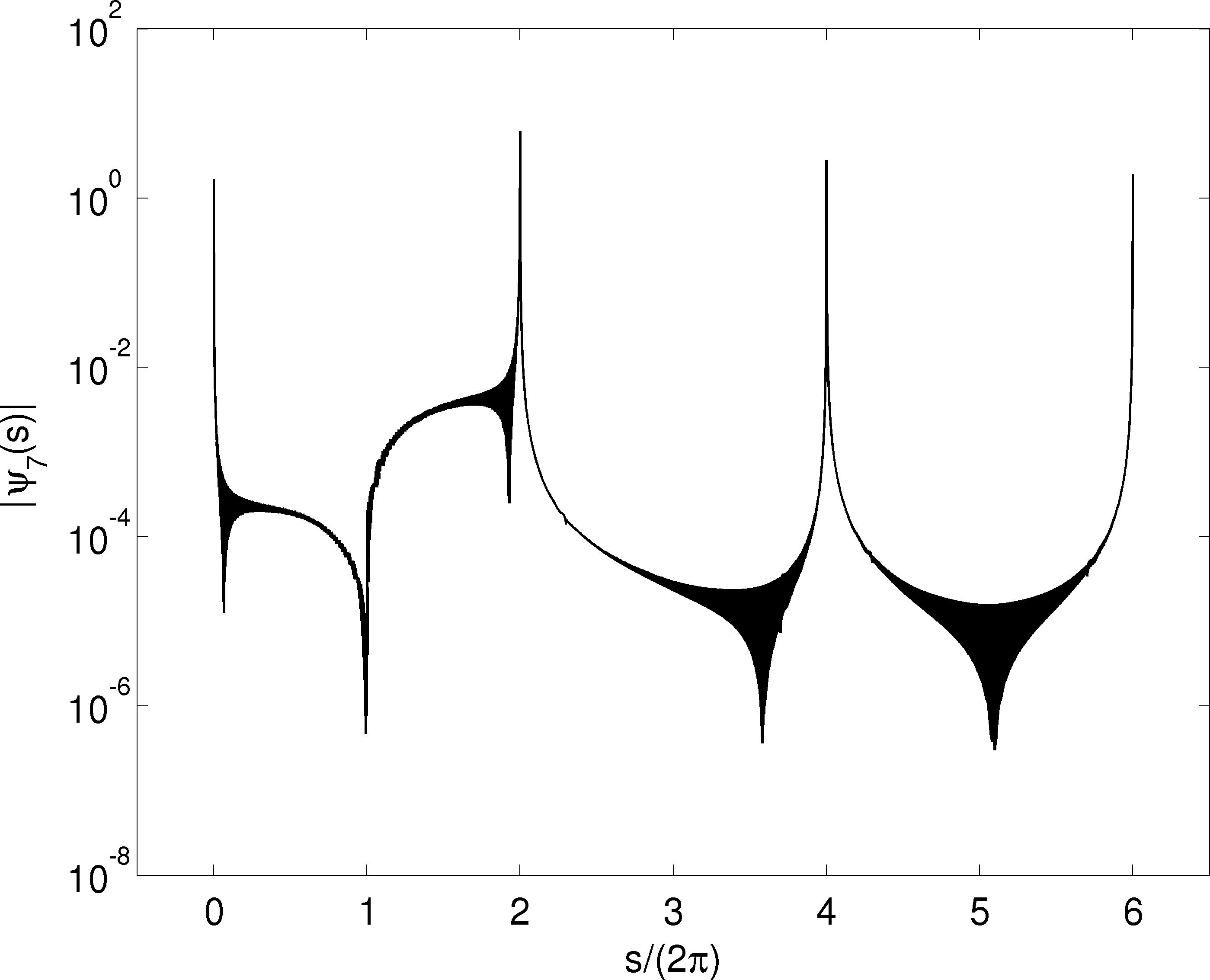}}
\end{center}
\caption{Boundary solution for $\alpha=5\pi/3$, with $k=10$ and $k=160$.}
\label{fig:bdy_soln1}
\end{figure}
The corner between the two nonconvex sides is at $s/(2\pi)=1$; the corners between convex and nonconvex sides are at  $s/(2\pi)=2$ and $s/(2\pi)=0$ (equivalently, by periodicity, $s/(2\pi)=6$), and the corner between the two convex sides is at $s/(2\pi)=4$.  There is a singularity in the solution $\varphi$ at all corners except the one between the nonconvex sides, where $\varphi=0$.  These singularities are evident in Figure~\ref{fig:bdy_soln1} as is the increased oscillation for larger~$k$.  (The apparent shaded region is an artefact of very high oscillation.) %

In Figure~\ref{fig:rel_errors} we plot the relative $L^2$ and $L^1$ errors against $p$, for the two angles of incidence, for three values of $k$.
\begin{figure}[tb]
\begin{center}
\subfigure[$\alpha=5\pi/4$ - relative $L^2$ errors]{\includegraphics[width=5.5cm]{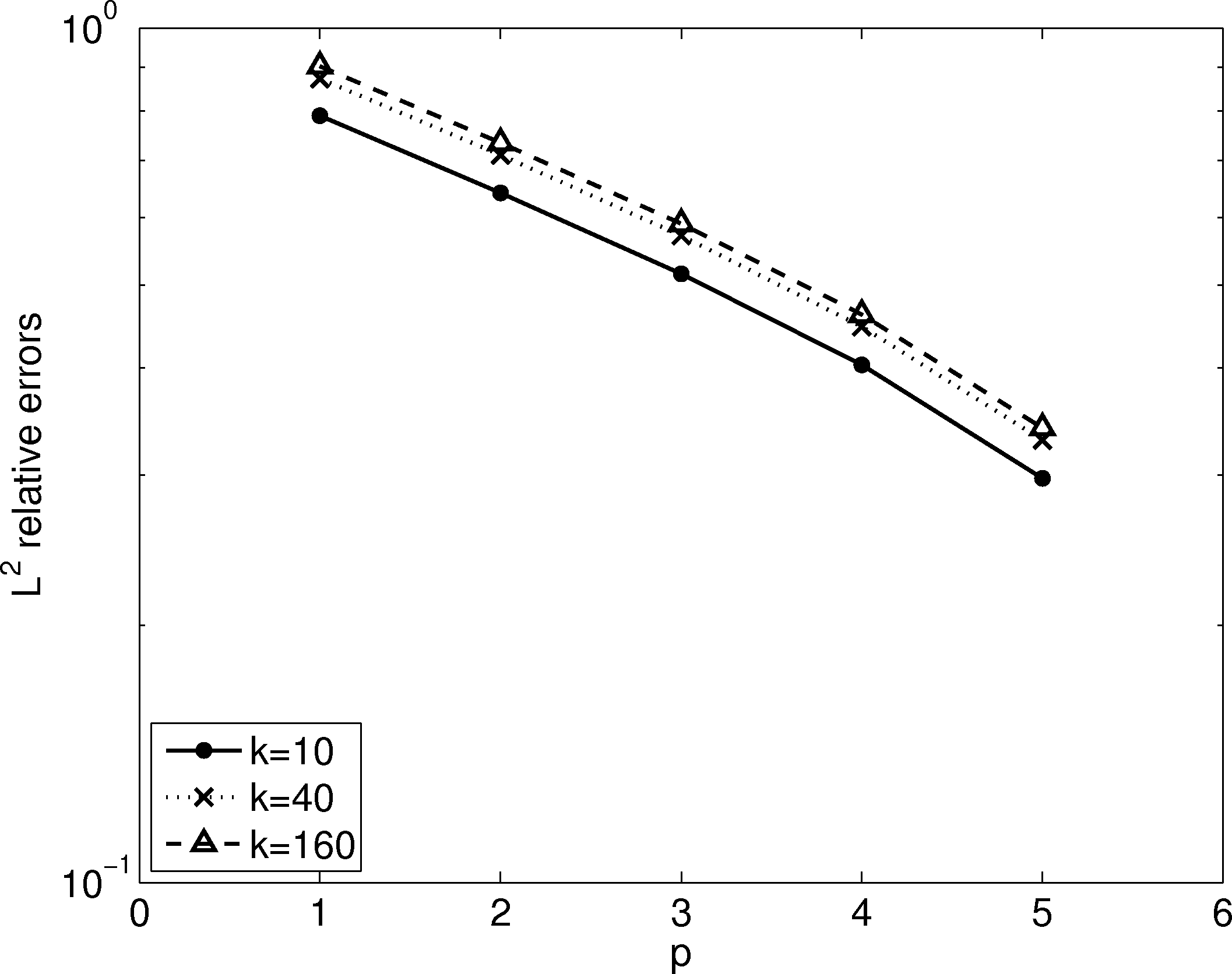}}
\hs{1}
\subfigure[$\alpha=5\pi/4$ - relative $L^1$ errors]{\includegraphics[width=5.5cm]{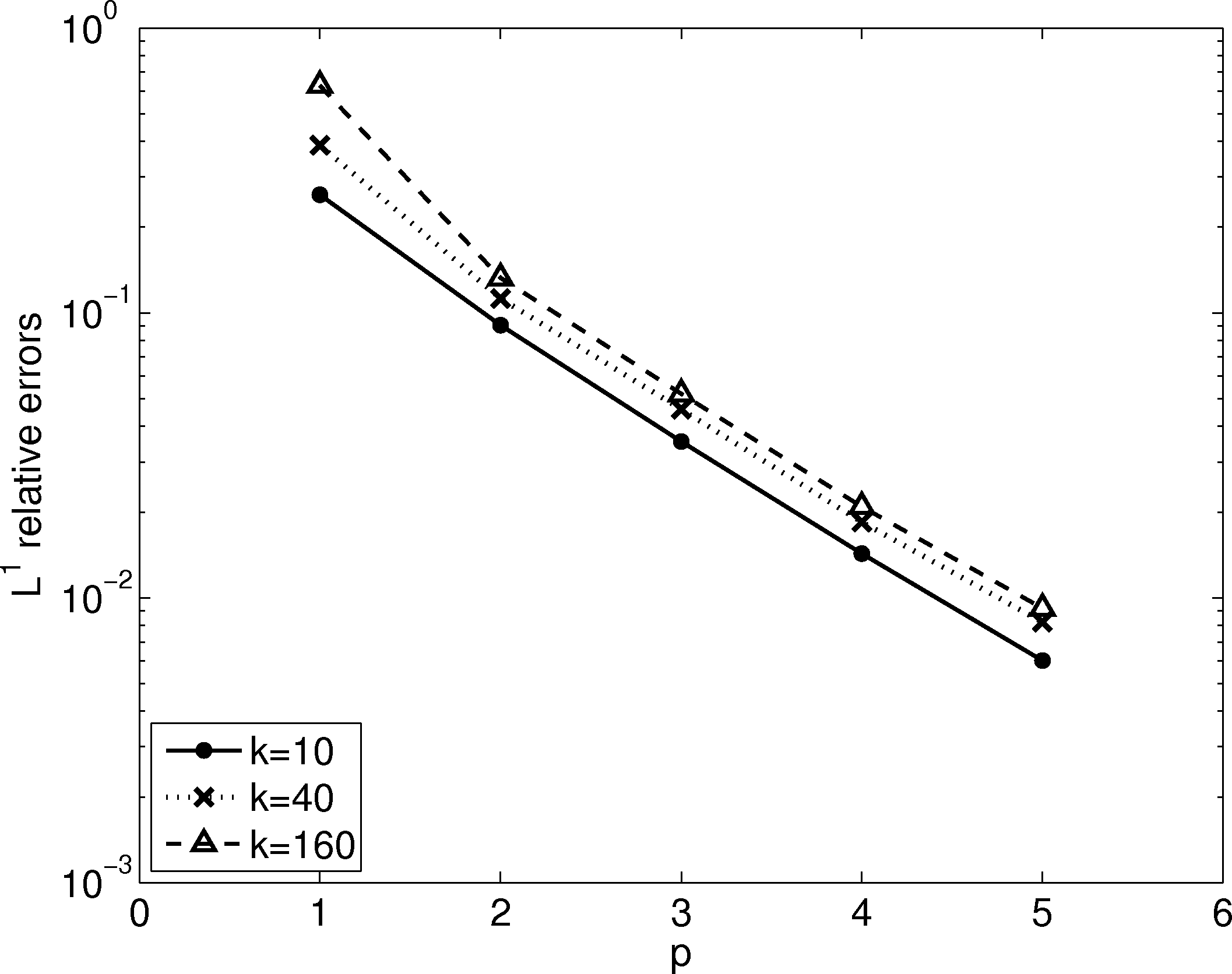}}
\vs{1}
\subfigure[$\alpha=5\pi/3$ - relative $L^2$ errors]{\includegraphics[width=5.5cm]{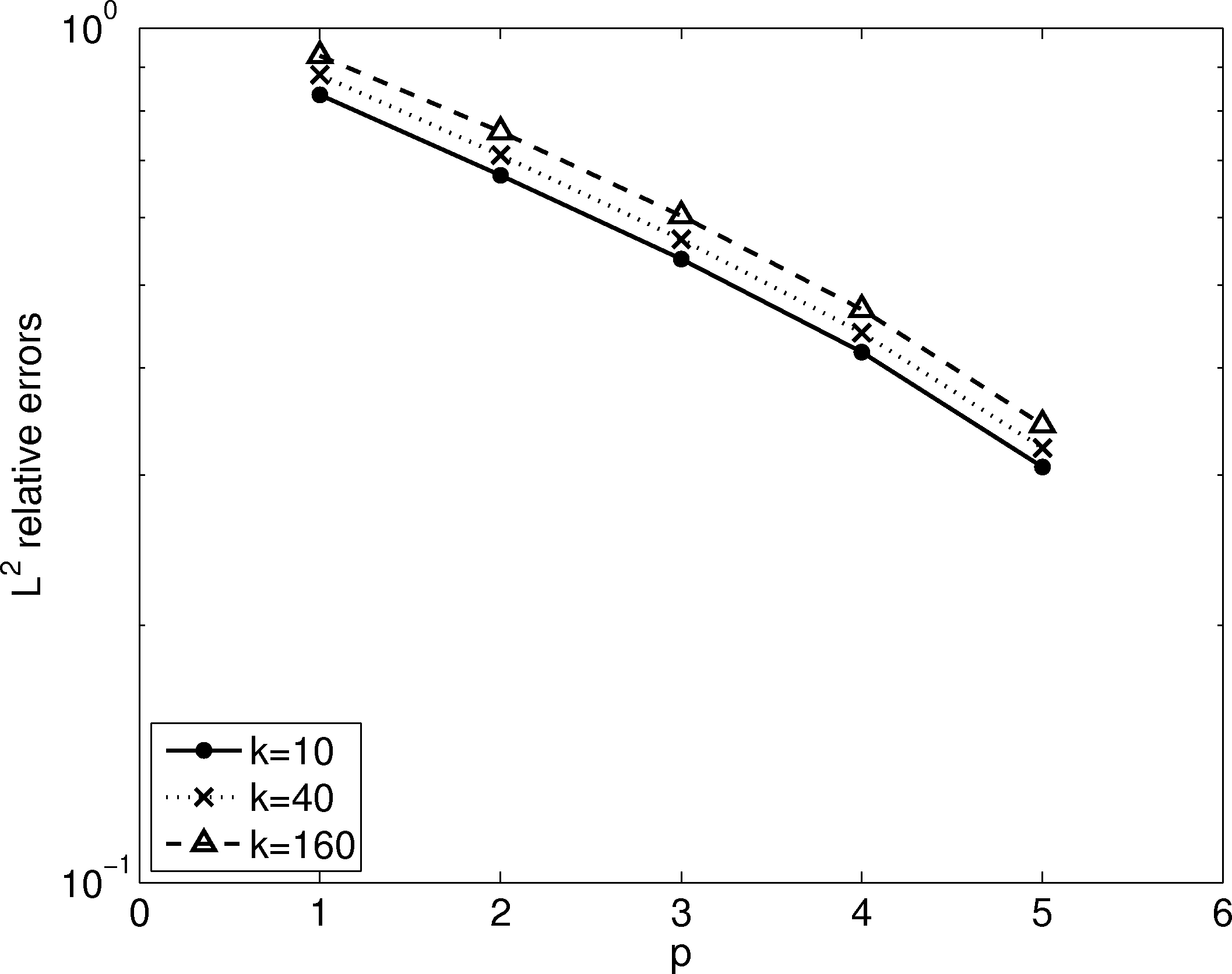}}
\hs{1}
\subfigure[$\alpha=5\pi/3$ - relative $L^1$ errors]{\includegraphics[width=5.5cm]{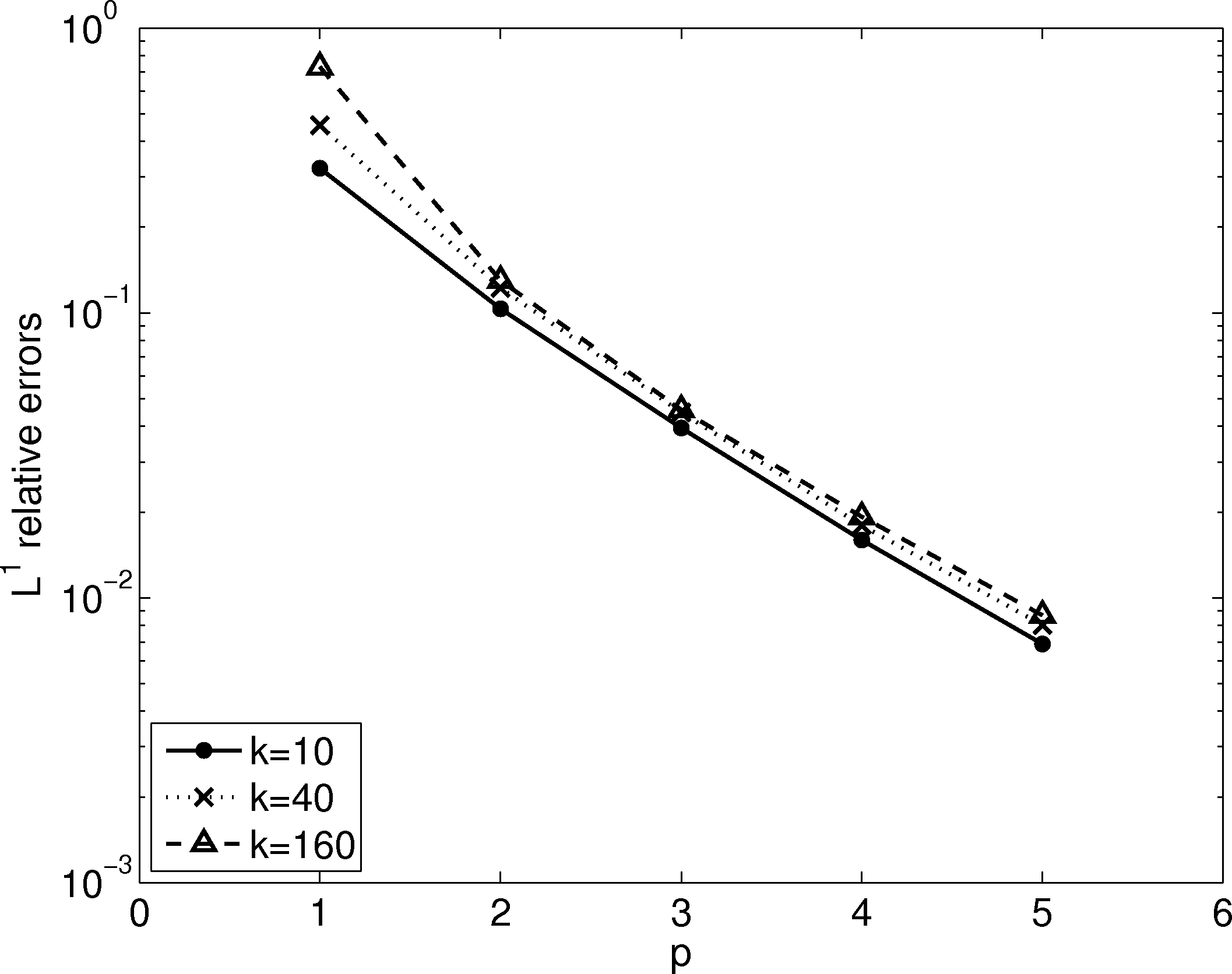}}
\end{center}
\caption{Relative $L^2$ and $L^1$ errors in boundary solution.}
\label{fig:rel_errors}
\end{figure}
We take the ``exact'' reference solutions to be those computed with $p=7$, as plotted in Figure~\ref{fig:bdy_soln1} for the case $\alpha =5\pi/3$.  The $L^2$ and $L^1$ norms are computed by high-order composite Gaussian quadrature on a mesh graded towards the corner singularities; experimental evidence suggests that these calculations are accurate to at least two significant figures.

Figure~\ref{fig:rel_errors} shows the exponential decay as $p$ increases that is predicted for the $L^2$ error by~(\ref{GalerkinErrorEst}).  %
 A key question is how the accuracy depends on $k$; we see that in all four plots in Figure~\ref{fig:rel_errors} the relative errors increase only very mildly as $k$ increases. To investigate this further, in Table~\ref{table2} we show results for the two angles of incidence for $p=4$ (and hence $N=320$), for a range of $k$. We tabulate $L^2$ errors, relative $L^2$ and $L^1$ errors, and also $N/(L/\lambda)$, the average number of degrees of freedom per wavelength.  As $k$ increases, the relative errors increase very slowly, the absolute $L^2$ error actually decreases, while the average number of degrees of freedom per wavelength decreases in proportion to $k^{-1}$.
\begin{table}[htbp]
  \caption{$L^2$ and $L^1$ errors for each example, fixed $p=4$ (and hence $N=320$), various~$k$, with $N/(L/\lambda)$ the average number of degrees of freedom per wavelength along the boundary.}
  \label{table2}
  \begin{center}
    \begin{tabular}{|c|c|r|c|c|c|c|c|}
       \hline
      $\alpha$ & $k$ & $\frac{N}{L/\lambda}$  & $\|\psi_7-\psi_4\|_{L^2(\Gamma)}$ & $\mu$ & $\frac{\|\psi_7-\psi_4\|_{L^2(\Gamma)}}{\|\psi_7\|_{L^2(\Gamma)}}$ & $\frac{\|\psi_7-\psi_4\|_{L^1(\Gamma)}}{\|\psi_7\|_{L^1(\Gamma)}}$ & COND \\
      \hline
 $5\pi/4$ &   5 &  10.67 &   8.37$\times10^{-1}$ &  -0.35 &   3.90$\times10^{-1}$ &   1.03$\times10^{-2}$ & 3.36$\times10^{5}$ \\
          &  10 &   5.33 &   6.55$\times10^{-1}$ &  -0.19 &   4.04$\times10^{-1}$ &   1.43$\times10^{-2}$ & 1.87$\times10^{2}$ \\
          &  20 &   2.67 &   5.72$\times10^{-1}$ &  -0.29 &   4.24$\times10^{-1}$ &   1.69$\times10^{-2}$ & 1.34$\times10^{2}$ \\
          &  40 &   1.33 &   4.68$\times10^{-1}$ &  -0.91 &   4.47$\times10^{-1}$ &   1.85$\times10^{-2}$ & 1.73$\times10^{2}$ \\
          &  80 &   0.67 &   2.48$\times10^{-1}$ &  -0.20 &   4.39$\times10^{-1}$ &   1.91$\times10^{-2}$ & 2.30$\times10^{2}$ \\
          & 160 &   0.33 &   2.16$\times10^{-1}$ &        &   4.62$\times10^{-1}$ &   2.09$\times10^{-2}$ & 3.03$\times10^{2}$ \\
   \hline
 $5\pi/3$ &   5 &  10.67 &   8.64$\times10^{-1}$ &  -0.46 &   4.05$\times10^{-1}$ &   1.17$\times10^{-2}$ & 3.36$\times10^{5}$ \\
          &  10 &   5.33 &   6.30$\times10^{-1}$ &  -0.54 &   4.18$\times10^{-1}$ &   1.60$\times10^{-2}$ & 1.87$\times10^{2}$ \\
          &  20 &   2.67 &   4.32$\times10^{-1}$ &  -0.46 &   4.27$\times10^{-1}$ &   1.80$\times10^{-2}$ & 1.34$\times10^{2}$ \\
          &  40 &   1.33 &   3.15$\times10^{-1}$ &  -0.46 &   4.40$\times10^{-1}$ &   1.80$\times10^{-2}$ & 1.73$\times10^{2}$ \\
          &  80 &   0.67 &   2.30$\times10^{-1}$ &  -0.45 &   4.54$\times10^{-1}$ &   1.88$\times10^{-2}$ & 2.30$\times10^{2}$ \\
          & 160 &   0.33 &   1.69$\times10^{-1}$ &        &   4.69$\times10^{-1}$ &   1.92$\times10^{-2}$ & 3.03$\times10^{2}$ \\
   \hline
    \end{tabular}
  \end{center}
\end{table}
We also tabulate $\log_2(\mbox{error}(2k)/\mbox{error}(k))$,  where error$(k)$ refers to the absolute $L^2$ error for a particular value of $k$. This is an estimate of the order of convergence, $\mu$, on a hypothesis that $\mbox{error}(k) \sim k^{\mu}$ as $k\rightarrow\infty$. Since, for this scatterer, $\delta_*\approx0.4350$,  a value $\mu\approx 0.5650$ is the largest consistent with the bound (\ref{GalerkinErrorEstStarlike}). In fact, we see values in the range $(-0.91,-0.19)$, suggestive that the bound (\ref{GalerkinErrorEstStarlike}) overestimates the error growth as $k$ increases. In part this overestimate may be due to using the bound \eqref{MBoundProp2} to get \eqref{GalerkinErrorEstStarlike}; as noted in Remark~\ref{rem:M} we conjecture that in fact $\uM=\ord{1}$ as $k\to\infty$.

In the final column of Table~\ref{table2} we also show the condition number (COND) of the $N$-dimensional linear system arising from~\rf{eqn:gal}.  For fixed $p=4$, the condition number increases slowly as $k$ increases, for $k\geq10$.  The condition number is significantly larger for $k=5$.  We investigate the dependence of the condition number on both $k$ and $p$ further in Figure~\ref{fig:cond}.
\begin{figure}[htbp]
\begin{center}
  \includegraphics[height=7cm]{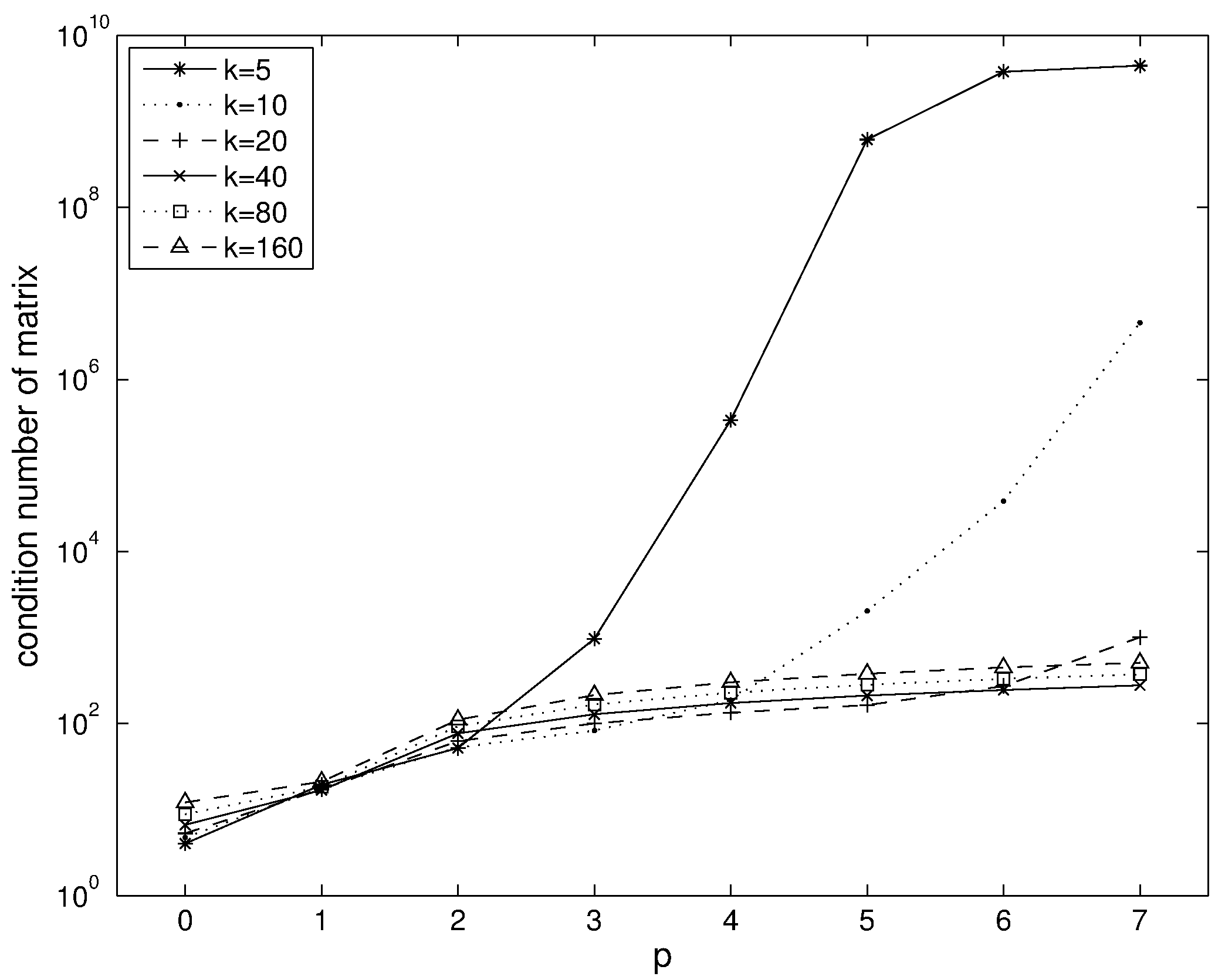}
\end{center}
\caption{Condition number of the $N$-dimensional linear system arising from~\rf{eqn:gal}.}
\label{fig:cond}
\end{figure}
For fixed $k$, the condition number grows slowly (approximately with $O(\log^2 p)$) as $p$ increases, for $p<p_0$, before growing approximately exponentially with respect to $p$ for $p>p_0$, where the value of $p_0$ appears to increase as $k$ increases; indeed, it appears from Figure~\ref{fig:cond} that, for $k=5, 10, 20$,
the value of $p_0$ corresponds to the point at which our discretisation is equivalent to approximately 5--7 degrees of freedom per wavelength (recalling that  $N = 12p^2 + 28p + 16$ and the boundary is $6k$ wavelengths long).
For $k=5$, as $p$ increases beyond $p=5$ (representing approximately 15 degrees of freedom per wavelength) this exponential growth appears to tail off.

We now return to Figure~\ref{fig:rel_errors} where we see that the $L^2$ errors, while decreasing exponentially as $p$ increases, are large in absolute value. Errors of a similar magnitude are seen in the corresponding convex case \cite{HeLaMe:11}. There it is noted that the $L^2$ errors blow up as the largest exterior angle, $\omega_{\rm max}$, approaches $2\pi$, this because $\|\varphi\|_{L^2(\Gamma)}$ itself blows up in the same limit (this can be seen from the bound \eqref{vjpmBounds} which is sharp in the limit $s\to0$). Thus large $L^2$ errors are inevitable for $\omega_{\rm max}$ close to $2\pi$. One ``solution'' is to measure errors in a more appropriate norm: in particular this blow up is not seen in the $L^1$ norm and, indeed, the relative $L^1$ errors in Figure \ref{fig:rel_errors} are 20--40 times smaller than the corresponding $L^2$ errors (note the different scales in (b) and (d) compared to (a) and (c)).

We now turn our attention to the approximation of $u(\bx)$, $\bx\in D$, and of the far field pattern $F$ (often the quantities of real interest in scattering problems).  As is common for linear functionals of the solution on the boundary, the errors in $u(\bx)$ and $F(\hat \bx)$ are, in general, much smaller than the relative errors in $\varphi$.  To investigate the accuracy of $u_N(\bx)$, we compute the error in this solution on a circle of radius $3\pi$ surrounding the scatterer, as illustrated in Figure~\ref{fig:total1}.  To allow easy comparison between different discretizations, noting again that for each example $N$ depends only on $p$, we denote the solution on this circle  (with a slight abuse of notation) by
$u_p(t) := u_N(\bx(t))$, $t\in[0,2\pi]$,
where $t=0$ corresponds to the direction from which $u^i$ is incident, and $\bx(t)$ is a point at angular distance $t$ around the circle.

In Figure~\ref{fig:total4} we plot for each example the relative maximum error on the circle,
\[ \frac{\max_{t\in[0,2\pi]}|u_7(t)-u_p(t)|}{\max_{t\in[0,2\pi]}|u_7(t)|}, \]
computed over 30,000 evenly spaced points in $[0,2\pi]$, for $k=10$, $40$, and $160$.
\begin{figure}[htbp]
\begin{center}
\subfigure[$\alpha=5\pi/4$]{\includegraphics[width=5.9cm]{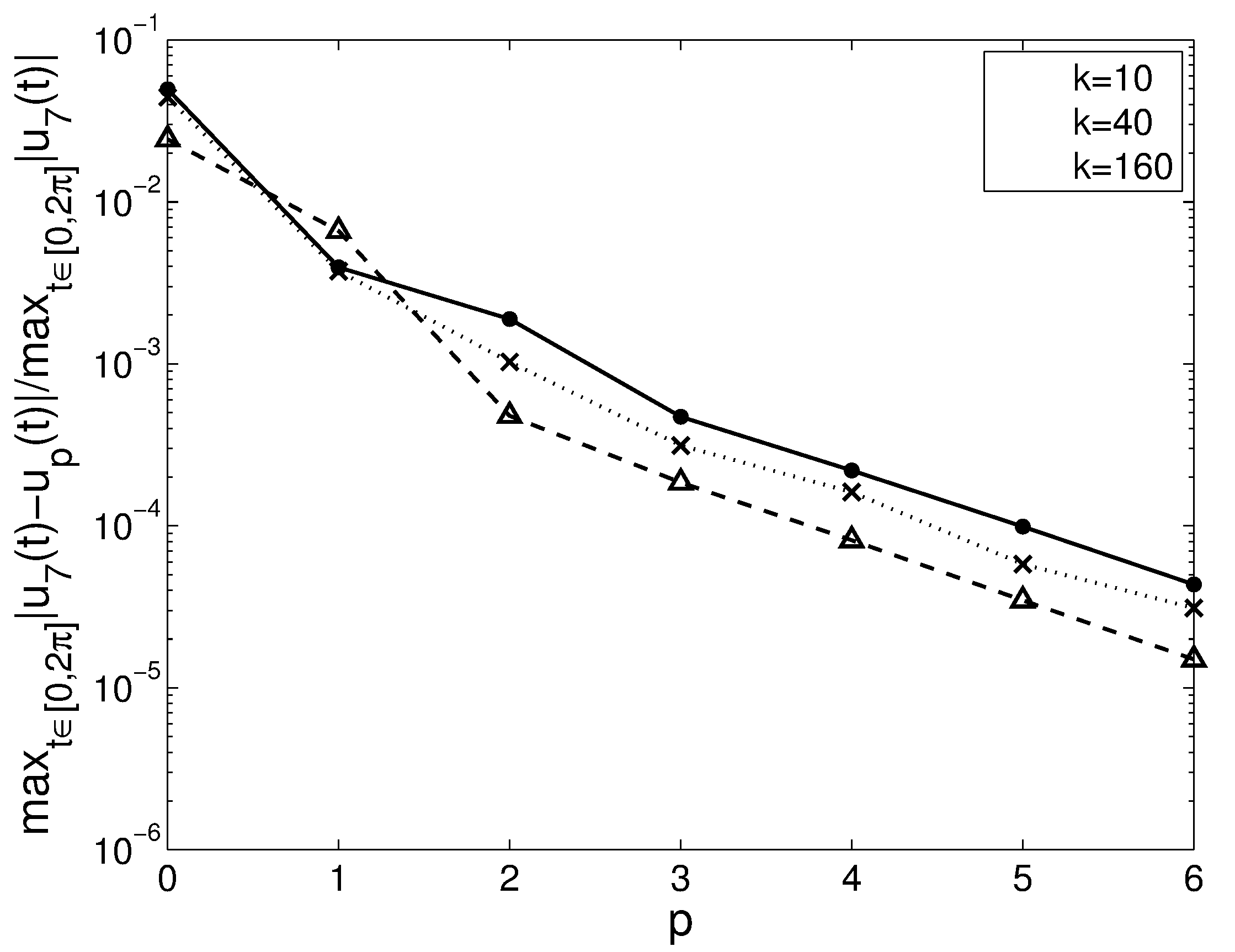}}
\hs{2}
\subfigure[$\alpha=5\pi/3$]{\includegraphics[width=5.9cm]{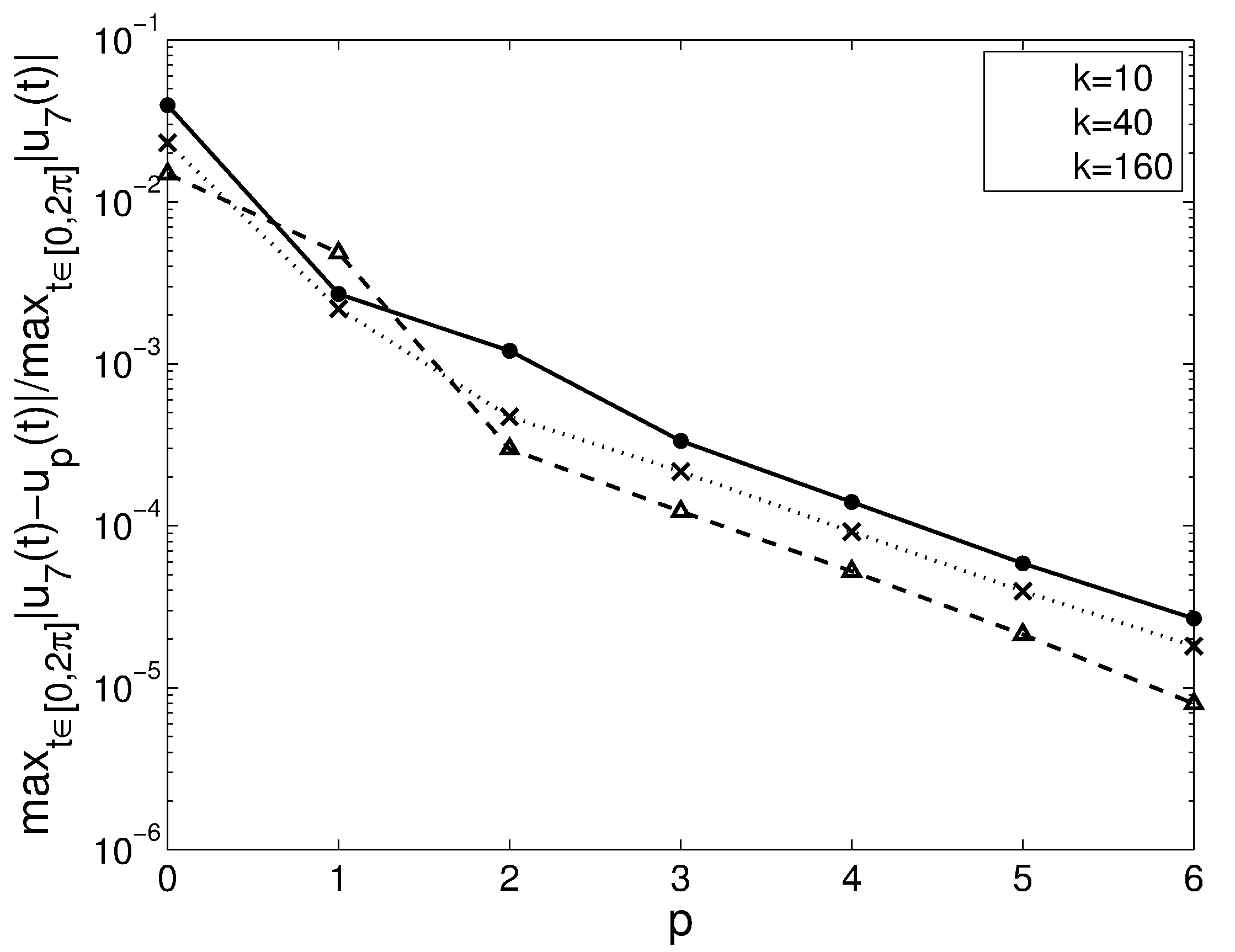}}
\end{center}
\caption{Relative maximum errors on the circle of Figure~\ref{fig:total1}.}
\label{fig:total4}
\end{figure}
The exponential decay as $p$ increases predicted by Theorem~\ref{DomainThm} is clear. %
 Moreover, for fixed $p\geq 2$, the relative maximum error decreases as $k$ increases; this is better than the mild growth with $k$ of the bound \eqref{DomainErrorEstStarlike}. %
 These relative errors are much smaller than those on the boundary in Figure~\ref{fig:rel_errors}.  %

Finally, we compute our approximation~(\ref{eqn:FFP_approx}) to the far field pattern.
Again, %
with a slight abuse of notation, we define
$F_p(t) := F_N(\hat{\bx}(t))$, $t\in[0,2\pi]$,
where $t=0$ corresponds to the direction from which $u^i$ is incident and $\hat \bx(t)$ is a point at angular distance $t$ around the unit circle.
Plots of $|F_7(t)|$ (the magnitude of the far field pattern computed with our finest discretization), for $k=10$ and $160$ and the two incident directions, are shown in Figure~\ref{fig:FFP1}.
In Figure~\ref{fig:FFP3} we plot approximations to $\|F_7-F_p\|_{L^{\infty}(\mathbb{S}^1)}$ for $k=10$, $40$, and $160$, for the two incident directions.
\begin{figure}[htbp]
\begin{center}
\subfigure[$\alpha=5\pi/4$, $k=10$]{\includegraphics[width=5.9cm]{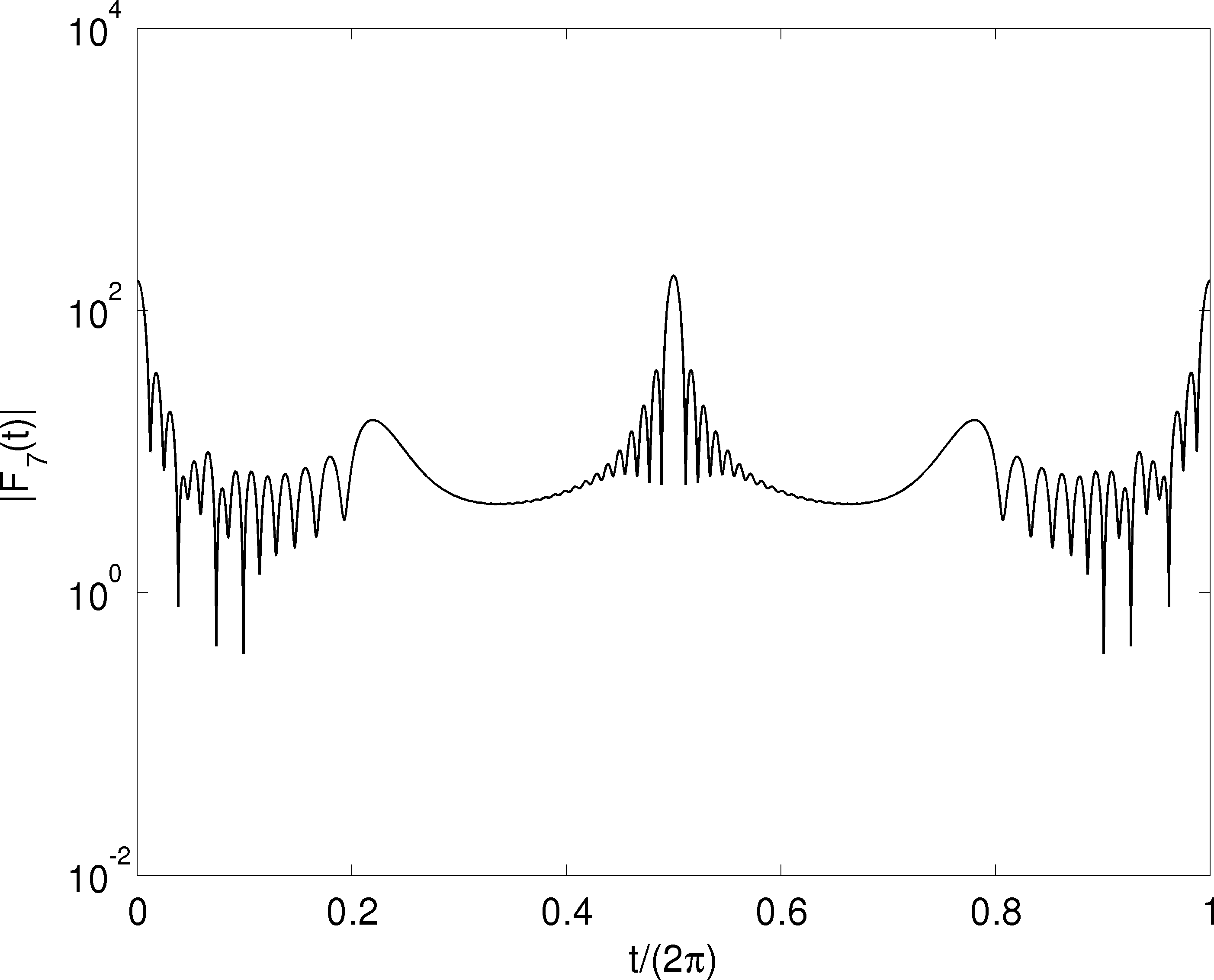}}
\hs{2}
\subfigure[$\alpha=5\pi/4$, $k=160$]{\includegraphics[width=5.9cm]{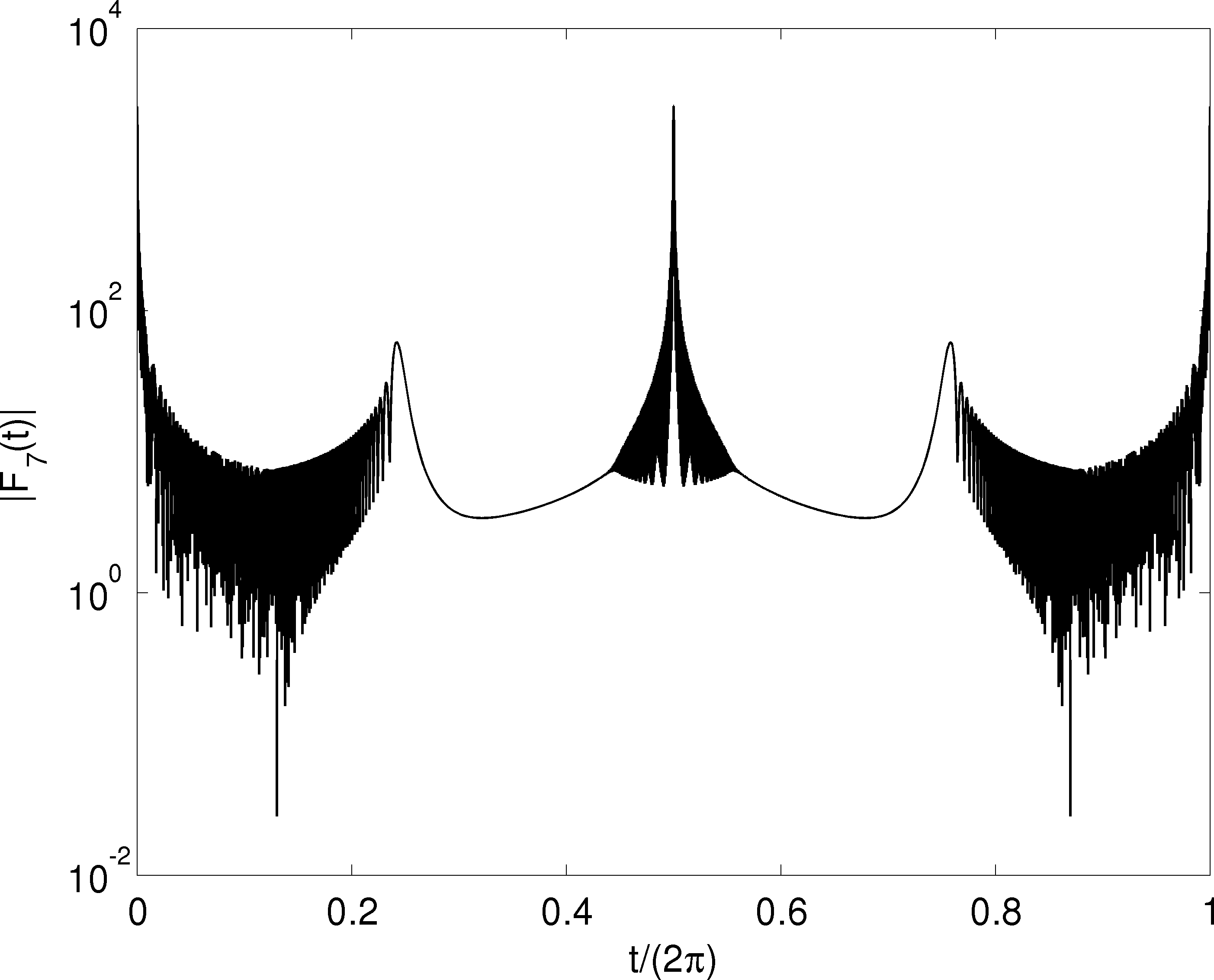}}
\vs{1}
\subfigure[$\alpha=5\pi/3$, $k=10$]{\includegraphics[width=5.9cm]{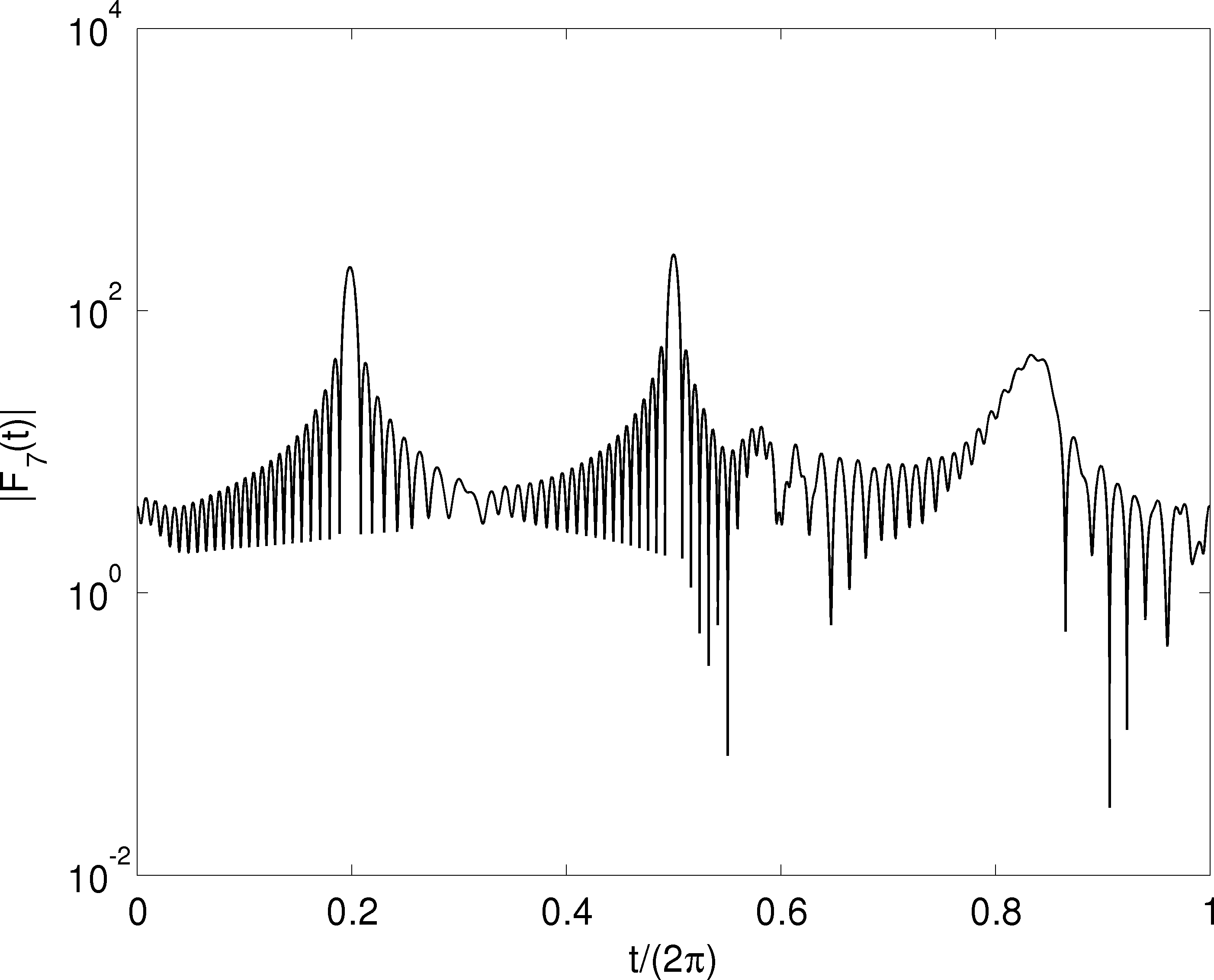}}
\hs{2}
\subfigure[$\alpha=5\pi/3$, $k=160$]{\includegraphics[width=5.9cm]{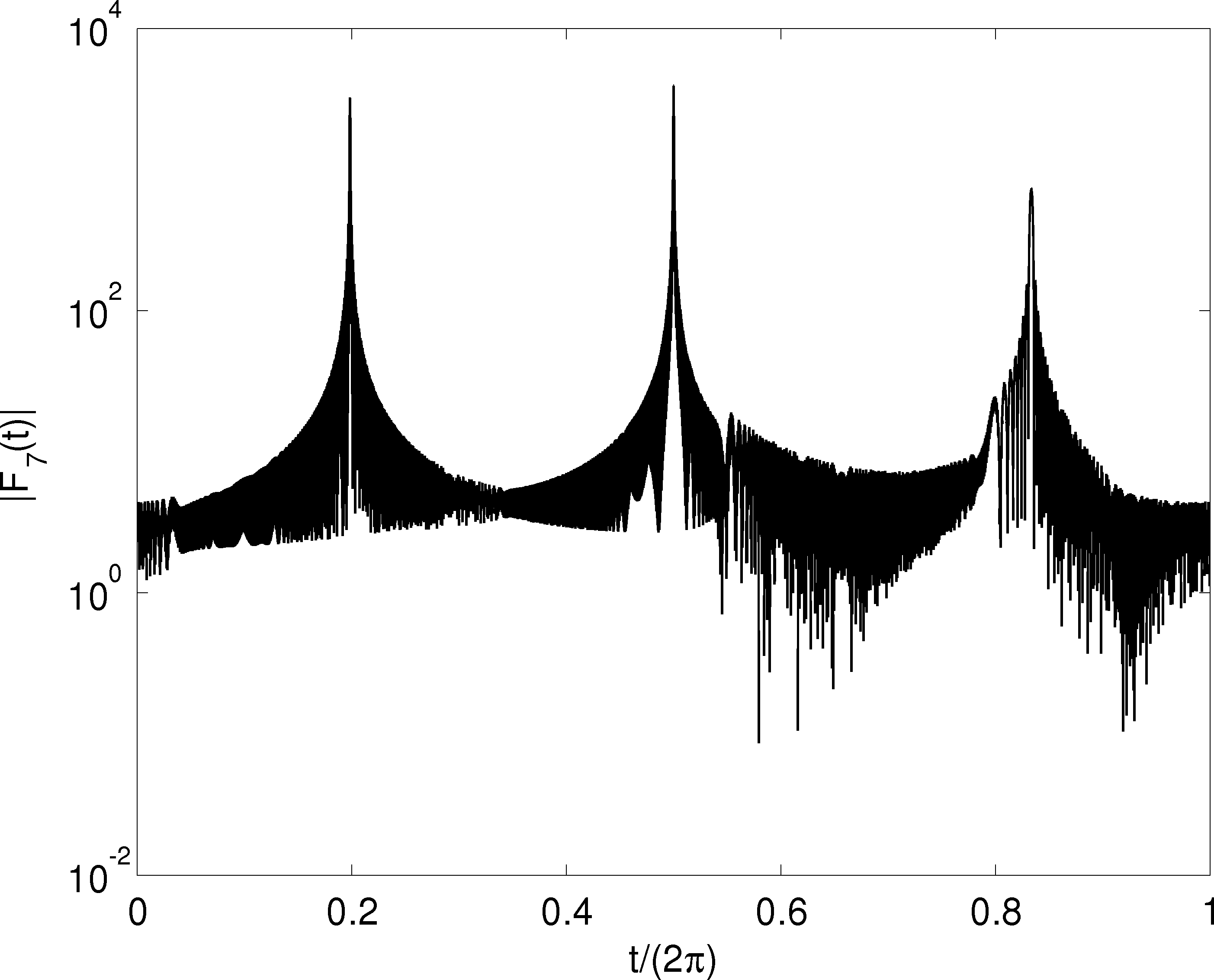}}
\end{center}
\caption{Far field patterns, $|F_7(t)|\approx|F(t)|$, $k=10$ and $k=160$.}
\label{fig:FFP1}
\end{figure}
\begin{figure}[htbp]
\begin{center}
\subfigure[$\alpha=5\pi/4$]{\includegraphics[width=5.9cm]{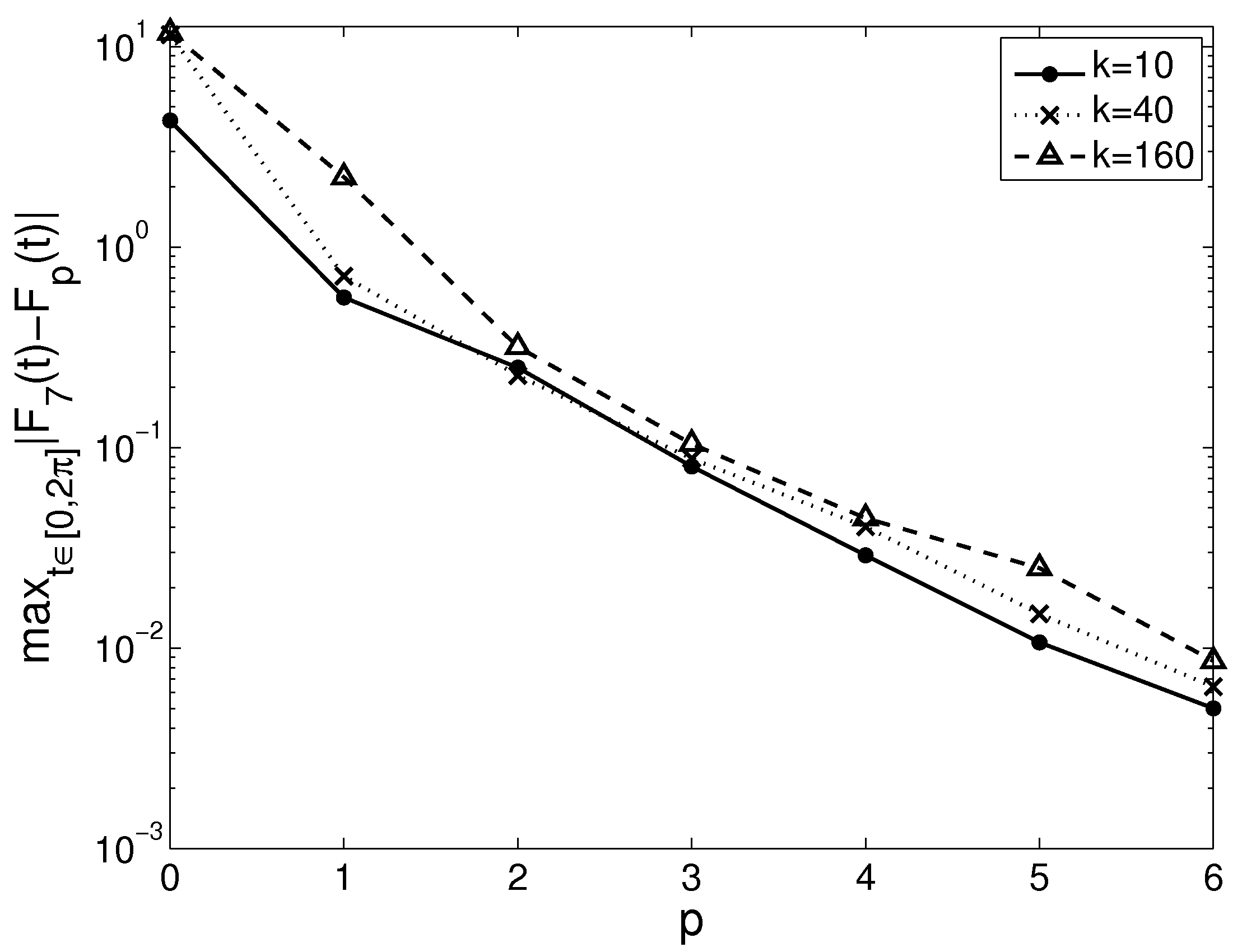}}
\hs{2}
\subfigure[$\alpha=5\pi/3$]{\includegraphics[width=5.9cm]{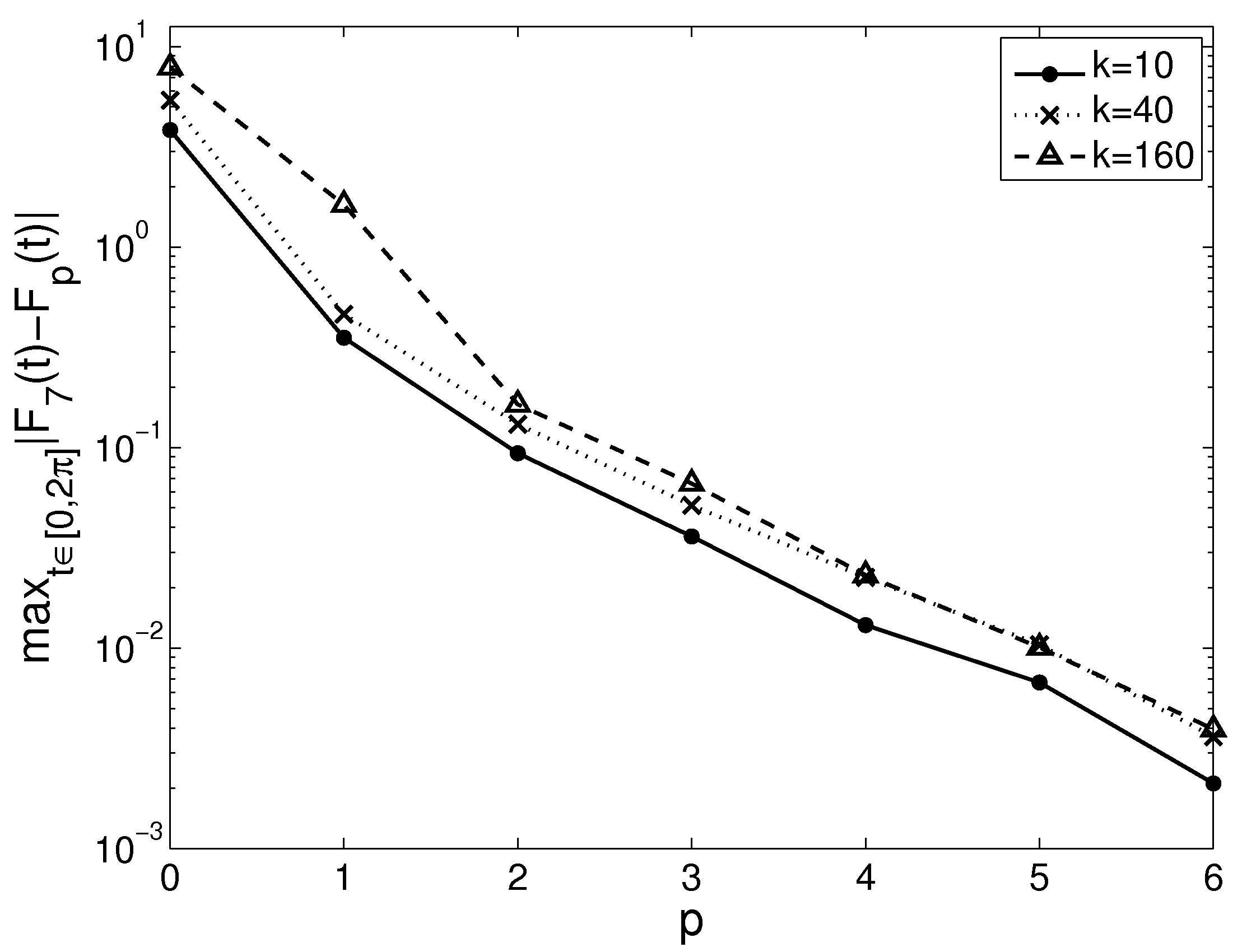}}
\end{center}
\caption{Absolute maximum errors $\|F_7-F_p\|_{L^{\infty}(0,2\pi)}$ in the far field pattern.}%
\label{fig:FFP3}
\end{figure}
To approximate the $L^{\infty}$ norm, we compute $F_7$ and $F_p$ at 30,000 evenly spaced points on the unit circle.
The exponential decay as $p$ increases predicted by Theorem~\ref{FarFieldThm} is clear. %
For fixed $p$, the error does not grow significantly as $k$ increases, indicating that the mild $k$-dependence of the bound~(\ref{FarFieldErrorEstStarlike}) may not be optimal.  The errors are comparable in magnitude for each incidence angle, suggesting that our algorithm copes equally well with cases of multiple reflection and partial illumination.

In summary, our numerical examples demonstrate that the predicted exponential convergence of our $hp$ scheme is achieved in practice.  Moreover, for a fixed number of degrees of freedom, the accuracy of our numerical solution appears to deteriorate only very slowly (or not at all) as the wavenumber $k$ increases.  The $p$- and $k$-dependence of our results appears to mimic closely that of the comparable results for the convex polygon in \cite{HeLaMe:11}.   The $k$-explicit error bounds in Corollary~\ref{StarlikeCor} predict at worst mild growth in errors as $k$ increases, which can be controlled by a logarithmic growth in the degrees of freedom $N$, as discussed in Remark \ref{rem:log2k}. The numerical results support the conjecture that this mild growth is pessimistic; the estimates in Corollary \ref{StarlikeCor} are not quite sharp in their $k$-dependence. We suspect this is due to lack of sharpness in $k$-dependence of the estimate~(\ref{MBoundProp2}) for $\uM$, of our best approximation estimate~(\ref{BestAppdudn}), and of the quasi-optimality estimate \eqref{quasi-opt}.

\section{Discussion - extension to more general nonconvex polygons}
\label{sec:generalisations}

In this section we discuss the possibility of extending our algorithm and analysis to more general nonconvex polygons not in the class $\cC$ of Definition \ref{classCDef}. We provide suggestions, informed by high frequency asymptotics, as to how the conditions of Definition \ref{classCDef} might be relaxed, and what effect this would have on our HNA approximation space and the accompanying analysis.

We first make the rather trivial remark that we expect the ``visibility'' condition (ii) of Definition \ref{classCDef} can be relaxed, without any change to our approximation space, to the following slightly weaker condition, illustrated in Figure \ref{fig:scatgen1}(a).

\medskip
\noindent
Condition (ii)$'$:
For each neighbouring pair $\{\Gamma_{\rm nc}, \Gamma_{\rm nc}'\}$ of nonconvex sides, let $\bP$ and $\bQ$ denote the endpoints of $\Gamma_{\rm nc}$, and let $\bQ$ and $\bR$ denote those of $\Gamma_{\rm nc}'$. Then $\overline\Omega\setminus \{\bP,\bR\}$ must lie entirely on one side of the line (shown as dashed in Figure \ref{fig:scatgen1}(a)) through $\bP$ and $\bR$.

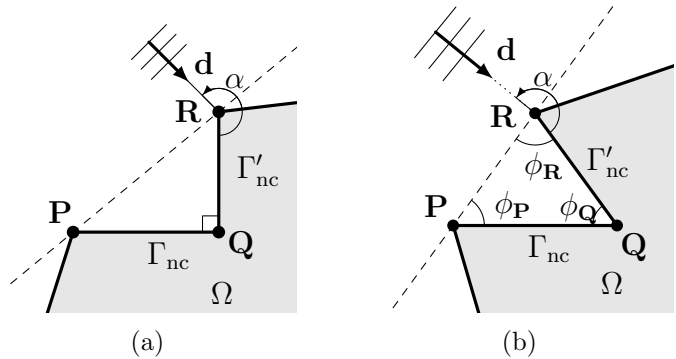
\begin{figure}[htbp]
  \begin{center}
     \subfigure[]{\label{fig:scatgen1a}
     \begin{tikzpicture}[line cap=round,line join=round,>=latex,x=0.5cm,y=0.5cm,scale=1.6]
\clip(-0.5,0.67) rectangle (4.3,5.55);
\draw(3,2.27) -- (2.73,2.27) -- (2.73,2) -- (3,2) -- cycle;
\draw [shift={(3,4)}] (0,0) -- (-90:0.39) arc (-90:135:0.39) -- cycle;
\fill[line width=0pt,color=cqcqcq,fill=cqcqcq] (4.86,4.22) -- (3,4) -- (3,2) -- (0.58,2) -- (0.14,0.62) -- (4.86,0.62) -- cycle;
\draw [line width=1.2pt] (3,4)-- (3,2);
\draw [line width=1.2pt] (0.58,2)-- (3,2);
\draw [line width=1.2pt] (3,4)-- (4.86,4.22);
\draw [line width=1.2pt] (0.58,2)-- (0.14,0.62);
\draw [dash pattern=on 3pt off 3pt,domain=-0.66:4.73] plot(\x,{(-3.68-2*\x)/-2.42});
\draw (1.63,2.01) node[anchor=north west] {$\Gamma_{\rm nc}$};
\draw (3.12,3.47) node[anchor=north west] {$\Gamma_{\rm nc}'$};
\draw (3,4)-- (1.84,5.16);
\draw (2.13,5.45)-- (1.55,4.87);
\draw (2.29,5.29)-- (1.71,4.71);
\draw (2.46,5.12)-- (1.88,4.54);
\draw [->,line width=1.2pt] (1.84,5.16) -- (2.51,4.49);
\draw [shift={(3,4)},->] (-90:0.39) arc (-90:135:0.39);
\draw (2.92,4.8) node[anchor=north west] {$\alpha$};
\draw (2.69,1.34) node[anchor=north west] {$\Omega$};
\draw (2.41,5.18) node[anchor=north west] {$\mathbf{d}$};
\fill [color=black] (3,4) circle (1.5pt);
\draw[color=black] (2.49,4.01) node {$\mathbf{R}$};
\fill [color=black] (3,2) circle (1.5pt);
\draw[color=black] (3.37,1.79) node {$\mathbf{Q}$};
\fill [color=black] (0.58,2) circle (1.5pt);
\draw[color=black] (0.37,2.34) node {$\mathbf{P}$};
\end{tikzpicture}
     }
\hs{5}
     \subfigure[]{\label{fig:scatgen1b}
\begin{tikzpicture}[line cap=round,line join=round,>=latex,x=0.6cm,y=0.6cm,scale=1.5]
\clip(-0.56,0.72) rectangle (3.83,5.3);
\draw [shift={(3,2)}] (0,0) -- (126.04:0.36) arc (126.04:180:0.36) -- cycle;
\draw [shift={(1.79,3.66)}] (0,0) -- (-53.96:0.36) arc (-53.96:140.91:0.36) -- cycle;
\fill[line width=0pt,color=cqcqcq,fill=cqcqcq] (1.07,0.32) -- (4.51,0.32) -- (4.51,4.59) -- (1.79,3.66) -- (3,2) -- (0.58,2) -- cycle;
\draw [line width=1.2pt] (1.79,3.66)-- (3,2);
\draw [line width=1.2pt] (0.58,2)-- (3,2);
\draw [line width=1.2pt] (1.79,3.66)-- (4.51,4.59);
\draw [line width=1.2pt] (0.58,2)-- (1.07,0.32);
\draw [dash pattern=on 3pt off 3pt,domain=-0.86:4.14] plot(\x,{(-1.46-1.66*\x)/-1.21});
\draw (1.52,2.01) node[anchor=north west] {$\Gamma_{\rm nc}$};
\draw (2.39,3.33) node[anchor=north west] {$\Gamma_{\rm nc}'$};
\draw [dotted] (1.79,3.66)-- (0.32,4.86);
\draw (0.62,5.23)-- (0.01,4.48);
\draw (0.83,5.07)-- (0.22,4.32);
\draw (1.04,4.89)-- (0.43,4.14);
\draw [->,line width=1.2pt] (0.32,4.86) -- (1.17,4.17);
\draw [shift={(1.79,3.66)},->] (-53.96:0.36) arc (-53.96:140.91:0.36);
\draw (1.6,4.43) node[anchor=north west] {$\alpha$};
\draw [shift={(1.79,3.66)}] plot[domain=4.08:5.34,variable=\t]({1*0.47*cos(\t r)+0*0.47*sin(\t r)},{0*0.47*cos(\t r)+1*0.47*sin(\t r)});
\draw (1.47,3.25) node[anchor=north west] {$\phi_\bR$};
\draw (1.03,2.67) node[anchor=north west] {$\phi_\bP$};
\draw [shift={(0.58,2)}] plot[domain=0:0.94,variable=\t]({1*0.46*cos(\t r)+0*0.46*sin(\t r)},{0*0.46*cos(\t r)+1*0.46*sin(\t r)});
\draw (2.6,1.47) node[anchor=north west] {$\Omega$};
\draw (1.99,2.64) node[anchor=north west] {$\phi_\bQ$};
\draw (1.06,4.94) node[anchor=north west] {$\mathbf{d}$};
\fill [color=black] (1.79,3.66) circle (1.5pt);
\draw[color=black] (1.32,3.62) node {$\mathbf{R}$};
\fill [color=black] (3,2) circle (1.5pt);
\draw[color=black] (3.25,1.68) node {$\mathbf{Q}$};
\fill [color=black] (0.58,2) circle (1.5pt);
\draw[color=black] (0.33,2.29) node {$\mathbf{P}$};
\end{tikzpicture}
}
  \end{center}
  \caption{Generalising the conditions of Definition \ref{classCDef}.
The whole of $\overline\Omega\setminus \{\bP,\bR\}$ must lie on one side of the (dashed) line through $\bP$ and $\bR$.}
  \label{fig:scatgen1}
\end{figure}
This weakened assumption is still sufficient to ensure that only three corners of $\Omega$ are visible at any point of $\Gamma$; it also ensures that any shadow boundaries associated with diffraction at $\bP$ and $\bR$ of waves scattered from other parts of $\Gamma$ do not intersect $\Gamma_{\rm nc}$ or $\Gamma_{\rm nc}'$, respectively.
We believe it should also be possible to extend our rigorous analysis to this case; in particular we expect that Theorem \ref{dudnThm}, for example, should still hold. However, a proof of this would require modification and generalisation of the results in Lemmas \ref{ImChiLem}--\ref{SkLem}, which we have yet to achieve. %

Next we consider relaxing the ``orthogonality'' condition (i) of Definition \ref{classCDef}, which stipulates that neighbouring nonconvex sides must meet at right-angles.
We expect this condition can be relaxed completely to allow the angle $\phi_\bQ$ between $\Gamma_{\rm nc}$ and $\Gamma_{\rm nc}'$ to be any angle between $0$ and $\pi$, with condition (ii) of Definition \ref{classCDef} replaced by condition (ii)$'$ above, with the more general geometry illustrated in Figure \ref{fig:scatgen1}(b). However, to return similar performance and accuracy for the same number of degrees of freedom we would, to cope with this extension, need to make significant changes to our HNA approximation space on $\Gamma_{\rm nc}$, as we now explain. In general, the complexity of the approximation space (in particular the number of terms required in the ansatz \rf{eqn:ansatz}) will need to increase as the angle $\phi_\bQ$ decreases, in order to capture the increasing number of multiple reflections that can occur between the two sides $\Gamma_{\rm nc}$ and $\Gamma_{\rm nc}'$. The form of the approximation space will also differ depending on whether or not $\pi/\phi_\bQ$ is an integer.

We first consider the case where $\phi_\bQ=\pi/m$ for some integer $m\geq 2$.
(In this case, we note that the method of images provides a simple closed form Green's function for the relevant canonical problem of scattering in a sector of angle $\phi_\bQ$.)
Informed by the case $m=2$ (cf.\ in particular the discussion in Remark \ref{rem:LO}), we would define our (known) ``leading order'' behaviour (i.e.\ the generalisation of the first term $\Psi$ in \rf{dudnGamma2Rep}) to be two times the normal derivative of a modified geometrical optics approximation to $\pdonetext{u}{\bn}$ on $\Gamma_{\rm nc}$, which would be a sum of $m$ terms corresponding to the incident wave and the $m-1$ higher order reflections of it in the sides $\Gamma_{\rm nc}$ and $\Gamma_{\rm nc}'$, with Fresnel integrals used to deal with shadow boundary effects.
According to the principles of the Geometrical Theory of Diffraction (see, e.g.,\ \cite{BoKi:94}), the remainder of the field on $\Gamma_{\rm nc}$ should then comprise diffracted waves emanating from $\bP$ and $\bR$, and the (multiple) reflections of these waves in the sides $\Gamma_{\rm nc}$ and $\Gamma_{\rm nc}'$ (we shall call such waves ``diffracted-reflected").
To determine the phases associated with each of the diffracted-reflected waves, we appeal to the method of images, thinking of each diffracted-reflected wave as emanating from a certain ``image corner'', obtained by an appropriate series of reflections of either $\bP$ or $\bR$ in the lines $\Gamma_{\rm nc}$ and $\Gamma_{\rm nc}'$.
Recalling from the case $m=2$ the interpretation in Remark \ref{rem:LO} of the second term in \rf{dudnGamma2Rep} as originating from the image corner $\bP'$ shown in Figure \ref{fig-1}(a), we can rewrite the second, third and fourth terms on the right hand side of \rf{dudnGamma2Rep} as
\begin{align}
\label{eqn:GeneralisedAnsatz0}
v_\bP \re^{\ri kr_\bP} + v_\bR \re^{\ri kr_\bR} + v_{\bP'} \re^{\ri kr_{\bP'}},
\end{align}
where, for an observation point $\bx\in\Gamma_{\rm nc}$, we define $r_\bP:=\|\bx-\bP\|$ etc. Recall that the amplitude $v_\bP$ is approximated on a mesh geometrically graded towards $\bP$, and $v_\bR$ and $v_{\bP'}$ are approximated by single polynomials supported on the whole side $\Gamma_{\rm nc}$.
The situation for the case $m=3$ is illustrated in Figure \ref{fig:scatgen2}(a).
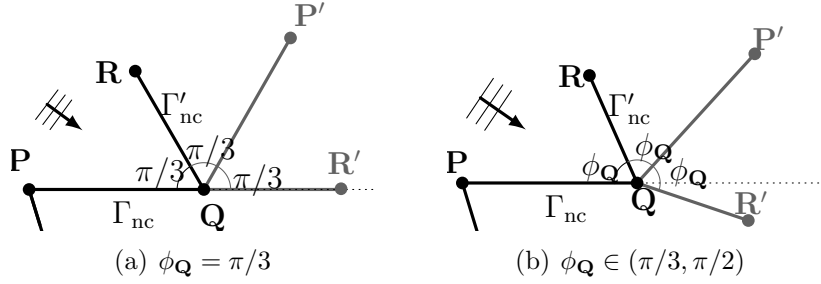
\begin{figure}[tbp]
  \begin{center}
  \hs{-5}
\definecolor{wqwqwq}{rgb}{0.38,0.38,0.38}
     \subfigure[$\phi_\bQ=\pi/3$]{\label{fig:scatgen2a}
\begin{tikzpicture}[line cap=round,line join=round,>=latex,x=0.6cm,y=0.6cm, scale=1.6]
\clip(0.34,1.43) rectangle (5.35,4.7);
\draw [line width=1.2pt] (2.05,3.64)-- (3,2);
\draw [line width=1.2pt,color=wqwqwq] (4.2,4.1)-- (3,2);
\draw [line width=1.2pt,color=wqwqwq] (4.9,2)-- (3,2);
\draw [line width=1.2pt] (0.58,2)-- (3,2);
\draw [line width=1.2pt] (0.58,2)-- (0.86,1.08);
\draw (1.58,1.98) node[anchor=north west] {$\Gamma_{\rm nc}$};
\draw (2.25,3.5) node[anchor=north west] {$\Gamma_{\rm nc}'$};
\draw (0.98,3.38)-- (0.67,2.97);
\draw (1.08,3.31)-- (0.77,2.9);
\draw (1.17,3.24)-- (0.87,2.83);
\draw [->,line width=1.2pt] (0.83,3.18) -- (1.32,2.82);
\draw [line width=1.2pt] (0.58,2)-- (0.73,1.51);
\draw [dotted,domain=0.58:6.39010660819369] plot(\x,{(--4.84-0*\x)/2.42});
\draw [shift={(3,2)}] plot[domain=2.1:3.14,variable=\t]({1*0.37*cos(\t r)+0*0.37*sin(\t r)},{0*0.37*cos(\t r)+1*0.37*sin(\t r)});
\draw [shift={(3,2)},color=wqwqwq]  plot[domain=2.1:3.14,variable=\t]({-0.5*0.37*cos(\t r)+-0.87*0.37*sin(\t r)},{-0.87*0.37*cos(\t r)+0.5*0.37*sin(\t r)});
\draw [shift={(3,2)},color=wqwqwq]  plot[domain=2.1:3.14,variable=\t]({-0.5*0.37*cos(\t r)+0.87*0.37*sin(\t r)},{-0.87*0.37*cos(\t r)+-0.5*0.37*sin(\t r)});
\draw (1.9,2.6) node[anchor=north west] {$\pi/3$};
\draw (2.6,2.9) node[anchor=north west] {$\pi/3$};
\draw (3.25,2.5) node[anchor=north west] {$\pi/3$};
\fill [color=black] (2.05,3.64) circle (1.5pt);
\draw[color=black] (1.68,3.63) node {$\mathbf{R}$};
\fill [color=black] (0.58,2) circle (1.5pt);
\draw[color=black] (0.45,2.41) node {$\mathbf{P}$};
\fill [color=black] (3,2) circle (1.5pt);
\draw[color=black] (3.12,1.6) node {$\mathbf{Q}$};
\fill [color=wqwqwq] (4.2,4.1) circle (1.5pt);
\draw[color=wqwqwq] (4.47,4.43) node {$\mathbf{P}'$};
\fill [color=wqwqwq] (4.9,2.01) circle (1.5pt);
\draw[color=wqwqwq] (4.95,2.38) node {$\mathbf{R}'$};
\end{tikzpicture}
}
  \hs{3}
     \subfigure[$\phi_\bQ\in(\pi/3,\pi/2)$]{\label{fig:scatgen2b}
     \definecolor{wqwqwq}{rgb}{0.38,0.38,0.38}
\begin{tikzpicture}[line cap=round,line join=round,>=latex,x=0.6cm,y=0.6cm, scale=1.6]
\clip(0.38,1.34) rectangle (5.57,4.74);
\draw [line width=1.2pt] (2.34,3.49)-- (3,2);
\draw [line width=1.2pt,color=wqwqwq] (4.63,3.79)-- (3,2);
\draw [line width=1.2pt,color=wqwqwq] (4.54,1.48)-- (3,2);
\draw [line width=1.2pt] (0.58,2)-- (3,2);
\draw [line width=1.2pt] (0.58,2)-- (0.86,1.08);
\draw (1.57,1.98) node[anchor=north west] {$\Gamma_{\rm nc}$};
\draw (2.45,3.4) node[anchor=north west] {$\Gamma_{\rm nc}'$};
\draw (1.01,3.45)-- (0.64,2.9);
\draw (1.14,3.36)-- (0.77,2.82);
\draw (1.26,3.28)-- (0.89,2.73);
\draw [->,line width=1.2pt] (0.83,3.18) -- (1.46,2.74);
\draw [dotted,domain=0.58:5.570598516359489] plot(\x,{(--4.84-0*\x)/2.42});
\draw (2.08,2.57) node[anchor=north west] {$\phi_\bQ$};
\draw [shift={(3,2)}] plot[domain=1.99:3.14,variable=\t]({1*0.32*cos(\t r)+0*0.32*sin(\t r)},{0*0.32*cos(\t r)+1*0.32*sin(\t r)});
\draw [shift={(3,2)},color=wqwqwq]  plot[domain=1.99:3.14,variable=\t]({-0.67*0.32*cos(\t r)+-0.74*0.32*sin(\t r)},{-0.74*0.32*cos(\t r)+0.67*0.32*sin(\t r)});
\draw [shift={(3,2)},color=wqwqwq]  plot[domain=1.99:3.14,variable=\t]({-0.67*0.32*cos(\t r)+0.74*0.32*sin(\t r)},{-0.74*0.32*cos(\t r)+-0.67*0.32*sin(\t r)});
\draw (2.82,2.85) node[anchor=north west] {$\phi_\bQ$};
\draw (3.3,2.5) node[anchor=north west] {$\phi_\bQ$};
\fill [color=black] (0.58,2) circle (1.5pt);
\draw[color=black] (0.5,2.29) node {$\mathbf{P}$};
\fill [color=black] (2.34,3.49) circle (1.5pt);
\draw[color=black] (2.09,3.48) node {$\mathbf{R}$};
\fill [color=black] (3,2) circle (1.5pt);
\draw[color=black] (3.09,1.73) node {$\mathbf{Q}$};
\fill [color=wqwqwq] (4.63,3.79) circle (1.5pt);
\draw[color=wqwqwq] (4.82,4.01) node {$\mathbf{P}'$};
\fill [color=wqwqwq] (4.54,1.48) circle (1.5pt);
\draw[color=wqwqwq] (4.58,1.73) node {$\mathbf{R}'$};
\end{tikzpicture}
}
  \end{center}
  \caption{Identifying phases of diffracted-reflected waves using the method of images.}
  \label{fig:scatgen2}
\end{figure}
Here there are two image corners to consider: $\bP'$, the reflection of $\bP$ in $\Gamma_{\rm nc}'$ (corresponding to diffracted waves emanating from $\bP$ and being reflected onto $\Gamma_{\rm nc}$ by $\Gamma_{\rm nc}'$), and $\bR'$, the reflection of $\bR$, first in $\Gamma_{\rm nc}$, then in $\Gamma_{\rm nc}'$ (corresponding to diffracted waves emanating from $\bR$ and being reflected onto $\Gamma_{\rm nc}$ via first $\Gamma_{\rm nc}$ then $\Gamma_{\rm nc}'$).
In the case $m=3$ our HNA ansatz for $\pdonetext{u}{\bn}$ on $\Gamma_{\rm nc}$ would then comprise the three-term leading order behaviour mentioned above, plus the sum
\begin{align}
\label{eqn:GeneralisedAnsatz}
v_\bP \re^{\ri kr_\bP} + v_\bR \re^{\ri kr_\bR} + v_{\bP'} \re^{\ri kr_{\bP'}} + v_{\bR'} \re^{\ri kr_{\bR'}},
\end{align}
where the amplitudes $v_\bP$, $v_\bR$, $v_{\bP'}$ and $v_{\bR'}$ are to be approximated numerically.
As in the case $m=2$, we expect $v_\bP$ to have a singularity at $\bP$, and therefore propose to approximate it on a mesh geometrically graded towards $\bP$ (as per the middle mesh in Figure \ref{fig:meshes}).
We expect $v_\bR$ to be slowly-varying on $\Gamma_{\rm nc}$, and propose to approximate it by a single polynomial supported on the whole of $\Gamma_{\rm nc}$; we also expect that the same approximation strategy should work for the amplitudes $v_{\bP'}$ and $v_{\bR'}$ associated with the diffracted-reflected waves, provided that the shadow boundaries generated by the reflection processes involved do not intersect $\Gamma_{\rm nc}$. A sufficient condition to ensure that such intersection does not occur is that $\max(\phi_\bR,\phi_\bP)<\pi/2$, where the angles $\phi_\bR$ and $\phi_\bP$ are defined as in Figure \ref{fig:scatgen1}(b). When this condition fails it would be necessary to modify the approximation strategy for $v_{\bP'}$ and $v_{\bR'}$ to deal with possible rapid variation across the shadow boundaries. One approach to this could be to premultiply $v_{\bP'}$ and $v_{\bR'}$ by appropriate special functions/canonical solutions such as Fresnel integrals or generalised Fresnel integrals (cf.\ \cite[\S5.10]{BoKi:94}); another could be to approximate $v_{\bP'}$ and $v_{\bR'}$ on meshes geometrically graded towards the relevant shadow boundaries.

For $\phi_\bQ=\pi/m$, $m>3$, the above remarks generalise in a straightfoward way: to capture the diffracted-reflected fields one must add to the leading order behaviour a generalisation of the sum \rf{eqn:GeneralisedAnsatz0} consisting of $m+1$ terms, with the final term in \rf{eqn:GeneralisedAnsatz0} replaced by a sum of $m-1$ terms associated with the first $m-1$ image corners encountered when moving clockwise around $\bQ$, starting from $\bR$, in angular increments of $\phi_\bQ$. Provided that $\max(\phi_\bR,\phi_\bP)<\pi/2$, each of the associated amplitudes would be approximated by a single polynomial supported on the whole side $\Gamma_{\rm nc}$. (By symmetry, the image corners encountered when moving anti-clockwise around $\bQ$ need not be considered, since these produce waves which have the same phases on $\Gamma_{\rm nc}$ as the clockwise image corners.)

When $\phi_\bQ\in (\pi/m,\pi/(m-1))$ for some integer $m\geq 2$, the situation is a little more complicated. (In this case, we note that the method of images no longer provides an exact Green's function for scattering in a sector of angle $\phi_\bQ$; the Green's function now has a component corresponding to diffraction from the reentrant corner.)
Provided that $\max(\phi_\bR,\phi_\bP)<\pi/2$, we would, as usual, base our (known) leading order behaviour on a modified geometrical optics approximation, with shadow boundary effects dealt with using Fresnel integrals. This would again involve a sum of $m$ terms corresponding to the incident wave and its $m-1$ higher order reflections; but geometrical considerations imply that the highest order reflected wave in the geometrical optics approximation is non-zero on $\Gamma_{nc}$ over a reduced range of incidence directions compared to the other reflected waves.

To illustrate this, it is simplest to consider the case $m=2$, so that $\phi_\bQ\in(\pi/2,\pi)$. In this case we would take our leading order behaviour to be two times the normal derivative of the sum
\begin{align}
\label{eqn:GeneralisedLOB}
\begin{cases}
E(r,\theta-\alpha), & \alpha\in[\pi/2,\pi+\phi_\bQ],\\
0, & \text{otherwise}
\end{cases}
+
\begin{cases}
E(r,\theta+\alpha), & \alpha\in[\pi/2,2\pi-\phi_\bQ],\\
0, & \text{otherwise}
\end{cases}
\end{align}
where $E(r,\psi)$ is defined as in Lemma \ref{def:udDef}. Note that the term corresponding to the incident wave ($E(r,\theta-\alpha)$) is non-zero for $\alpha$ up to $\pi+\phi_\bQ$, whereas the term corresponding to the reflected wave ($E(r,\theta+\alpha)$) is non-zero only for $\alpha$ up to $2\pi-\phi_\bQ$, because the reflected rays do not strike $\Gamma_{\rm nc}$ for $\alpha \in [2\pi-\phi_\bQ,\pi+\phi_\bQ)$.

To determine the phases present in the remainder of the field, and obtain an ansatz similar to \rf{eqn:GeneralisedAnsatz0} or \rf{eqn:GeneralisedAnsatz}, one can again appeal to the method of images, as illustrated for the case $m=3$ in Figure \ref{fig:scatgen2}(b). But we need to make two changes compared to the case $\phi_\bQ=\pi/m$. First, we need only consider the first $m-2$ image corners encountered when moving clockwise around $\bQ$, starting from $\bR$, in angular increments of $\phi_\bQ$, because the $(m-1)$th image corner is no longer ``visible'' on $\Gamma_{\rm nc}$.
(So in the case $m=2$ we should remove the term $v_{\bP'} \re^{\ri kr_{\bP'}}$ from the ansatz \rf{eqn:GeneralisedAnsatz0}; in the case $m=3$ we should remove the term $v_{\bR'} \re^{\ri kr_{\bR'}}$ from \rf{eqn:GeneralisedAnsatz}). Second, we need to add a term $v_{\bQ} \re^{\ri kr_{\bQ}}$, corresponding to diffraction from the reentrant corner $\bQ$. The amplitude $v_{\bQ}$ will have a derivative singularity at $\bQ$ (in contrast to the case $\phi_\bQ=\pi/m$ when the solution is smooth at $\bQ$), and we therefore propose to approximate it on a geometric mesh graded towards $\bQ$.

To summarize, we have sketched how to modify our HNA approximation space for the numerical approximation of the solution of the Dirichlet scattering problem for polygons in the following class (which contains our original class $\cC$):
\begin{defn}[The class $\cC'$]
A polygon $\Omega\subset \R^2$ is a member of the class $\cC'$ if, relative to each corner $\bQ$ at which the exterior angle $\phi_\bQ$ is smaller than $\pi$, the following two conditions hold (where $\bP$, $\bQ$, $\bR$, $\phi_\bR$ and $\phi_\bP$ are as in Figure \ref{fig:scatgen1}(b)):
\begin{enumerate}[(i)]
\item The whole of $\overline\Omega \setminus\{\bP,\bR\}$ lies on one side of the line through $\bP$ and $\bR$;
\item $\max(\phi_\bR,\phi_\bP)<\pi/2$.
\end{enumerate}
\end{defn}
We believe that with the modifications described above, one should observe the same qualitative performance of our BEM to that for the class $\cC$ (i.e.\ exponential decay in error with increasing polynomial degree and only logarithmic growth in number of degrees of freedom to maintain accuracy as $k$ increases).
We leave experimental verification of this for future work.
At present our rigorous best approximation analysis holds only for the case $\phi_\bQ= \pi/2$. But it seems plausible that, with significant further work, our analysis could be generalised, at least to the case $\phi_\bQ=\pi/m$, where $m\geq 3$ an integer, because of the existence of a simple closed form Green's function for scattering in a sector of angle $\pi/m$ (this was a key ingredient in our analysis for the case $\phi_\bQ= \pi/2$). However, we anticipate that extending the analysis to general $\phi_\bQ$ would be considerably more challenging.

Further generalisation to polygons outside the class $\cC'$ would require more significant modifications to our HNA approximation space.
In particular, when more than three corners of the polygon are visible from one side of the polygon, the multiple scattering effects are in general considerably more complicated.
However, as remarked in \S\ref{Introduction}, algorithms developed for determining the high frequency behaviour in the case of scattering by multiple smooth convex scatterers (e.g.\ \cite{BrGeRe:05,Ec:05,EcRe:09,AnBoEcRe:10}) may be helpful as a source of ideas for how to deal with the interactions between distant parts of the scatterer which are visible to each other.  We note also the recent work \cite{GrHeLa:13} on the design of HNA approximation spaces for transmission problems, which exhibit similar multiple scattering phenomena to those encountered here.  However, we leave further discussion to future work.

\bibliographystyle{siam}      %
\bibliography{nonconvex_biblio}   %
\end{document}